\documentclass[12pt,onecolumn]{IEEEtran}
\ifCLASSINFOpdf
\else
\fi
\usepackage{float}
\usepackage{textcomp}
\usepackage{stfloats}
\usepackage{graphicx}
\usepackage{color}
\usepackage{amsthm,amsfonts}
\usepackage{amssymb}

\usepackage{cite}

\usepackage{amsmath,mathrsfs}
\makeatletter
\let\over=\@@over \let\overwithdelims=\@@overwithdelims
\let\atop=\@@atop \let\atopwithdelims=\@@atopwithdelims
\let\above=\@@above \let\abovewithdelims=\@@abovewithdelims
\makeatother
\usepackage{rotating}
\usepackage{multirow}

\usepackage{subfigure}
\usepackage{psfrag}

\usepackage{prettyref}

\usepackage{tikz}
\usetikzlibrary{arrows}
\tikzstyle{int}=[draw, fill=blue!20, minimum size=2em]
\tikzstyle{dot}=[circle, draw, fill=blue!20, minimum size=2em]
\tikzstyle{init} = [pin edge={to-,thin,black}]

\usepackage[
CJKbookmarks=true,
bookmarksnumbered=true,
bookmarksopen=true,
colorlinks=true,
citecolor=red,
linkcolor=blue,
anchorcolor=red,
urlcolor=blue
]{hyperref}

\usepackage[all]{xy}

\usepackage{paralist}

\newboolean{aos}
\setboolean{aos}{FALSE}

\usepackage{algorithm}
\usepackage{algpseudocode}

\usepackage{xspace,exscale,relsize}
\usepackage{fancybox,shadow}
\usepackage{url}
\usepackage{authblk}

\usepackage{extarrows}

\usepackage{ifpdf}

\usepackage{psfrag}

\usepackage{prettyref}
\usepackage{enumerate}
\usepackage{enumitem}

\usepackage[all]{xy}

\usepackage{mathtools}

\usepackage{ifthen}
\setboolean{aos}{FALSE}

\ifx\eqref\undefined
\newcommand{\eqref}[1]{~(\ref{#1})}
\fi
\ifx\mod\undefined
\def\mod{\mathop{\rm mod}}
\fi

\usepackage{bm}

\def\argmin{\mathop{\rm argmin}}

\def\exp{\mathop{\rm exp}}

\def\EE{\Expect}
\DeclareMathOperator\sign{\rm sign}

\def\PP{\mathbb{P}}

\def\eqdef{\triangleq}

\def\simiid{\stackrel{iid}{\sim}}

\newcommand{\abs}[1]{\left| #1 \right|}

\newcommand{\stepa}[1]{\overset{\rm (a)}{#1}}
\newcommand{\stepb}[1]{\overset{\rm (b)}{#1}}
\newcommand{\stepc}[1]{\overset{\rm (c)}{#1}}
\newcommand{\stepd}[1]{\overset{\rm (d)}{#1}}
\newcommand{\stepe}[1]{\overset{\rm (e)}{#1}}

\newcommand{\floor}[1]{{\left\lfloor {#1} \right \rfloor}}
\newcommand{\ceil}[1]{{\left\lceil {#1} \right \rceil}}

\newcommand{\reals}{\mathbb{R}}
\newcommand{\naturals}{\mathbb{N}}

\newcommand{\Expect}{\mathbb{E}}

\newcommand{\iid}{i.i.d.\xspace}

\newcommand{\pth}[1]{\left( #1 \right)}
\newcommand{\qth}[1]{\left[ #1 \right]}
\newcommand{\sth}[1]{\left\{ #1 \right\}}

\newcommand{\Binom}{\text{Binom}}

\newcommand{\indc}[1]{{\mathbf{1}_{\left\{{#1}\right\}}}}

\definecolor{myblue}{rgb}{.8, .8, 1}
\definecolor{mathblue}{rgb}{0.2472, 0.24, 0.6} 
\definecolor{mathred}{rgb}{0.6, 0.24, 0.442893}
\definecolor{mathyellow}{rgb}{0.6, 0.547014, 0.24}

\newcommand{\sfM}{{\mathsf{M}}}

\newcommand{\sfT}{{\mathsf{T}}}

\newcommand{\calE}{{\mathcal{E}}}

\newcommand{\calN}{{\mathcal{N}}}

\newrefformat{eq}{(\ref{#1})}
\newrefformat{thm}{Theorem~\ref{#1}}
\newrefformat{th}{Theorem~\ref{#1}}
\newrefformat{chap}{Chapter~\ref{#1}}
\newrefformat{sec}{Section~\ref{#1}}
\newrefformat{algo}{Algorithm~\ref{#1}}
\newrefformat{fig}{Figure~\ref{#1}}
\newrefformat{tab}{Table~\ref{#1}}
\newrefformat{rmk}{Remark~\ref{#1}}
\newrefformat{clm}{Claim~\ref{#1}}
\newrefformat{def}{Definition~\ref{#1}}
\newrefformat{cor}{Corollary~\ref{#1}}
\newrefformat{lmm}{Lemma~\ref{#1}}
\newrefformat{propo}{Proposition~\ref{#1}}
\newrefformat{proper}{Property~\ref{#1}}
\newrefformat{app}{Appendix~\ref{#1}}
\newrefformat{apx}{Appendix~\ref{#1}}
\newrefformat{ex}{Example~\ref{#1}}
\newrefformat{exer}{Exercise~\ref{#1}}
\newrefformat{soln}{Solution~\ref{#1}}
\newrefformat{pt}{Assumption~\ref{#1}}

\def\unifto{\mathop{{\mskip 3mu plus 2mu minus 1mu%
			\setbox0=\hbox{$\mathchar"3221$}%
			\raise.6ex\copy0\kern-\wd0%
			\lower0.5ex\hbox{$\mathchar"3221$}}\mskip 3mu plus 2mu minus 1mu}}

\ifx\lesssim\undefined
\def\simleq{{{\mskip 3mu plus 2mu minus 1mu%
			\setbox0=\hbox{$\mathchar"013C$}%
			\raise.2ex\copy0\kern-\wd0%
			\lower0.9ex\hbox{$\mathchar"0218$}}\mskip 3mu plus 2mu minus 1mu}}
\else
\def\simleq{\lesssim}
\fi

\ifx\gtrsim\undefined
\def\simgeq{{{\mskip 3mu plus 2mu minus 1mu%
			\setbox0=\hbox{$\mathchar"013E$}%
			\raise.2ex\copy0\kern-\wd0%
			\lower0.9ex\hbox{$\mathchar"0218$}}\mskip 3mu plus 2mu minus 1mu}}
\else
\def\simgeq{\gtrsim}
\fi

\newtheorem{theorem}{Theorem}
\newtheorem{lemma}[theorem]{Lemma}

\theoremstyle{definition}

\newtheorem{remark}{Remark}

\newif\ifmapx
{\catcode`/=0 \catcode`\\=12/gdef/mkillslash\#1{#1}}
\edef\jobnametmp{\expandafter\string\csname ic_apx\endcsname}
\edef\jobnameapx{\expandafter\mkillslash\jobnametmp}
\edef\jobnameexpand{\jobname}
\ifx\jobnameexpand\jobnameapx
\mapxtrue
\else
\mapxfalse
\fi

\renewcommand{\hat}{\widehat}
\renewcommand{\tilde}{\widetilde}

\renewcommand{\hat}{\widehat}
\renewcommand{\tilde}{\widetilde}

\newcommand{\subG}{\mathsf{SubG}}

\newcommand{\lcomed}{$k$-medians-hybrid~}
\newcommand{\kmed}{$k$-medians-${\ell}_1$~}
\newcommand{\kmeans}{$k$-means~}

\newcommand{\snr}{\mathsf{SNR}}
\newcommand{\out}{{\mathsf{out}}}

\newcommand{\med}{\mathsf{median}}

\newcommand\identity{1\kern-0.25em\text{l}}
\newcommand{\eord}{\calE^{\sf{ord}}}
\newcommand{\econ}{\calE^{\sf{con}}}
\newcommand{\eEigen}{\calE^{\sf{eigen}}}

\begin{document}
		%
		%
		%
		%
		
		\title{Adversarially robust clustering with optimality guarantees}
		
		\author{Soham Jana, Kun Yang, and Sanjeev Kulkarni\thanks{S.~J. is with the Department of Applied and Computational Mathematics and Statistics, University of Notre Dame, Notre Dame, IN, USA, email: \url{soham.jana@nd.edu}.  K.~Y. completed the work when he was with the Department of Operations Research and Financial Engineering, Princeton University, Princeton, NJ, USA, email: \url{kun88.yang@gmail.com}. S.~K. is with the Department of Electrical and Computer Engineering and Department of Operations Research and Financial Engineering, Princeton University, Princeton, NJ, USA email: \url{kulkarni@princeton.edu}.}
		}
	\maketitle
	
	\begin{abstract}
		We consider the problem of clustering data points coming from sub-Gaussian mixtures. Existing methods that provably achieve the optimal mislabeling error, such as the Lloyd algorithm, are usually vulnerable to outliers. In contrast, clustering methods seemingly robust to adversarial perturbations are not known to satisfy the optimal statistical guarantees. We propose a simple robust algorithm based on the coordinatewise median that obtains the optimal mislabeling rate even when we allow adversarial outliers to be present. Our algorithm achieves the optimal error rate in constant iterations when a weak initialization condition is satisfied. In the absence of outliers, in fixed dimensions, our theoretical guarantees are similar to that of the Lloyd algorithm. Extensive experiments on various simulated and public datasets are conducted to support the theoretical guarantees of our method.
	\end{abstract}
	

	%
	\IEEEpeerreviewmaketitle

	\section{Introduction}
	
	\subsection{Problem}
	
	Clustering a set of observations is a classical learning problem in statistics and machine learning \cite{jian2009data}. When the observed data is generated via a mixture of distributions, a large body of research has studied algorithms to help classify points belonging to the same component. This clustering task facilitates the learning of parameters for different mixture components with high accuracy. Applications exist in diverse areas, e.g., organizing wireless sensor networks \cite{ABBASI20072826,sasikumar2012k}, grouping different biological species \cite{maravelias1999habitat,pigolotti2007species},
	medical imaging \cite{ng2006medical,ajala2012fuzzy}, and social
	network analysis \cite{mishra2007clustering,ding2010clustering}. In practice, the data is likely to contain noise and outliers, and clustering techniques must be robust to optimize various learning tasks \cite{dave1997robust,hardin2004outlier,garcia2010review}. Several new techniques have been developed to perform robust clustering to varying degrees \cite{dave1991characterization,dave1993robust,jolion1991robust,krishnapuram1993possibilistic}.
	
	In this paper, we focus on the center-based robust clustering method. Center-based clustering has received significant attention \cite{awasthi2014center,MKCWM2015,anegg2020technique,zhang2022practical}, specifically when the data are distributed according to sub-Gaussian noise around the location parameters \cite{lu2016statistical,srivastava2023robust,zhang2023upper,makarychev2019performance,lyu2022optimal,bakshi2020outlier,abbe2022}. In this setup, one assumes that the underlying cluster components can be identified using centroids of the components distributions. When the number of cluster components is assumed to be known (often denoted by $k$), and initial centroid estimates are available, simple iterative techniques are often used for clustering. Suppose we observe data points $Y_1,\dots,Y_n$ coming from $k$ clusters with centroids $\theta_1,\dots,\theta_k$ respectively. Let $\cup_{h\in \sth{1,\dots k}}T_h^*$ denote the partition of the set of data indices $\{1,\dots,n\}$ that gives us the true (unknown) cluster memberships (i.e., the subset of data coming from the $h$-th cluster is $\sth{Y_i:i\in T_h^*}$). Iterative algorithms largely follow two main steps at each iteration $s\geq 1$: \begin{itemize}
		\item {\it Labeling step:} Given an estimate of the centroids $\hat \theta^{(s)}_h$, construct cluster estimates $T_h^{(s)},h\in\{1,\dots,k\}$ using a suitable cost function;
		
		\item {\it Estimation step:} For each of the clusters, compute the new centroid estimates $\hat \theta^{(s+1)}_h$ using the points from the estimated cluster $\sth{Y_i:i\in T_h^{(s)}}$.
	\end{itemize}
	This process is then repeated until the clusters do not change significantly over subsequent iterations or until preset thresholds, such as the number of iterations, are reached. For performance evaluation, suppose the underlying labels of the points (i.e., the true cluster $T_h^*$ from which the data point originated) and the centroids are known beforehand. In this case, one can estimate the performance of the clustering algorithm using the mislabeling error (the proportion of points that originated from some cluster $T_h^*$ but were not clustered as part of $T_h^{(s)}$ at the end of the iterations) and the centroid estimation error.
	
	The most popular example of iterative clustering is the Lloyd algorithm for $k$-means \cite{awasthi2012center}. This special case of the iterative method is obtained when the labeling step is performed using the squared Euclidean distance ($\ell_2^2$) as the cost function, and the estimation step is performed using the mean of data points in the cluster estimates. The Lloyd algorithm is known to be a {\it greedy method}, as it reduces the within-cluster sum of squared error $\sum_{h\in \sth{1,\dots,k}}\sum_{i\in \hat T_h^{(s)}}\ell_2^2(Y_i,\hat \theta_h^{(s)})$ at each successive iteration $s\geq 1$. Recently, \cite{lu2016statistical} showed that the mislabeling error produced by the Lloyd algorithm is optimal when the data have sub-Gaussian distributions. They also show that when the Lloyd algorithm converges, the final estimates of the centroids are consistent as well. Unfortunately, in the presence of adversarial outliers, the performance of the Lloyd algorithm can be severely compromised. This is a common problem of mean-based clustering algorithms \cite{charikar2001algorithms,olukanmi2017k,gupta2017local} as the naive mean used to produce the centroid estimates has almost no robustness guarantees. This is easily visible in the simple two-cluster setup. For example, suppose an adversary is allowed to add even one outlying observation. In that case, one can place the outlier so far away that the corresponding centroid estimate for whichever cluster the outlier is included in becomes inconsistent. 
	
	A reasonable approach to obtain a robust clustering method that has been considered in the literature is to use a robust centroid estimator in the estimation step. For example, the $k$-medians method considered in \cite{makarychev2019performance,balcan2017differentially} aims to find centers $\hat \theta_1,\dots,\hat\theta_k$ to minimize the sum of Euclidean errors $\sum_{i=1}^n \min_{h\in [k]}\ell_2(Y_i,\hat \theta_h)$ (as opposed to squared errors).  The centroid estimation is equivalent to computing the geometric median (also called the Euclidean median) \cite{bajaj1986proving} for the corresponding clusters \cite{cohen2016geometric}. The geometric median is known to have strong robustness properties such as a $50\%$ breakdown point \cite{lopuhaa1989relation}, meaning that one needs to change at least half the points of any cluster to change the centroid estimates by an arbitrarily large amount. However, computation of the geometric median is difficult, and different approximations have been studied. See \cite{cohen2016geometric,weiszfeld1937point} for related references.
	
	A well-known robust estimator, which provides robust (e.g., $50\%$ breakdown point) location estimates in multiple dimensions and is simple to compute, is the coordinatewise median \cite{bickel1964some}. Given data points from the estimated clusters, the coordinatewise median computes the one-dimensional median on each coordinate. Recently \cite{yin2018byzantine} used the coordinatewise median to achieve a near-optimal rate of parameter estimation in robust distributed learning problems. The use of the coordinatewise median for clustering is also not new. In the iterative clustering scheme mentioned above, the coordinatewise median for centroid estimation can be viewed as using the $\ell_1$ metric (also known as the Manhattan distance or the city-block distance \cite{DESOUZA2004353}), as this estimate $\hat \theta_h^{(s)}$ is given by $\argmin_\theta \sum_{i\in \hat T_h^{(s)}}\ell_1(Y_i,\theta)$. 
	The corresponding greedy clustering algorithm at each iteration reduces the within-cluster sum of error $\sum_{h\in \sth{1,\dots,k}}\sum_{i\in \hat T_h^{(s)}}\ell_1(Y_i,\hat \theta_h^{(s)})$.  This algorithm, which we will refer to as the \kmed algorithm, produces the coordinatewise median of $\{Y_i:i\in T_h^{(s)}\}$ as the centroid estimators \cite{bradley1996clustering} in the estimation step, and assigns each data point to the nearest centroid using the $\ell_1$ distance in the labeling step. Although this algorithm has been used in the literature and inherits some robustness properties from the coordinatewise median, the statistical properties of this greedy clustering algorithm based on the $\ell_1$ metric have been less analyzed. We ask the question: if we have a sub-Gaussian data distribution, can \kmed clustering achieve the optimal mislabeling error rate similar to the Lloyd algorithm? The answer turns out to be no, as we will explain in \prettyref{sec:l1-subopt}. 
	
	To address these various issues, the focus of the current paper is to devise and analyze a fast algorithm to perform robust clustering in the presence of adversarial outliers.
	As a first step in this direction, we propose a hybrid version of $k$-medians that uses the coordinatewise median (i.e., $\ell_1$) for the estimation step but uses the Euclidean distance (i.e., $\ell_2$) in the labeling step.  
	In view of this, we refer to our method as the \lcomed algorithm.
	We analyze the statistical robustness properties of the \lcomed algorithm in terms of the mislabeling and centroid estimation errors
	from a non-asymptotic perspective when the underlying data distribution is assumed to be sub-Gaussian and when adversarial outliers are present.
	
	This paper makes several contributions toward the sub-Gaussian clustering problem in the presence of adversarial outliers. \begin{itemize}
		\item First, we explore the mislabeling guarantees of the traditional $\ell_1$ metric used for robust clustering. We show that when the data is distributed according to sub-Gaussian errors around the centroids, even in the absence of outliers, the \kmed algorithm can not simultaneously guarantee consistent centroid estimates and optimal mislabeling. More specifically, suppose that the data dimension and number of clusters are fixed. Then there is a subset of the parameter space, with the non-vanishing volume on which, whenever the centroids are picked, if the centroid estimates are within a ball of fixed radius around the actual centroids, the expected mislabeling error obtained by \kmed algorithm is strictly suboptimal. 
		\item Secondly, we show that with reasonably weak initialization conditions, the mislabeling error rate produced by the \lcomed algorithm, even in the presence of adversarial outliers, matches the optimal mislabeling rate for sub-Gaussian clustering. With a high probability, the optimal rate is produced in just two iterations, and the worst-case guarantee does not degrade in subsequent iterations. Analysis of such an iterative technique, in particular when outliers are present, is challenging as the consecutive steps are dependent. Due to the high breakdown point of the coordinatewise median, we can allow the total adversarial outliers to be close (but strictly smaller than) to the size of the smallest cluster. In a sense, this is the maximum data perturbation any clustering algorithm can tolerate; if any adversary can add more outliers than the size of the smallest cluster, it can significantly disturb the centroid estimates. 
		\item Additionally, we show that with a high probability, the centroid estimates produced by our algorithm achieve the optimal rate of sub-Gaussian mean estimation, up to a logarithmic factor, when the error distributions are symmetric around the locations. To our knowledge, this is the first adversarially robust clustering algorithm with provably optimal guarantees for mislabeling and consistency guarantees for centroid estimation.
		
	\end{itemize}
	{ The rest of the paper is organized as follows. In \prettyref{sec:l1-subopt} we describe the suboptimal mislabeling error of the \kmed algorithm. Then we present our \lcomed algorithm and implementation in \prettyref{sec:algorithm}. The results related to the statistical guarantees of our algorithm in terms of mislabeling errors and centroid estimation errors have been provided in \prettyref{sec:main-results}. We present our simulation and real data studies in \prettyref{sec:experiments}. Finally, we end the main paper with a sketch of the proofs of our main results in \prettyref{sec:sketch}. The proof of the suboptimality of the \kmed algorithm is provided in \prettyref{app:l1-mislabel}. The proof of our main result, i.e., the mislabeling guarantee of our method \lcomed, is given in \prettyref{app:main-proof-mislabeling}, and the proof of centroid estimation guarantees of our algorithm is presented in \prettyref{app:proof-centroid}.}
	
	\begin{remark}
		Recall, as we mentioned before, that a complete iterative clustering algorithm involves two main components: (a) obtaining an initial cluster assignment or centroid estimates and (b) the iterative components, which gradually improve upon the initial clustering. Our work tries to answer the second half of the problem, i.e., optimizing the clustering outputs when a naive initial clustering is available. Note that without a good initialization, standard iterative methods, such as $k$-means, can get stuck at some suboptimal local solution \cite{lu2016statistical}. Clustering initialization itself is a challenging task in the literature. Classical initialization schemes, e.g., spectral methods \cite{vempala2004spectral}, $(1+\epsilon)$-approximate $k$-means \cite{kumar2004simple}, are not well-equipped to deal with outliers. Unfortunately, constructing robust initialization schemes is beyond the scope of our current work as well. Notably, the recent work \cite{srivastava2023robust} aims to provide a robust initialization method based on semidefinite programming when the data is generated via a mixture of sub-Gaussian distributions. However, their work fails to achieve the optimal mislabeling guarantee. In contrast, our work can be thought of as a follow-up work to theirs, which might utilize their clustering output as initialization and then guarantee an optimal result in a few finite steps. 
	\end{remark}
	
	\subsection{Related works}
	
	A well-known topic close to the clustering problem we analyze is the Gaussian mixture estimation problem. The Gaussian mixture model is a well-studied problem in the literature \cite{pearson1894contributions}. A classical technique for estimating the Gaussian mixtures is the method of moments \cite{day1969estimating,anandkumar2012method,doss2020optimal}. Notably, approximating the Gaussian or sub-Gaussian mixtures does not require any separation condition, as the very close mixture components can be approximated using a single sub-Gaussian component. However, to approximate individual parameters or the cluster memberships, some separation conditions on the clusters are necessary \cite{moitra2010settling,vempala2004spectral,belkin2009learning}.

	We study the additive sub-Gaussian noise model, i.e., points in each cluster are distributed around the centroid with an independent zero mean sub-Gaussian distribution. Let $\Delta$ denote the minimum separation between the different centroids, $\sigma>0$ denote the maximum standard deviation in each coordinate (see \prettyref{sec:model} for more details), and the smallest cluster has at least $n\alpha$ many points. Indeed, if the centroids of any two components of mixtures are within a finite distance of each other, then with constant probability, we will not be able to differentiate between the labels. Thus, to attain a vanishing proportion of mislabeling, we need to allow the signal-to-noise ratio ${\Delta\over \sigma}$ to go to infinity. \cite{lu2016statistical} showed that in the absence of outliers, with good initialization, the Lloyd algorithm produces the optimal mislabeling rate $e^{-(1+o(1)){\Delta^2\over 8\sigma^2}}$. They obtain the result when ${\Delta\over \sigma}\sqrt{\alpha\over{1+{kd\over n}}}$ is large, there are no outliers, and for $n$ samples, they need $O(\log n)$ many steps for convergence. In contrast, our method produces the optimal rate in just two steps, which, in essence, is comparable to the two-round variant of EM in \cite{dasgupta2007probabilistic} for a spherical Gaussian mixture estimator. Recently, \cite{loffler2021optimality} established similar optimality mislabeling rates using spectral initialization technique of \cite{vempala2004spectral} and \cite{chen2021optimal} established a similar rate for anisotropic Gaussian mixtures; however, the robustness of these algorithms are unknown.

	Several previous works have tried to address the problem of robust clustering \cite{cuesta1997trimmed,bojchevski2017robust,zhang2018understanding}. However, none of these works discuss the label recovery guarantees. The closest contender for obtaining the mislabeling error bound, in the presence of adversarial outliers, is arguably \cite[Remark 7]{srivastava2023robust}. When ${\Delta\over \sigma\sqrt d}$ is large, by using a robust spectral initialization and then performing $k$-means, the work guarantees a best mislabeling rate of $e^{-(1+o(1)){\Delta^2\over 33\sigma^2}}$. However, their results \cite[Theorem 2]{srivastava2023robust} require the clusters to be of equal order. In contrast, we allow the minimum cluster proportion $\alpha$ to be of the order $\log n \over n$, and show that our algorithm can achieve the optimal rate $e^{-(1+o(1)){\Delta^2\over 8\sigma^2}}$ when ${\Delta\over \sigma}\sqrt{\alpha\over d}$ is large. When the cluster sizes are of similar order, we require ${\Delta\over \sigma\sqrt{kd}}$ to be large to obtain the optimal rates. The dependency on $d$ in the above requirement for our algorithm is due to the centroid estimation guarantees of the coordinatewise median, which are known to depend on $d$ in the presence of outliers \cite[Proposition 2.1]{chen2018robust}. It might be possible to use other robust alternatives, such as Tukey's Half-space median, which has been proven to show better consistency guarantees \cite[Theorem 2.1]{chen2018robust}. To improve upon the coordinatewise median estimator, for example, convert it into equivariant estimates of multivariate locations, researchers often use transformation and retransformation procedures \cite{chakraborty1999note,chaudhuri1996geometric,chakraborty1996transformation}. However, such analyses are beyond the scope of our current work. 
	
	Notably, a similar technique of clustering using mismatched metrics has been proposed previously in the Partitioning around the medoid (PAM) algorithm \cite{kaufman2009finding,rousseeuw1987clustering}. In this algorithm, one updates the centroid estimates using a point from the data set (these centroid estimates are referred to as the medoids of the clusters) based on some dissimilarity metric instead of using a proper location estimator. For example, \cite{rousseeuw1987clustering} used the $\ell_1$ distance as the dissimilarity metric and argued the robustness of the corresponding $\ell_1$ based PAM algorithm. This is close, but not the same, as choosing the coordinatewise median for the estimation step as done in our algorithm. Whether the PAM algorithm with $\ell_1$ metric provides similar statistical guarantees as our method is an open question; for other centroid-based robust clustering methods, see \cite{appert2021new,klochkov2021robust,brunet2022k}.
	
	For our analysis, we use the adversarial contamination model. In this model, upon observing the true data points, a powerful adversary can add any number of points of their choosing, and our theoretical results depend on the number of outliers added. This contamination model is arguably stronger than the traditional Huber contamination model \cite{huber1965robust,huber1992robust}, which assumes that the contamination originates from a fixed distribution via an \iid mechanism. Our contamination model is similar to the adversarial contamination model studied in \cite{lugosi2021robust,diakonikolas2019robust} for robust mean estimation. In the study of robust learning of Gaussian mixtures, an almost similar adversarial setup has been studied in \cite{bakshi2020outlier}. For robust clustering of Gaussian mixtures, \cite{liu2023robustly} examines a similar contamination model. However, their article considers different loss functions.
	
	{ 
		
		Our work is closely related to the topic of robust mean estimation. In particular, our proof technique demonstrates that we can utilize any robust mean estimation method that accurately estimates the cluster centroids in the presence of outliers to substitute for the coordinatewise median in the "Estimation" step of Algorithm 1, thereby retaining the theoretical guarantees for SubGaussian data clustering. As a result, it might be possible to use other classical robust estimators in the literature, such as the polynomial time algorithm of \cite{diakonikolas2019robust}, the trimmed mean estimator \cite{lugosi2021robust}, or the Tukey's median \cite{chen2018robust}, that aims to obtain dimension-independent error for location estimation. However, the problem with the above estimators is their runtime: the robust method of \cite{diakonikolas2019robust} has a polynomial runtime, whereas the other two estimators involve optimization over an exponential number of vectors in $\reals^d$. The problem of dimension-independent estimation of location is indeed challenging, and recent work \cite{hopkins2019hard} suggests that polynomial time is probably the best we can achieve.
		In comparison, the coordinatewise median estimator operates in a time that is almost linear in the sample size, but its location estimation error scales with the square root of the data dimension. We use the coordinatewise median estimator for our algorithm as (a) we intend to construct a fast iterative clustering algorithm and (b) the simple structure of the coordinatewise median makes the proof of our iterative algorithm simpler. We leave it for the future to investigate whether the above centroid estimators with dimension-independent error guarantees can be used to guarantee an exponentially small mislabeling.
		
		The kernel $k$-means type algorithms \cite{dhillon2004kernel,jayasumana2015kernel} use a kernel $K=\{K_{ij}\}_{i,j=1}^n$ build on $n$ points $\{X_1,\dots,X_n\}$ defined as $K_{ij}=\langle \phi(X_i),\phi(X_j)\rangle$, where $\phi$ is a map lifts the data to a high-dimensional manifold where the clusters are linearly separated, and $\langle\cdot,\cdot\rangle$ is an inner-product there. The benefit of the kernel trick lies in that one only needs to construct the $K$ matrix based on the data points, and an exact knowledge of the map $\phi$ is not required. Finding such maps $\phi$ to guarantee that the clusters in the image of $\phi$ are separated is more challenging; related approaches involve the t-SNE \cite{van2008visualizing}, albeit for projecting lower-dimensional spaces. If one has access to $\phi$, our algorithm can be applied to obtain clustering in the projected space, and its statistical properties can then be analyzed. However, such directions are beyond the scope of this current work, and we only address situations where clusters are already linearly separable.}

	\section{Suboptimality of the $\ell_1$ metric for cluster labeling}
	\label{sec:l1-subopt}
	We begin our analysis by explaining why the $\ell_1$-based greedy clustering algorithm, which also utilizes the coordinatewise median for estimating the centroids, fails to produce an optimal mislabeling. We use the additive Gaussian noise setup to provide scenarios where the suboptimal rates are observed. In other words, we make the following model assumption on the underlying data-generating distribution.
	\begin{align}\label{eq:model-gauss}
		Y_i=\theta_{z_i}+w_i,\quad w_i\simiid \calN(0,\sigma^2 I_{d\times d})  \ i=1,\dots,n,
	\end{align}
	where $z = \sth{z_i}_{i=1}^n\in[k]^n$ denote the underlying unknown component labels of the points. Given label estimates $\hat z=\sth{\hat z_i}_{i=1}^n$, we measure the mislabeling proportion as
	\begin{align}\label{eq:kmed_1}
		\ell(\hat z,z)=\frac 1n \sum_{i=1}^n \indc{z_i\neq \hat z_i}.
	\end{align}
	We intend to construct a lower bound for the expected mislabeling error, i.e., $\frac 1n \sum_{i=1}^n \PP\qth{z_i\neq \hat z_i}
	$. Given any data point $Y_i$, it is assigned the label $\hat z_i = h$ when the centroid estimate for cluster $h$ is the closest to it among all the other centroid estimates. When the centroid estimates are consistent, assigning the labels to the data points is essentially equivalent to labeling the data points according to the Voronoi regions of the true centroids for the metric used for clustering. It has already been established in the literature (while proving the optimality of $k$-means \cite{lu2016statistical}) that the $\ell_2$ based clustering is optimal for labeling Gaussian mixtures. As a result, for any metric with different Voronoi regions compared to the $\ell_2$ based Voronoi regions, one may suspect suboptimal mislabeling. However, there appears to be no study in the literature on this topic. Nonetheless, the fact that the Voronoi regions of $\ell_1$ and $\ell_2$ metrics are different, observed previously in the literature \cite{chew1985voronoi,klein1989concrete}, supplied us with the intuition behind proving suboptimality of the $\ell_1$ based clustering. Before jumping into the formal statement, we do a quick example using $k,d=2$. In two dimensions, the Voronoi regions of $\ell_1$ and $\ell_2$ metrics are strictly different if the centroids do not lie on either of the axes or on either the 45-degree or the 135-degree lines. As a result, to demonstrate a difference in performance, it is helpful to choose the centroids elsewhere. For simulation purpose we chose the cluster centroids to be $\theta_1=(-5,6)$ and $\theta_2=(5,-6)$ and $\sigma=10$. For a total of 1000 runs, we generate 500 points from each of the above components at each simulation. Then, the proportion of mislabeled points using the $\ell_1$ distance and $\ell_2$ distance-based clustering with the true centroids turns out to be 0.233 (with a standard deviation of 0.0004) and 0.218 (with a standard deviation of 0.0004), respectively. See \prettyref{fig:L1vL2} below for an illustration of one of the instances, along with the different Voronoi regions. Notably, the above example is provided for visualization purposes only, and as we will see in \prettyref{sec:experiments}, the differences in performance are more significant in various general simulation studies and real data analyses. Nonetheless, in the next result, we demonstrate that asymptotically, our algorithm outperforms the mislabeling behavior for the $\ell_1$ metric with a general number of clusters and dimensions.
	\begin{figure}[ht]
		\begin{minipage}{0.5\linewidth}
			\centering
			\includegraphics[height=8cm]{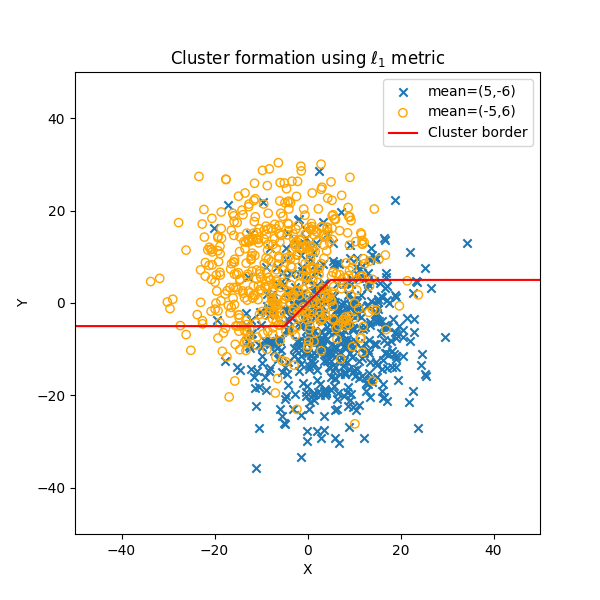}
		\end{minipage}
		\begin{minipage}{0.5\linewidth}
			\centering
			\includegraphics[height=8cm]{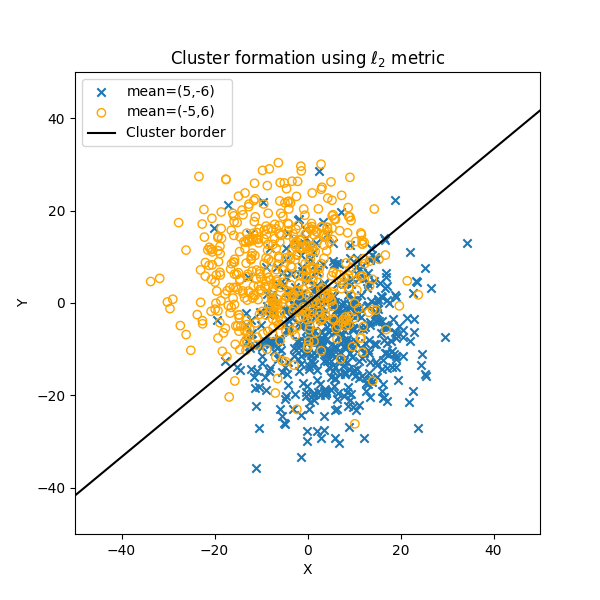}
		\end{minipage}
		\caption{Comparison of clustering via $\ell_1$ and $\ell_2$ metrics after initializing at true centroids}
		\label{fig:L1vL2}
	\end{figure}
	
	\begin{theorem}\label{thm:ell1-subopt}
		Suppose that the minimum separation between the centroids is $\Delta$. For any constant $C_0>0$, the following is satisfied given any data-dependent centroid estimates that are within $C_0\sigma$ Euclidean distance from the true centroids.
		\begin{enumerate}[label=(\roman*)]
			\item  There exists $c>0$ depending on $C_0$ and the number of clusters in the dataset such that whenever ${\Delta> c\sigma\sqrt d}$,  we have that the mislabeling error produced by the $\ell_1$ distance is at least $e^{-{\Delta^2\over 8\sigma^2}\pth{1+o(1)}}$. Here, the $o(1)$ term approaches 0 as $\Delta/\sigma$ increases to infinity.
			\item Given any $C\in (0,2)$, there exists a constant $\tilde c>0$, depending on $C_0,C$, and the number of clusters in the data set, such that whenever $\Delta>\tilde c\sigma\max\sth{\sqrt d,\log(1/\alpha)}$ the maximum mislabeling error produced by the $\ell_1$ distance is at least $e^{-{\Delta^2\over (8+C)\sigma^2}}$.
		\end{enumerate}
	\end{theorem}
	
	The core idea of the proof is the following. Using a similar technique as in the proof of \cite[Theorem 3.3]{lu2016statistical} and \cite[Theorem 2]{gao2018community}, it is not difficult to show that the mislabeling error can be bounded from below, up to a factor involving $k$, using the mislabeling error for the case when there are only two clusters present. In the two cluster cases, we choose the centers to lie in a neighborhood around the identity line (i.e., the line where all coordinates are equal), excluding the identity line itself. Then, considering that the centroid estimates are close to the actual centroids, a detailed computation based on the Gaussian probabilities reveals the lower bound. For formal proof, see \prettyref{app:l1-mislabel}.

	\section{Methodology}
	\label{sec:algorithm}
	In view of the above argument regarding the Voronoi diagrams, it seems that using a non-$\ell_2$ type metric need not achieve optimal mislabeling even though it might give rise to good algorithms for estimating the centroids. As a result, to retain the optimal mislabeling, we propose using the $\ell_2 $ metric for labeling. In the current manuscript, we analyze the performance of a $\ell_2$-based clustering algorithm when the estimation step is performed via the coordinate-wise median. Before jumping into the algorithm, we introduce some basic notation. 
	
	\noindent {\textbf{Notation}:} Given a set of $n$ real numbers $v_1,\dots,v_n$ let $v^{(1)}\geq \dots\geq v^{(n)}$ denote their order statistics. For any vector $v$ of length $n$ on the real line let us define the median by $v^{(\ceil{n/2})}$, i.e., the $\ceil{\frac n2}^{\text{th}}$ largest element. Given a set $V$ of vectors in $\reals^d$, let ${\med(V)}$ denote the coordinatewise median of the set of vectors. Let $\|\cdot \|_2$ denote the Euclidean norm unless otherwise specified. The set of integers $\sth{1,2\dots,k}$ is denoted by $[k]$. Let $Y_1,\dots,Y_n$ denote the observed data. For $s\geq 1$, let $\{\hat \theta_h^{(s)}:h\in [k]\}$ denote the estimated centroid at iteration $s$ and $\{T_h^{(s)}:h\in [k]\}$ denote the partition of $[n]$ that gives us the estimated clusters at iteration $s$. The case $s=0$ stands for initialization parameters. 
	
	We present our clustering methodology in \prettyref{algo:llone} below. For simulation purposes, we set the error threshold parameter $\epsilon$ to 0.001 and the maximum number of iterations $M$ to 100. If initial centroid estimates are unavailable beforehand, a random initializer can be used. Another well-known initializer, often used for mean-based clustering algorithms, is based on the $k$-means++ algorithm \cite{arthur2007k}. However, outliers can negatively impact the performance of $k$-means++. Recently \cite{deshpande2020robust} proposed an alternative robust initialization technique. Whether such initialization can directly provide a robust and optimal mislabeling error is beyond the scope of the current paper.
	\begin{algorithm}[t]
		\caption{The \lcomed Algorithm}\label{algo:llone}
		{\bf Input}:  Data $Y_1,\dots,Y_n$. Initial centroid estimates $(\hat \theta_1^{(0)},\dots,\hat \theta_k^{(0)})$ (or initial label estimates $(\hat z_1^{(0)},\dots,\hat z_n^{(0)})$). Error threshold $\epsilon$ and maximum iteration $\sfM$.\\
		{\bf Steps}: 
		\begin{algorithmic}[1]	
			\State Set $s=1$. 
			\If{$(\hat z_1^{(0)},\dots,\hat z_n^{(0)})$ are available} 
			compute $T_h^{(1)}=\{i\in [n]:\hat z_i^{(0)}=h\}$ for all $h\in [k]$ and directly go to the \textbf{Estimation step}.
			\EndIf
			\State \textbf{Labeling step}:
			\begin{equation*}
				\begin{aligned}
				T_h^{(s)}=
				\Biggl\{
				&i\in \sth{1,\dots,n}: \|Y_i-\hat\theta_h^{(s-1)}\|_2\leq \|Y_i-\hat\theta_a^{(s-1)}\|_2, 
				\\ &
				a\in\{1,\dots,k\}\Biggr\},
			\end{aligned}
			\end{equation*}
			with ties broken arbitrarily.
			\State \textbf{Estimation step:} 
			$\hat \theta_h^{(s)} = \med(\sth{Y_i:i\in T_h^{(s)}}),\quad h\in \sth{1,\dots,k},$
			\If{\textbf{either} $s=1$ \textbf{or} $\frac 1k\sum_{h=1}^k \|\hat\theta_h^{(s)}-\hat\theta_h^{(s-1)}\|_2^2 > \epsilon$ and $s<\sfM$}
			\State Update $s\leftarrow s+1$ and go back to the \textbf{Labeling step} and repeat the subsequent steps.
			\EndIf
		\end{algorithmic}
		{\bf Output}: $(\hat \theta_1^{(s)},\dots,\hat \theta_k^{(s)})$ and label estimate for $Y_i$ given by $\hat z_i^{(s+1)} = \argmin_{h\in\{1,\dots,k\}}\|Y_i-\hat\theta_h^{(s)}\|_2,$ with ties broken arbitrarily.
	\end{algorithm}

	\section{Main results}
	\label{sec:main-results}
	
	\subsection{Sub-Gaussian mixture model}
	\label{sec:model}
	In this section we explore the theoretical guarantees of the \lcomed algorithm. We study the algorithm in the additive sub-Gaussian error model. In our model, the observed data $Y_1,\dots,Y_n\in \reals^d$ are distributed as
	\begin{align}\label{eq:model}
		Y_i=\theta_{z_i}+w_i,\ i=1,\dots,n,
	\end{align}
	where $\sth{z_i}_{i=1}^n\in[k]^n$ denote the underlying unknown component labels of the points, and the $\{w_i\}_{i=1}^n$ denote the error variables. We assume that $\{w_i\}_{i=1}^n$ are independent zero mean sub-Gaussian vectors with parameter $\sigma>0$, i.e.,
	\begin{align}\label{eq:subg-defn}
		\EE\qth{e^{\langle a,w_i \rangle}}
		\leq e^{\sigma^2\|a\|_2^2\over 2}, \text{ for all $i\in \sth{1,\dots,n}$ and $a\in \reals^d$}.
	\end{align}
	In addition, we assume that after observing the data, an adversary can add $n^{\out}$ many data points of choice. Suppose that after $s$-steps the estimated centers are $\hat\theta^{(s)}_1,\dots,\hat\theta^{(s)}_k$ and the estimated labels are $\hat z^{(s)}=\sth{\hat z^{(s)}_1,\dots,\hat z^{(s)}_n}$. The initialization step corresponds to $s=0$. The above estimates are computed successively as
	\begin{align}
		\begin{gathered}
			\hat \theta^{(s)}_h
			=\med\sth{Y_i:\hat z^{(s)}_i=h},\ h\in\sth{1,2,\dots,k},\\
			\hat z^{(s+1)}_i
			= \argmin_{h\in[k]} \|Y_i-\hat\theta^{(s)}_h\|^2_2.
		\end{gathered}
	\end{align}
	We provide guarantees for mislabeling in \prettyref{eq:kmed_1} and the centroid estimation errors $\sth{\|\hat \theta_h^{(s)}-\theta_h\|_2^2:h=1,\dots,k}$.

	\subsection{Optimality guarantees for mislabeling proportion}
	\label{sec:res-wo-out}
	
	To better present our results, we first introduce some notation. For all $h,g\in [k]$, define 
	\begin{equation}\label{eq:cluster-def}
		\begin{gathered}
			T_h^*=\sth{i\in [n]:z_i=h},
			\hat T_h^{(s)}=\sth{i\in [n]:z_i^{(s)}=h}\\
			n^*_h = \abs{T^*_h}, n^{(s)}_h = \abs{T^{(s)}_h}, n^{(s)}_{hg} = \abs{T^*_h \cap T^{(s)}_g}
		\end{gathered}
	\end{equation}
	Note that for $s\geq 1$ this implies
	\begin{align}
		T_h^{(s)}=\sth{i\in[n]: \|Y_i-\hat\theta_h^{(s-1)}\|_2\leq \|Y_i-\hat\theta_a^{(s-1)}\|_2,a\in[k]}.
	\end{align}
	with ties broken arbitrarily. Let us define the minimum fraction of points in the data set that come from a single component in the sub-Gaussian mixture as
	\begin{align*}
		\alpha=\min_{g\in[k]} {n_{g}^*\over n}.
	\end{align*}
	Define the cluster-wise correct labeling proportion at step $s$ as
	\begin{align*}
		H_s=\min_{g\in [k]}\sth{\min\sth{{n_{gg}^{(s)}\over n_g^*},{n_{gg}^{(s)}\over n_{g}^{(s)}}}}.
	\end{align*}
	We denote by $\Delta = \min_{g\neq h\in [k]}\|\theta_g-\theta_h\|_2$ the minimum separation between the centroids. Let $\Lambda_s$ denote the error rate of estimating the centers at iteration $s$
	\begin{align*}
		\Lambda_s = \max_{h\in[k]}\frac 1{\Delta} \|\hat \theta_h^{(s)}-\theta_h\|_2.
	\end{align*}
	Our results are presented based on the signal to noise ratio in the model, defined as 
	$$
	\snr = {\Delta\over 2\sigma}.
	$$
	When $\snr$ is very small, there exist two clusters such that information theoretically we can not distinguish between the corresponding centroids with positive probability, even when the labels are known. As a consequence, we only study the mislabeling guarantees when the $\snr$ is large.
	
	We have the following theorem.
	\begin{theorem}
		\label{thm:main}
		Suppose that an adversary, after analyzing the data $Y_1,\dots,Y_n$ coming from the subgaussian mixture model \eqref{eq:model}, adds $n^{\out}=n\alpha(1-\delta)$ many outliers of its choice for some $\delta\in (0,1]$. Then there exist constants $C,c_0,c_1,c_2,c_3,c_4>0$ such that the following are satisfied. 
		If $n\alpha\geq c_0\log n$, 
		$$\sqrt \delta\cdot \snr\sqrt{\alpha/d}>C$$
		and the clustering initialization satisfies
		$${H_0}\geq \frac 12+\frac {c_1}{\pth{{\snr}\sqrt{\alpha\over d}}^2}\quad  \text{ or } \quad \Lambda_0\leq \frac 12-{c_2\over {\sqrt \delta\cdot {\snr}\sqrt{\alpha\over 1+dk/n}}},$$
		then our algorithm guarantees for all $s\geq 2$ 
		$$\ell(\hat z^{(s)},z)\leq \exp\sth{-\frac 12\pth{1-{c_4\over (\sqrt \delta \cdot \snr\sqrt{\alpha/d})^{1/2}}}{(\snr)^2}},
		$$
		with probability $1-4(k^2+k)n^{-4}-16dn^4e^{-0.3n}- 2k^2e^{-{\snr\over 4	}}.$
		
		
	\end{theorem}
	
	{ \begin{remark}[Necessity of the assumptions on $\snr$]
			The sufficient condition involving $\snr$ exhibits an optimal dependency on the data dimension $d$ for our coordinatewise median-based clustering algorithm to work. To see that, consider the Gaussian mixture model with unit variance, i.e., $\sigma=1$. Then \cite[Proposition 2.1]{chen2018robust} explains that given $X_1,\dots,X_n\simiid \delta P_{\theta}+(1-\delta) Q$, where $P_{\theta} N(\theta,I_d)$ and $Q$ is some contaminating distribution, the coordinatewise median $\hat \theta$ of $\{X_1,\dots,X_n\}$ satisfies
			$$
			\sup_{Q,\theta}\PP[\|\hat\theta-\theta\|_2\geq C\sqrt d]\geq c
			$$
			for constants $C,c>0$. This implies that with a constant probability, we can construct a scenario where the centroid estimation error for the coordinate-wise median will be greater than $C\sqrt {d} $. As a result, unless $\snr=\frac 12\min_{g\neq h}\|\theta_h-\theta_g\|_2$ is significantly larger than $C\sqrt d$, we can construct a Gaussian mixture model and outlier distribution such that, the \lcomed algorithm will produce inconsistent clustering with a constant probability. We may be able to achieve similar mislabeling guarantees for robust clustering with relaxed conditions on the data dimension by using alternative approaches that incorporate an additional dimension reduction step. For example, the separation condition might be improved to $\Delta\geq C\sigma \sqrt{\min\{d,k\}}$ by first obtaining a $k$ dimensional robust spectral projection of the data as described in \cite[Proposition 3]{srivastava2023robust} before subsequently running the \lcomed algorithm. However, the exact technical details are beyond the scope of current work. We leave it for future work to determine the optimal dependency of the centroid separation conditions on $\delta$.
		\end{remark}
		
		\begin{remark}[Robust initialization that meets our theoretical requirements]
			A theoretically sound option for robust initialization has been proposed in the recent work of \cite{jana2024provable}, known as Initialization via Ordered Distances ({\sf IOD}). In view of the result (Theorem 12) in their paper, we have that whenever the number of outliers is bounded as $n^\out\leq {n\alpha^2\over 32}$, the {\sf IOD} estimator guarantees estimates $\hat\theta_1,\dots,\hat\theta_k$ of the true centroids $\theta_1,\dots,\theta_k$ such that for a permutation $\pi$ of $\{1,\dots, k\}$ we have
			\begin{align}\label{eq:rev_1}
				\max_{i\in [k]}\|\theta_i-\hat \theta_i\|_2\leq {\Delta\over 3},
			\end{align}
			with a high probability, where $\Delta$ denotes the minimum separation of the centroids. We propose to use the centroid estimates $\{\hat\theta_1,\dots,\hat\theta_k\}$ as our initialization $\{\hat\theta_1^{(0)},\dots,\hat\theta_k^{(0)}\}$. Note that given any initialization $\{\hat\theta_1^{(0)},\dots,\hat\theta_k^{(0)}\}$, our theory (Theorem 2) requires that $$
			\Lambda_0={\max_{i\in [k]}\|\theta_i-\hat \theta_i^{(0)}\|_2\over\Delta}
			\leq \frac 1{2+c}
			$$
			for any small constant $c>0$, which is satisfied by the {\sf IOD} estimator, in view of \eqref{eq:rev_1}. This implies that whenever $\snr$ is large, our \lcomed algorithm initialized with the {\sf IOD} algorithm achieves the desired mislabeling guarantees. Unfortunately, the {\sf IOD} algorithm is extremely slow and needs around $k^{O(k)}n^2(d+\log n)$ runtime to output initial centroid estimates. We leave it for future work to explore the possibility of developing a fast and robust initialization algorithm that can guarantee our initialization conditions and yield a complete and efficient algorithm when combined with our iterative clustering strategy.
	\end{remark}}
	
	\begin{remark}[Comments on the initialization conditions]
		The result \cite[Theorem 3.3]{lu2016statistical} shows that it is not possible to improve the mislabeling error beyond the rate $e^{-(1+o(1)){\Delta^2\over 8\sigma^2}}$. { Our algorithm achieves the best mislabeling proportion for all large signal-to-noise ratios such that $(\sqrt \delta \snr \sqrt{\alpha/d})^{-1} = o(1)$. In that setup, using the definition $\snr={\Delta/2\sigma}$, we can rewrite the mislabeling guarantees in Theorem 2 as $\exp\sth{-(1-o(1)){\Delta^2\over 8\sigma^2}}$, which matches the guarantee in \cite[Theorem 3.3]{lu2016statistical}.} This establishes the optimality of the \lcomed algorithm. Notably, in the absence of outliers and fixed dimensions, one can argue that the initialization condition required to achieve the above mislabeling error rate is significantly weaker than the one used in \cite{lu2016statistical}. 
		
		\begin{itemize}
			\item Let $G_s=1-H_s$ denote the cluster-wise mislabeling rate 
			$$
			G_s = \max_{h\in [k]}\max\sth{{\sum_{g\neq h\in [k]}n_{gh}^{(s)}\over n_h^{(s)}},{\sum_{g\neq h\in [k]}n_{hg}^{(s)}\over n_h^*}}.
			$$
			Then the mislabeling guarantees in \cite[Theorem 3.2]{lu2016statistical} require 
			$$G_0\leq \frac 1\lambda\pth{\frac 12 - {c'\pth{\sqrt \delta\cdot {\snr}\sqrt{\alpha\over 1+dk/n}}^{-\frac 12}}},$$
			where $c'$ is some absolute constant and $\lambda = \frac 1{\Delta}\max_{g\neq h\in [k]}\|\theta_g-\theta_h\|_2$ denotes the ratio of maximum and minimum separation between the centroids. On the other hand, the requirement of \prettyref{thm:main} in terms of $G_0$ is given by $G_0\leq \frac 12-\frac {c_1}{\pth{{\snr}\sqrt{\alpha\over d}}^2}$. This is precisely due to the robustness property of our centroid estimate. Even when more than half of the points in all the clusters are correctly labeled, the naive mean-based centroid estimate may not perform well if contamination originates from a distant cluster. However, due to its 50\% breakdown property, the coordinatewise median can maintain the stability of the process in similar scenarios. 
			
			\item For any constant $\delta>0$, the initialization condition based on $\Lambda_0$ in the above theorem is also an improvement on the existing result in \cite[Theorem 3.2]{lu2016statistical}, which required $\Lambda_0\leq \frac 12 - {c_2}{\pth{{\snr} \sqrt{\alpha\over 1+{dk/n}}}^{-1/2}}$. The improvement is achieved by using some sharper identities compared to those used in the above reference, and similar initialization conditions can be derived for the results involving the Lloyd algorithm as well.
		\end{itemize}
	\end{remark}
	
	\begin{remark}[Convergence analysis of the \lcomed algorithm]
		The proof of \prettyref{thm:main} shows that the convergence of the \lcomed algorithm has three stages. Let $\delta$ be any constant in $(0,1)$. Then, starting from a cluster-wise correct labeling proportion $H_0$ around a small neighborhood of 1/2, after the first iteration, the \lcomed algorithm achieves a correct labeling proportion $H_1\geq \frac {1}{2} + c$, for some large constant $c$. Once this low mislabeling proportion is achieved, the robustness property of the coordinatewise median kicks in to produce a centroid estimate that differs on each coordinate by, at most, a constant (depending on $c$ and the quantile function of the data distribution on that coordinate). For any fixed $d$, this centroid estimation error is sufficient for the $\ell_2$ based labeling to produce the optimal statistical rate we aim for. Once this low mislabeling rate is achieved, the estimation and the labeling errors in all subsequent stages also remain good. In the analysis of the $k$-means algorithm, achieving a good centroid estimate usually requires more iterations, as the mean estimator does not necessarily guarantee low centroid estimation errors from low mislabeling errors.
		
	\end{remark}
	
	{ \begin{remark}[Runtime of our algorithm]\label{rmk:runtime}
			The \lcomed method reaches the theoretical mislabeling limit stated in \prettyref{thm:main} within $O(dn(k+\log n))$ time, which is almost linear in the sample size. To illustrate this, first note that we require a constant number of iterations to achieve the mislabeling guarantee stated in \prettyref{thm:main}. At each iteration, it takes $O(dnk)$ steps to assign all data points to their nearest cluster centroids. At each iteration $s$, for each cluster $g\in [k]$, it takes $O(d n_g^{(s)}\log (n_g^{(s)}))$ time to compute the coordinatewise median, where $n_g^{(s)}$ is the number of points in the estimated $g$-th cluster. Hence, the total time to compute all the cluster centroids is $O(\sum_{g\in [k]}d n_g^{(s)}\log (n_g^{(s)}))$, which is at most $O(dn\log n)$.
	\end{remark}}

	%
	
	\subsection{Statistical guarantees for centroid estimation}
	
	As mentioned in the previous section, once the \lcomed algorithm starts producing a low mislabeling error, the centroid estimates made by the coordinatewise median estimator differ from the true centroid only by a constant amount. In each coordinate, the worst instances of the deviations can be quantified using the corresponding data distribution function in each cluster component. To this end, let $F_{ij}$ denote the distribution of the $j$-th coordinate of the $i$-th error random variable $w_i$. Then, we have the following guarantees for estimating $\{\theta_g\}_{g=1}^k$ in the presence of $n^\out$ many extra adversarial points.
	\begin{theorem}\label{thm:centroid}
		Suppose that $n\alpha\geq 10\log n$ and the constraints in \prettyref{thm:main} on $H_0,\Lambda_0$ are satisfied. Then given any $\tau >0$, there exists a constant $C_\tau>0$ such that whenever $\sqrt \delta \cdot\snr\sqrt{\alpha/d}\geq C_\tau$ we get with probability at least $1-8n^{-1}-16dn^4e^{-0.3n}- 2k^2\exp\sth{-{\snr\over 4}}$, for $s\geq 2$
		\begin{equation*}
			\begin{aligned}
			\|\hat \theta_h^{(s)}-\theta_h\|_2^2\leq {d}
		&\Biggl[\max_{i\in T_h^*}\max_{j\in [d]} \max\Biggl\{\abs{F_{ij}^{-1}\pth{\frac 12 +p_\tau}},
		\\
		&\abs{F_{ij}^{-1}\pth{\frac 12 -p_\tau}}\Biggr\}\Biggr]^2,\quad s\geq 2,
		\end{aligned}
		\end{equation*}
		where 
		$p_\tau= \sqrt{3\log n\over n_h^*}+e^{-{1\over 2+\tau/4}(\snr)^2}+{n^\out\over n^\out+n_h^*}$.
	\end{theorem}
	For the sake of explaining the above result, suppose that all the coordinates of the error random variables have median zero (i.e., $\sth{\theta_{i}}_{i=1}^k$ are coordinatewise median of the clusters at the population level). Then whenever ${n^\out\over n_h^*}\to 0$ and $\snr\to \infty$, the final centroid estimate for the $h$-th cluster is consistent. In particular, if in each coordinate, the errors of the $h$-th cluster are distributed as the univariate Gaussian distribution with variance $\sigma^2$, then whenever
	$$
	{n^\out}\leq {\log n\over C}, \quad 
	\snr\geq C\sqrt{\log n}
	$$
	for a large enough constant $C>0$, with high probability we can achieve $O\pth{{d\sigma^2\over n_h^*}\log n}$ error rate for estimating the centroid of the $h$-th cluster. This is within a logarithmic factor of the optimal error rate for estimating the location when $n_h^*$ many independent data points are available from the $\calN(\theta_h,\sigma^2 I_{d\times d})$ distribution.

	\section{Numerical experiments}
	\label{sec:experiments}
	
	\subsection{Choice of initialization and error metric}

	We evaluate the performance of our algorithm against its two closest competitors: the $ \ell_1$-metric-based \kmed algorithm, which performs both the cluster labeling step and estimation step using the $ \ell_1$ metric, and the $ \ell_2^2$-distance-based \kmeans algorithm, which uses the $\ell_2^2$ distance for both steps. In theory, all three algorithms are expected to provide good statistical guarantees when initialized near the ground truth. However, in practice, a good initialization condition is not always met, and one often uses random initialization to initiate the clustering process. Accordingly, we explore the performance benefits of these clustering algorithms using the following initialization. 
	\begin{itemize}[leftmargin=*]
		\setlength\itemsep{0.5em}
		\item \textbf{Random}: We randomly select $k$ data points in the given dataset as the initial centroids of each clustering algorithm.  
		\item \textbf{Omniscient}: To assess the performance of the clustering techniques if the initializations were close to the ground truth, we initialize all the methods using the exact centroids.
	\end{itemize}
	Suppose the final labels for the true (excluding outlier) data points are $\hat z_1,\dots, \hat z_n$
	For both of the initializations, we compare the mislabeling proportion (MP) defined as $\ell(\{\hat z_i\}_{i=1}^n,\{z_i\}_{i=1}^n)$ in \eqref{eq:kmed_1}.
	
	{ \begin{remark}
			Although random initialization appears to work well in many clustering examples, provable guarantees are rarely established. The spectral initialization, often used in related literature \cite{lu2016statistical}, essentially refers to a combination of dimension reduction strategy via spectral Projection and a subsequent clustering technique (e.g., an approximate $k$-means optimization \cite{kumar2004simple} used in \cite[Algorithm 2]{loffler2021optimality}) to guarantee a decent cluster labeling that need not be the optimal one. Then, a subsequent iterative clustering is carried away to optimize the clustering output further. In other words, spectral initialization will itself come with the subpart of finding an initial clustering from the low-dimensional projections. Nonetheless, standard spectral projections rarely guarantee robustness against adversarial outliers, making the corresponding clustering techniques vulnerable to adversarial perturbations. Nevertheless, we will discuss numerical experiments based on a robust spectral projection at the end of this section.
	\end{remark}}
	
	\subsection{Experiments with synthetic datasets}
	
	In this section, we evaluate our proposed algorithm (i.e., \lcomed) on synthetic datasets and compare its performance in terms of MP in \eqref{eq:kmed_1} with classical \kmeans (e.g., the Lloyd–Forgy algorithm \cite{lloyd1982least}) and the \kmed (e.g., the $\ell_1$-based greedy clustering algorithm \cite{bradley1996clustering}). We simulate points from $\calN(\theta,\sigma^2 I_{d\times d})$ distribution (dimension $d$ and standard deviation $\sigma$ to be specified later) and with 4 different $\theta$ values (i.e., we have 4 clusters). The centroids of the cluster components are generated uniformly from the surface of the sphere of radius five around the origin. For each cluster, we generate 100 data points. To analyze the robustness guarantees, we add outlier points generated using the $\calN(\theta^\out,(\sigma^\out)^2I_{d\times d})$ distribution to the existing true dataset. 
	
	\subsubsection{Experiment setup}
	Our experiments are divided into the following regimes.
	\begin{itemize}
		
		\item \textbf{Different outlier proportions:}  We fix the data dimension at 10 and $\sigma=2$. The outlier points are generated with $\sigma^\out = 10$ and $\theta^\out=0$. We vary the number of outliers in the set $
		\{0,20,40,60,80\}$ (i.e., the proportion of outliers with respect to a single cluster varies in the set $\{0,0.2,0.4,0.6,0.8\}$).
		
		\item \textbf{Different outlier variances:} We fix the data dimension $d$ at 10. To generate data, we use $\sigma=2$. We add 60 outlier points. For generating the outlier points, we use $\theta^\out =0$, and we vary $\sigma^\out$ from 1 to 20.
		
		\item \textbf{Different dimensions:} The true points are generated with $\sigma=2$. We add 60 outlier points. The outlier points are generated from the multivariate Gaussian distribution with the $\sigma^\out = 10$ and $\theta^\out=0$. We vary the data dimension $d$ from 2 to 20.	
		
		\item \textbf{Different outlier locations:} The true points are generated with $\sigma=1$. We add 40 outlier points. We fix $d$ at 10. The outlier points are generated with $\sigma^\out=2$, and $\theta^\out$ is located in a randomly chosen direction with the norm varying within [0,100].
		
	\end{itemize}
	We repeat all the experiment setups 5,000 times to estimate the mislabeling proportion and its 95
	
	\begin{remark}
		Note that even though numerous methods in the literature aim to perform robust clustering (for instance, see the techniques compared with in \cite{srivastava2023robust}), in essence, almost all of them first use a spectral type dimension reduction technique and then apply a final iterative clustering step, either $k$-means, $k$-median or something similar. As a result, it might not be very sensible to compare our algorithm to those sophisticated methods as we can modify our method using all the dimension reduction techniques present therein. Hence, it is sufficient to focus on obtaining an iterative clustering technique with outstanding performance that can be used in combination with any desired dimension reduction techniques. Specifying the choice of robust dimension reduction technique is beyond the scope of current work.
	\end{remark}
	
	\subsubsection{Results}

	Next, we present the numerical study describing the effect of outlier proportions in \prettyref{fig:outlier_prop}. When the proportion of outliers is close to zero, the \kmeans algorithm performs quite well. In the absence of outliers, the mean of Gaussian samples is a consistent estimator of the location parameter and has an asymptotically lower variance than the median. With better location estimators and the excellent clustering properties of the $\ell_2$-based labeling step, \kmeans is expected to perform better than the other algorithms. As the number of outliers increases, the location estimates produced by the naive mean estimate tend to perform worse, increasing the mislabeling proportion for the \kmeans estimates. On the other hand, the median-based algorithms are less affected due to the robust location estimates. The comparatively worse performance of the \kmed algorithm can be attributed to the use of the $\ell_1$ metric in the labeling step.
	
	\begin{figure}[t]
		\centering
		\begin{center}
			\begin{minipage}{0.5\textwidth}
				\begin{center}
					{\small Random initialization}\\
					\includegraphics[height=5.5cm]{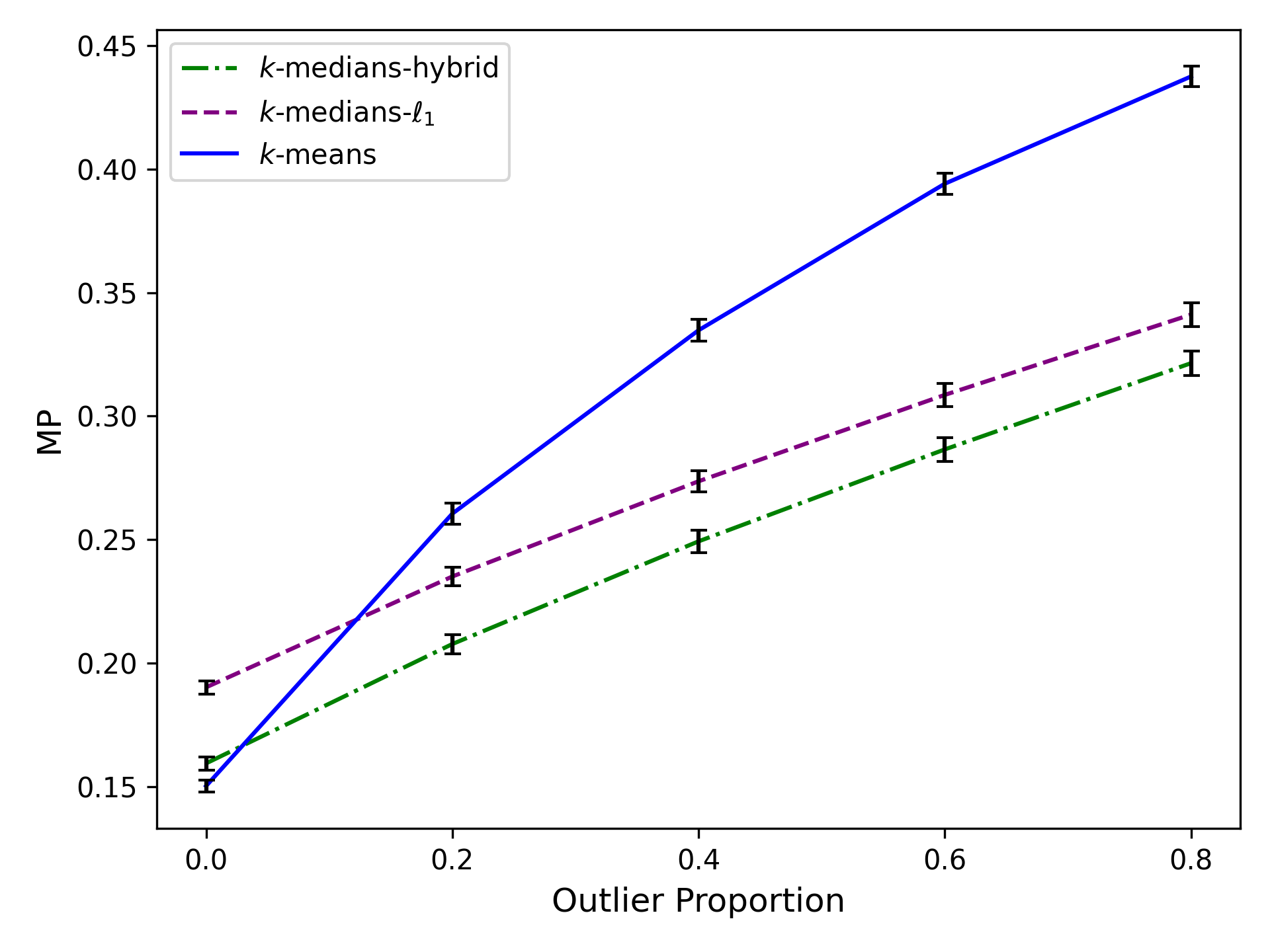}
				\end{center}
			\end{minipage}\hfill
			\begin{minipage}{0.5\textwidth}
				\begin{center}
					{\small Omniscient initialization}\\
					\includegraphics[height=5.5cm]{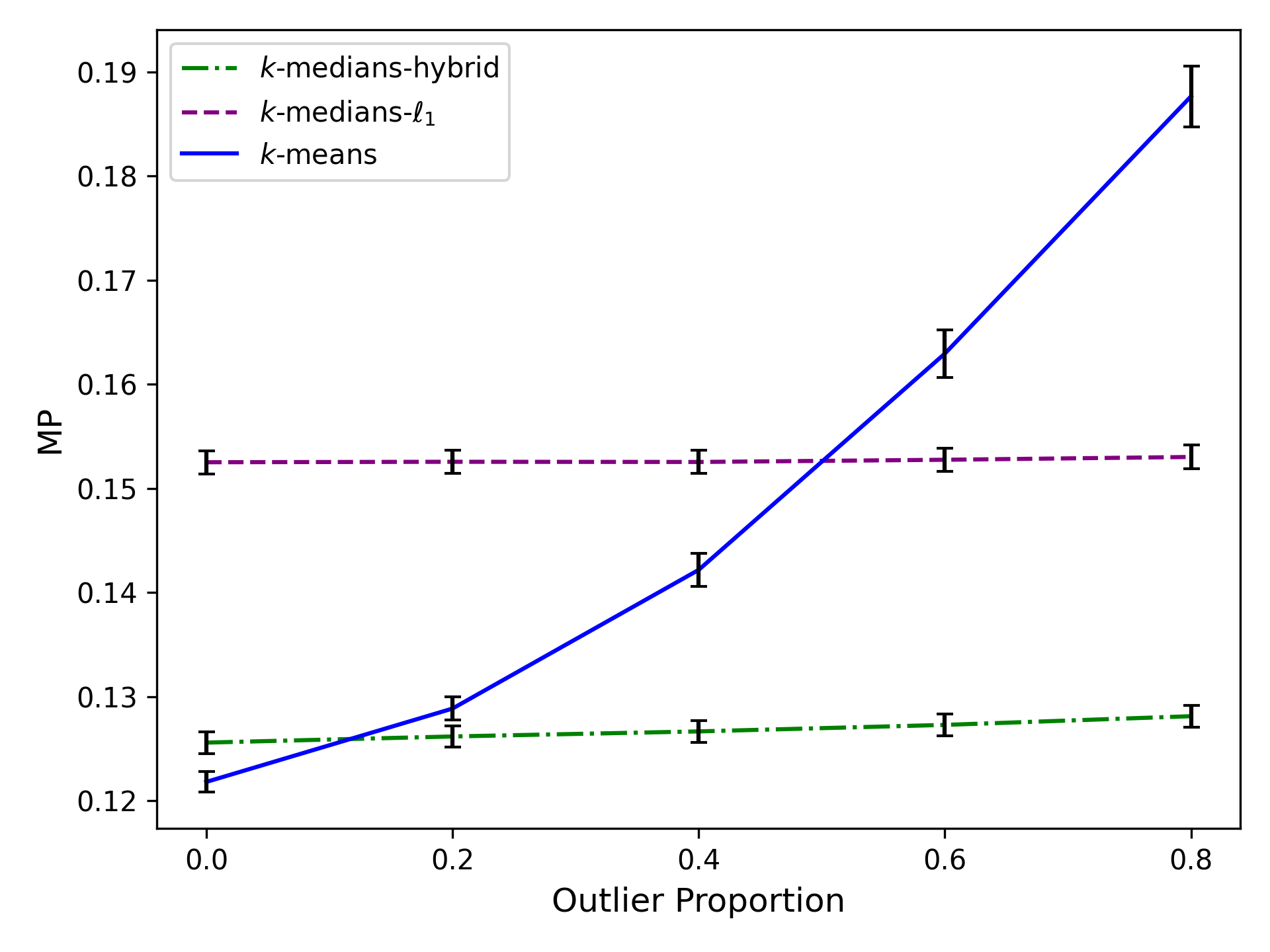}
				\end{center}
			\end{minipage}
			\caption{The effect of outlier proportions on all three algorithms.}
			\label{fig:outlier_prop}
		\end{center}
	\end{figure}
	
	\begin{figure}[t]
		\centering
		\begin{minipage}{0.5\textwidth}
			\begin{center}
				{\small Random initialization}\\
				\includegraphics[height=5.5cm]{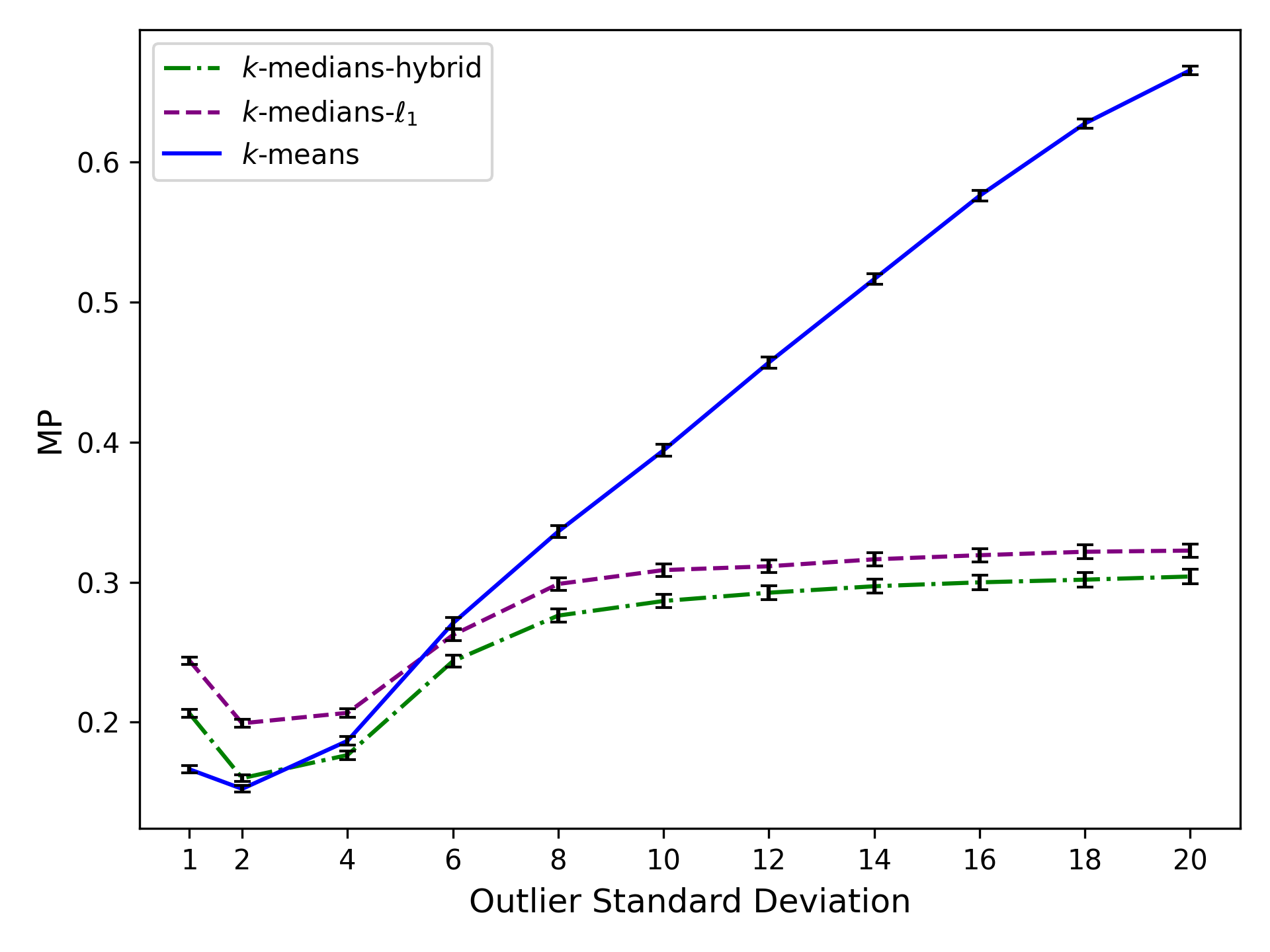}
			\end{center}
		\end{minipage}\hfill
		\begin{minipage}{0.5\textwidth}
			\begin{center}
				{\small Omniscient initialization}\\
				\includegraphics[height=5.5cm]{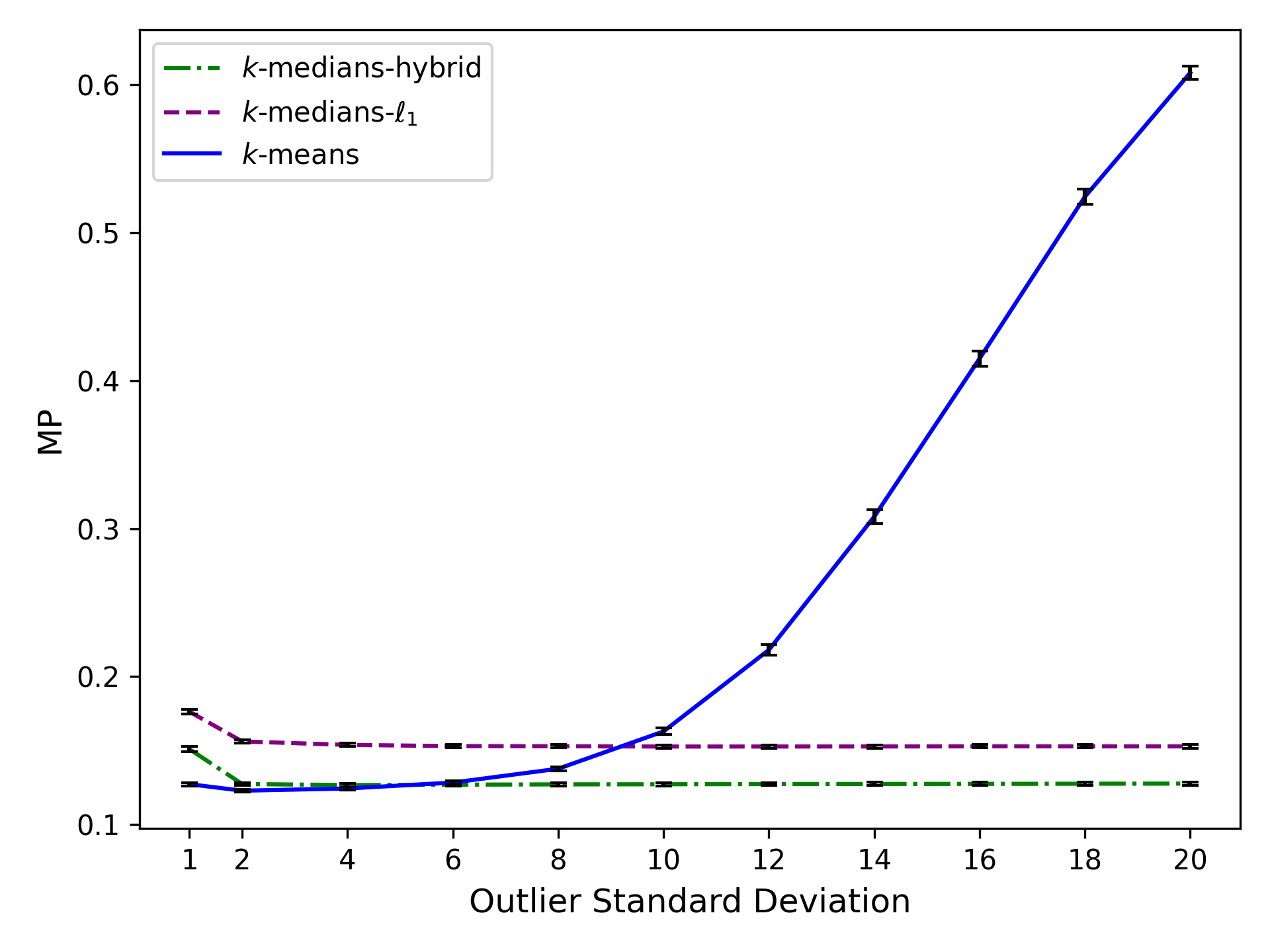}
			\end{center}
		\end{minipage}
		\caption{The effect of outlier variance on all three clustering algorithms.}
		\label{fig:variance}
	\end{figure}

	The effect of variances in the outlier distribution has been demonstrated in \prettyref{fig:variance}. Notably, even in the presence of outliers, when the outlier standard deviations are low, the \kmeans algorithm performs better than the other algorithms. We postulate that the outliers do not distort the location estimates in these setups enough to affect the final mislabeling proportions. As the outlier standard deviations increase, a significant portion of the outliers tends to be located away from the true data points, which disturbs the centroid estimation process and produces high mislabeling. Similar to the previous setups, the other clustering algorithms are affected less.
	
	The effect of data dimensions in the outlier distribution has been demonstrated in \prettyref{fig:dimension}. The increase in dimension of the system increases the surface of the sphere on which the centroids of the clusters lie. As the centroids are chosen uniformly on this surface, the increase in dimension also increases the average separation between the centroids. This is expected to induce a decline in the mislabeling proportion when the initializations are close to the actual centroids. However, as the dimensions increase, the outliers also tend to move away from the origin, and as a result, the mean estimates will perform poorly. In contrast, the median is expected to have better performances. Consequently, our algorithm produces a lesser mislabeling proportion among the two $\ell_2$ based clustering algorithms. 
	
	\begin{figure}[t]
		\centering
		\begin{minipage}{0.5\textwidth}
			\begin{center}
				{\small Random initialization}\\
				\includegraphics[height=5.5cm]{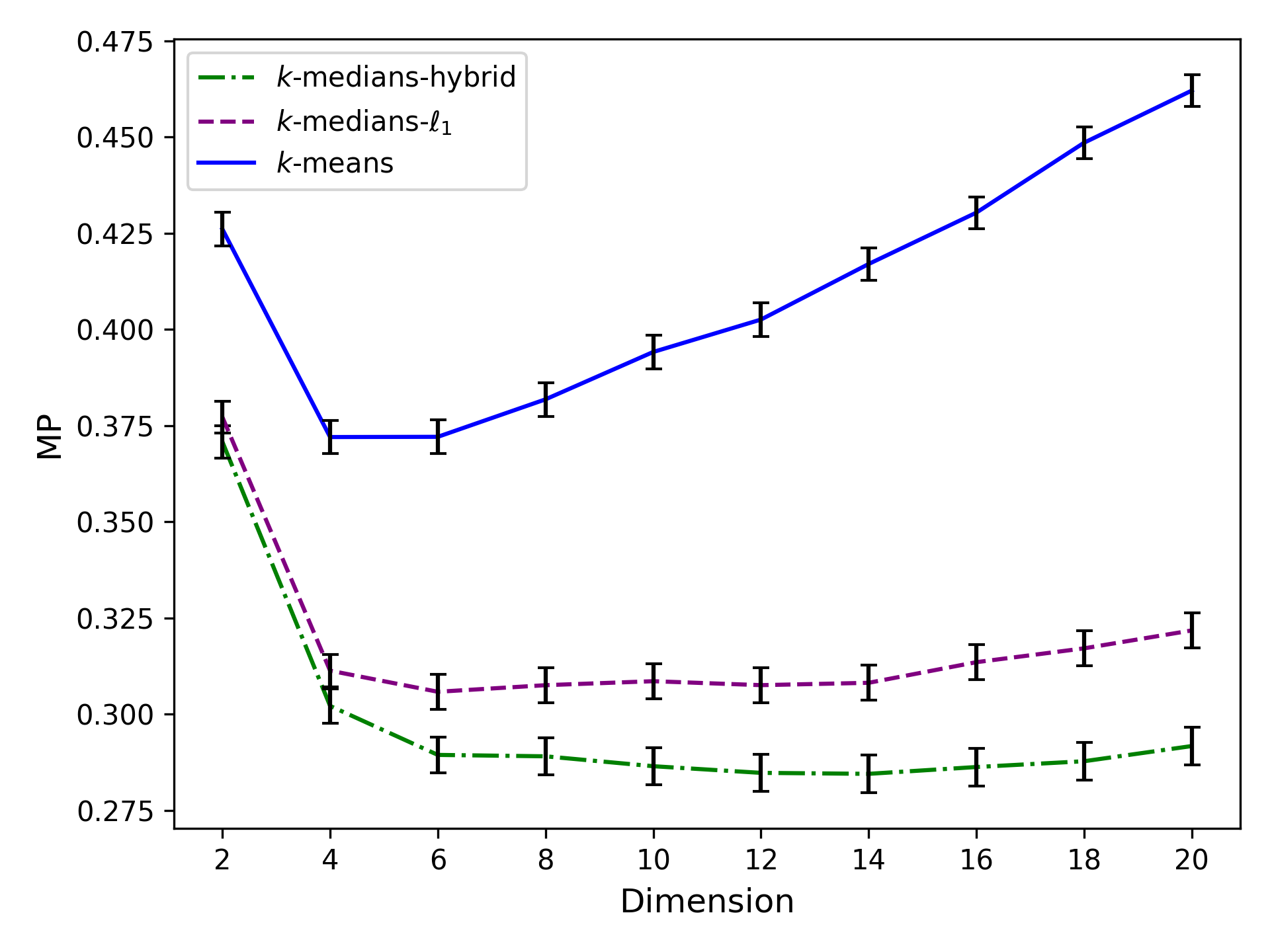}
			\end{center}
		\end{minipage}\hfill
		\begin{minipage}{0.5\textwidth}
			\begin{center}
				{\small Omniscient initialization}\\
				\includegraphics[height=5.5cm]{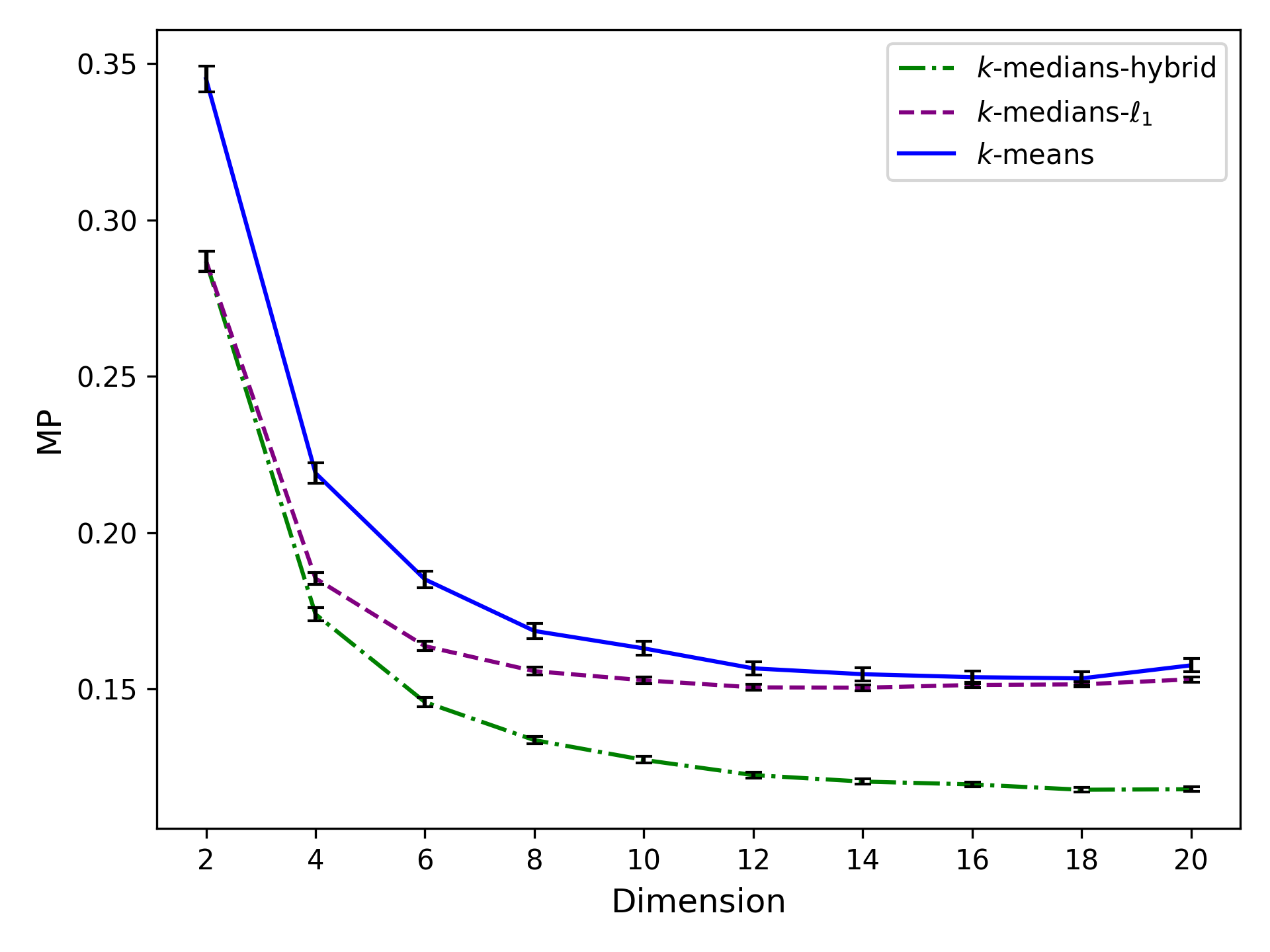}
			\end{center}
		\end{minipage}
		\caption{The effect of data dimension on all three clustering algorithms.}
		\label{fig:dimension}
	\end{figure}
	
	In \prettyref{fig:radius}, we study the effect of the location of outliers. As the centroid of the outlier distribution, the centroid estimates for \kmeans are expected to be pulled toward the same direction. Once the outliers are far away, the algorithm will start detecting the outliers as a single cluster. As a result, after a certain threshold on the norm of $\theta^\out$, the mislabeling error will stabilize at a high value. As the coordinatewise median-based centroid estimates will evolve according to the empirical order statistics, the centroid estimates will not move beyond the constellation of the true points for a moderate number of outliers. As a result, the mislabeling error will still stabilize but at a much smaller value. 
	\begin{figure}[t]
		\centering
		\begin{minipage}{0.5\textwidth}
			\begin{center}
				{\small Random initialization}
				\\
				\includegraphics[height=5.5cm]{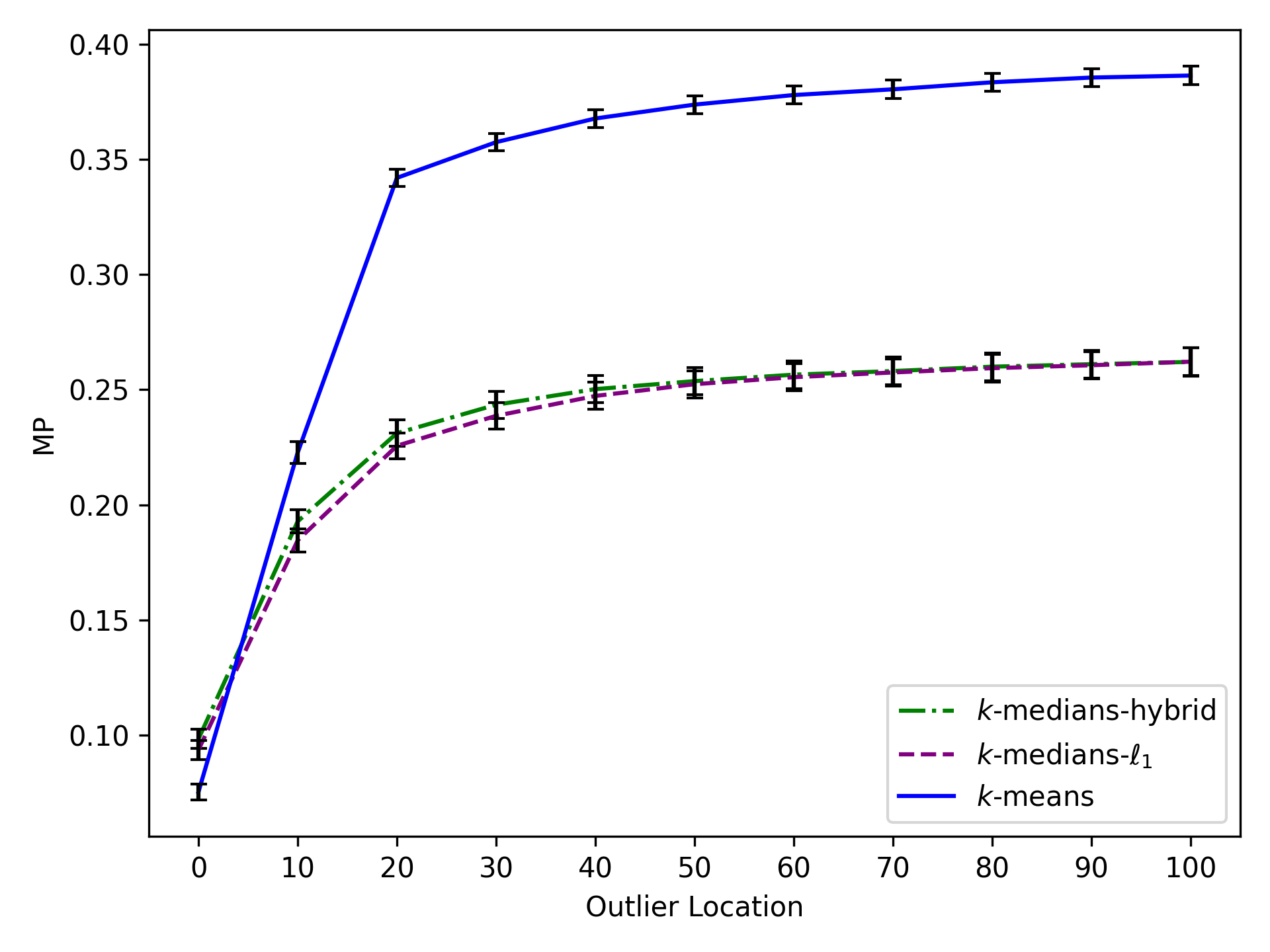}
			\end{center}
		\end{minipage}\hfill
		\begin{minipage}{0.5\textwidth}
			\begin{center}
				{\small Omniscient initialization}\\
				\includegraphics[height=5.5cm]{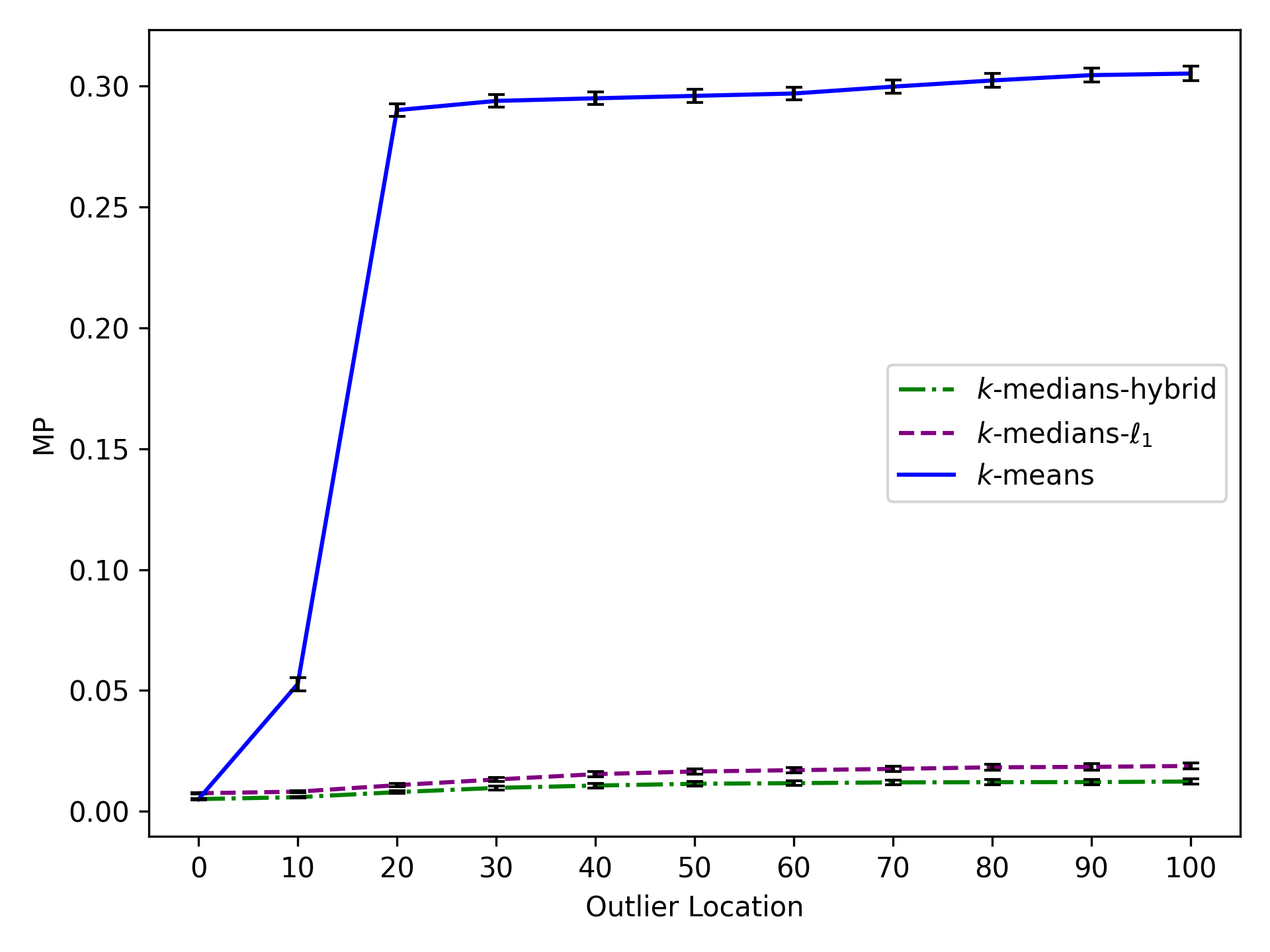}
			\end{center}
		\end{minipage}
		\caption{The effect of outlier location on all three algorithms.}
		\label{fig:radius}
	\end{figure}
	
	\subsection{Analysis with real datasets}
	Furthermore, we evaluate our proposed algorithm on public datasets:
	Letter Recognition ($D_{letter}$) \cite{misc_letter_recognition_59} and Pen-Based Recognition of Handwritten Digits ($D_{digit}$) \cite{misc_pen_based_recognition_of_handwritten_digits_81}. The $D_{letter}$ dataset contains 16 primitive numerical attributes (statistical moments and edge counts) of black-and-white rectangular pixel displays of the 26 capital letters in the English alphabet. The character images were based on 20 different fonts and each letter within these 20 fonts was randomly distorted to produce a file of 20,000 unique stimuli. The $D_{digit}$ dataset is a digit database containing attribute information on handwritten digits form 44 writers. These writers are asked to write 250 digits in random order inside boxes of 500 by 500 tablet pixel resolution. The 16 attributes of these images were collected in the dataset and we use all of them for the clustering work.
	%
	%
	
	\subsubsection{Experiment setup}
	
	For each public dataset, we choose 300 data points from three classes as inliers, where each class selects 100 data points.   
	Below are two ways we employed to obtain the outliers. 
	
	\begin{itemize}
		\item \textbf{Outliers from Multiple Classes (OMC):} We randomly choose outliers from the remaining classes. We vary the number of outliers in the set $
		\{0,20,40,60,80\}$.
		\item \textbf{Outliers from One Class (OOC):} We  choose outliers only from one of the remaining classes. As before, we vary the number of outliers in the set $
		\{0,20,40,60,80\}$.
		
	\end{itemize}
	
	\subsubsection{Results}
	
	\paragraph{Results on $D_{letter}$}
	We randomly select a set of 300 data points from three distinct letter classes: ``A", ``C", and ``F" and try to cluster them into three different classes. For the OOC outlier scenario we choose the letter class ``J" for sampling the outliers. We present the numerical study describing the effect of outlier proportions in \prettyref{fig:letter_multicalss_prop} and \prettyref{fig:letter_onecalss_prop} on $D_{letter}$.
	Both results show that our method consistently yields the lowest proportion of mislabeling in both scenarios (OMC and OOC), outperforming the other two algorithms. 
	Remarkably, our method yields better mislabeling rate even in absence of outliers. As expected, the performance of all three methods deteriorates as the proportion of outlier increases.
	
	\begin{figure}[t]
		\centering
		\begin{minipage}{0.5\textwidth}
			\begin{center}
				{\small Random initialization}\\
				\includegraphics[height=5.5cm]{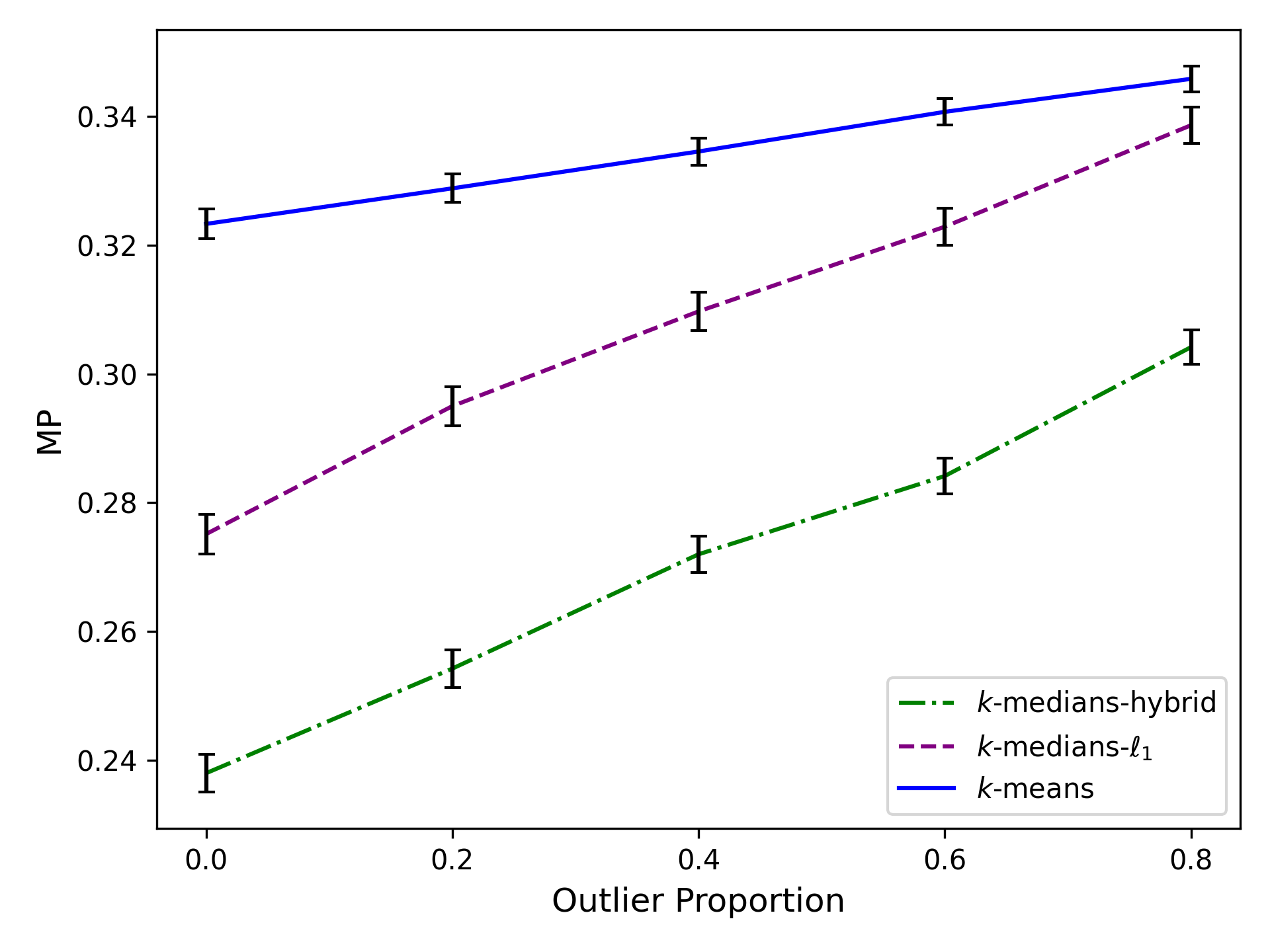}
			\end{center}
		\end{minipage}\hfill
		\begin{minipage}{0.5\textwidth}
			\begin{center}
				{\small Omniscient initialization}\\
				\includegraphics[height=5.5cm]{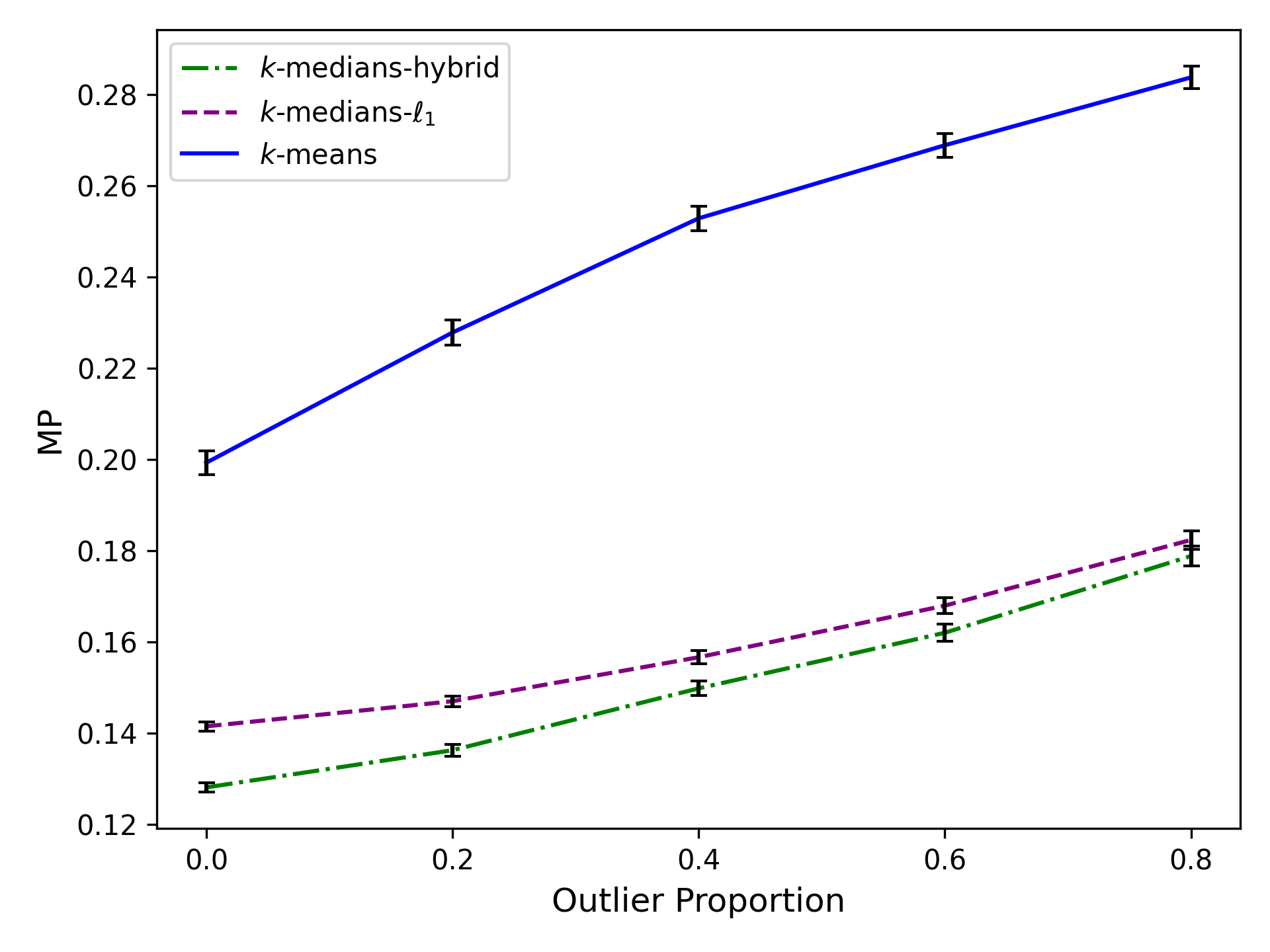}
			\end{center}
		\end{minipage}
		\caption{The effect of outlier proportion on clustering $D_{letter}$ (OMC).}
		\label{fig:letter_multicalss_prop}
	\end{figure}

	\begin{figure}[t]
		\centering
		\begin{minipage}{0.5\textwidth}
			\begin{center}
				{\small Random initialization}\\
				\includegraphics[height=5.5cm]{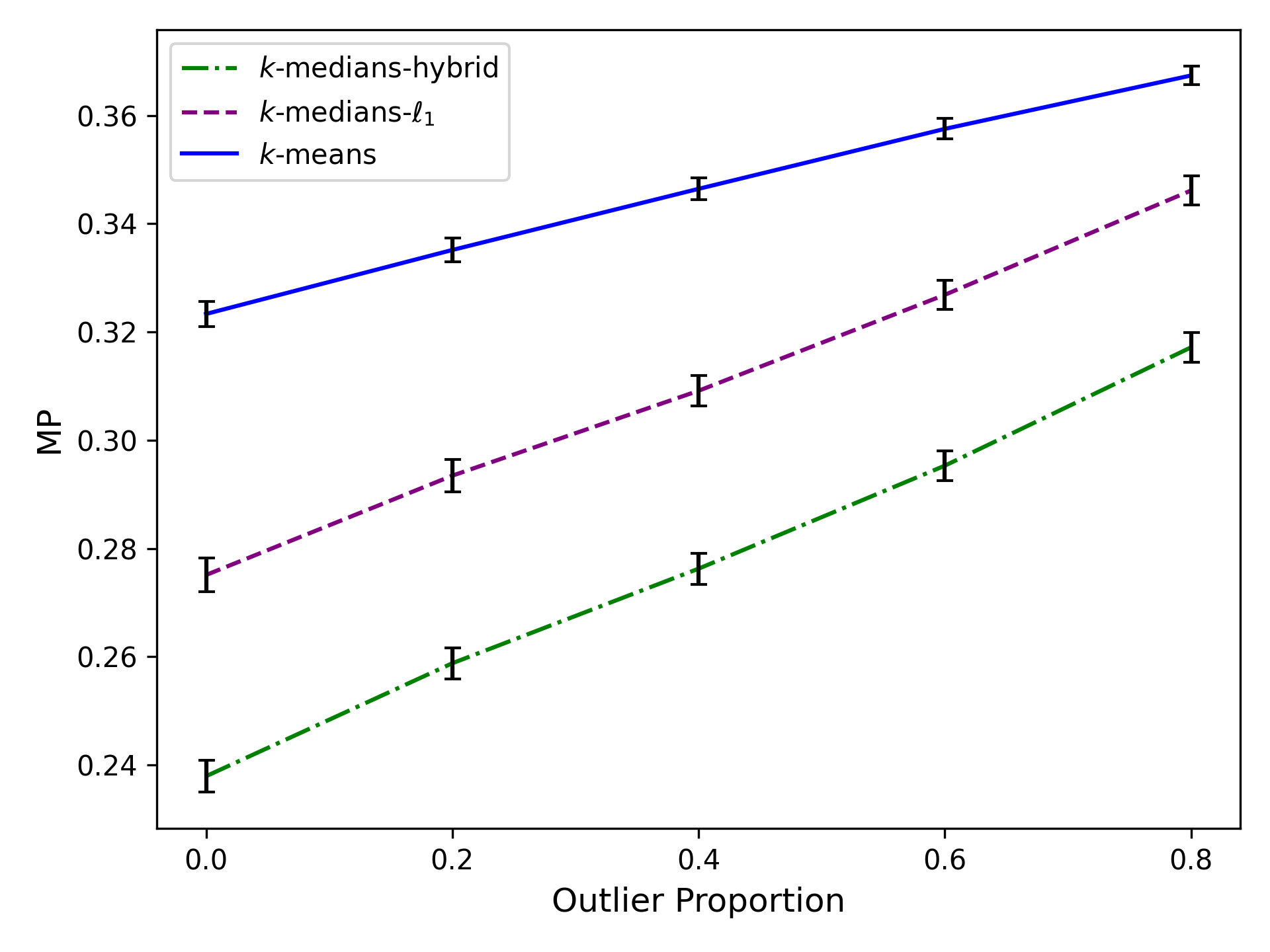}
			\end{center}
		\end{minipage}\hfill
		\begin{minipage}{0.5\textwidth}
			\begin{center}
				{\small Omniscient initialization}\\
				\includegraphics[height=5.5cm]{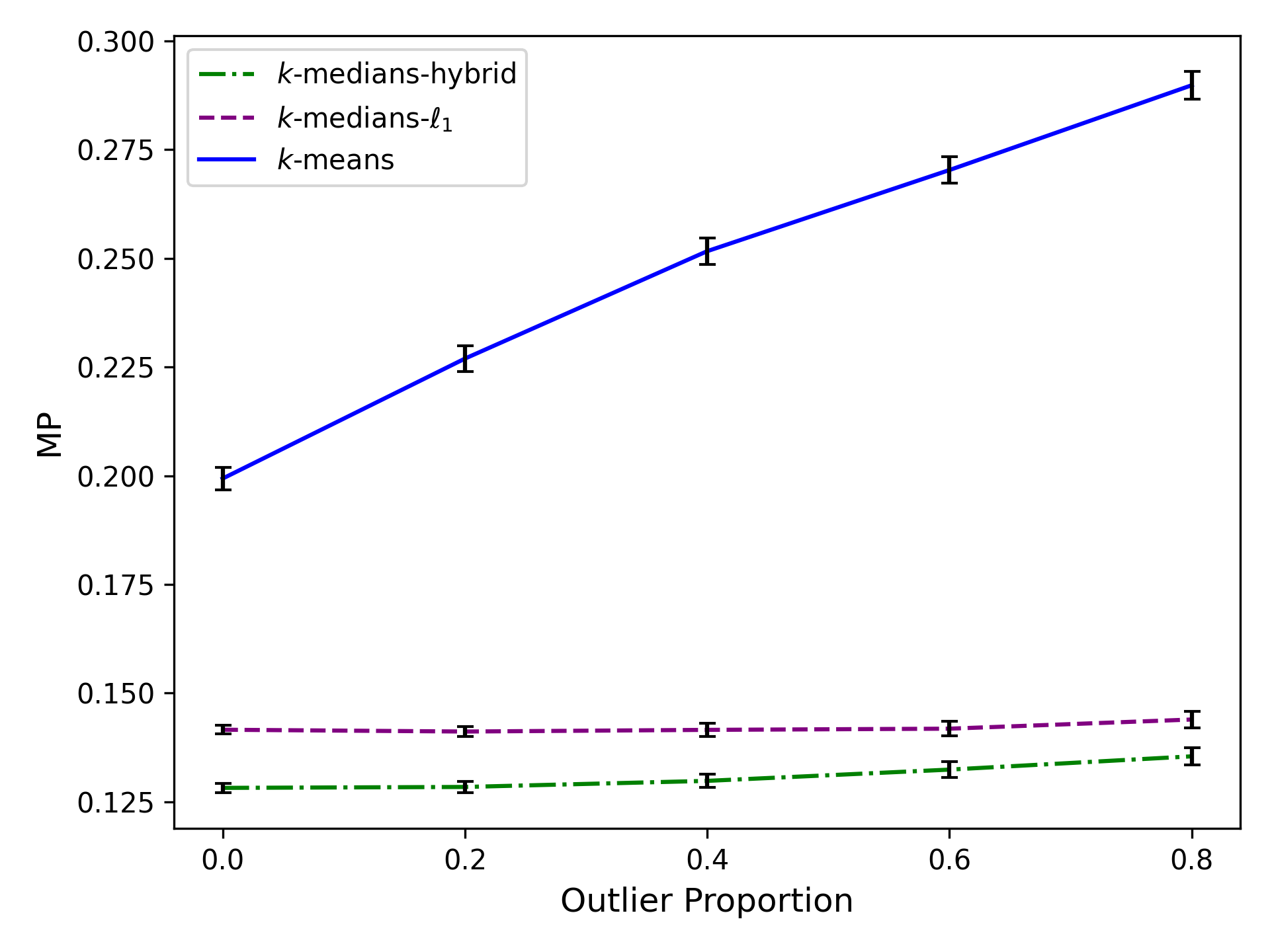}
			\end{center}
		\end{minipage}
		\caption{The effect of outlier proportion on clustering $D_{letter}$ (OOC).}
		\label{fig:letter_onecalss_prop}
	\end{figure}
	
	\paragraph{Results on $D_{digit}$}
	
	As before in $D_{letter}$, we randomly select a set of 300 data points, representing the inliers from three distinct digit classes: ``0", ``2", and ``5". In OMC, the outliers are randomly sampled from the remaining 7 digit classes, while in OOC, the outliers are exclusively drawn from the ``8" class.
	We present the numerical study describing the effect of outlier proportions in \prettyref{fig:pen_multicalss_prop} and \prettyref{fig:pen_onecalss_prop} on $D_{digit}$. The results show that our method yields the lowest proportion of mislabeling in both scenarios (OMC and OOC), outperforming the other two algorithms (except in OMC when the proportion of outliers surpasses 0.7). 
	Comparatively, \kmeans is more sensitive to outliers than other two, which leads it to have the worst performance. 
	As the number of outliers increases, 
	the performance of \kmeans deteriorates much more than the other two as we expected. 
	

	\begin{figure}[t]
		\centering
		\begin{minipage}{0.5\textwidth}
			\begin{center}
				{\small Random initialization}\\
				\includegraphics[height=5.5cm]{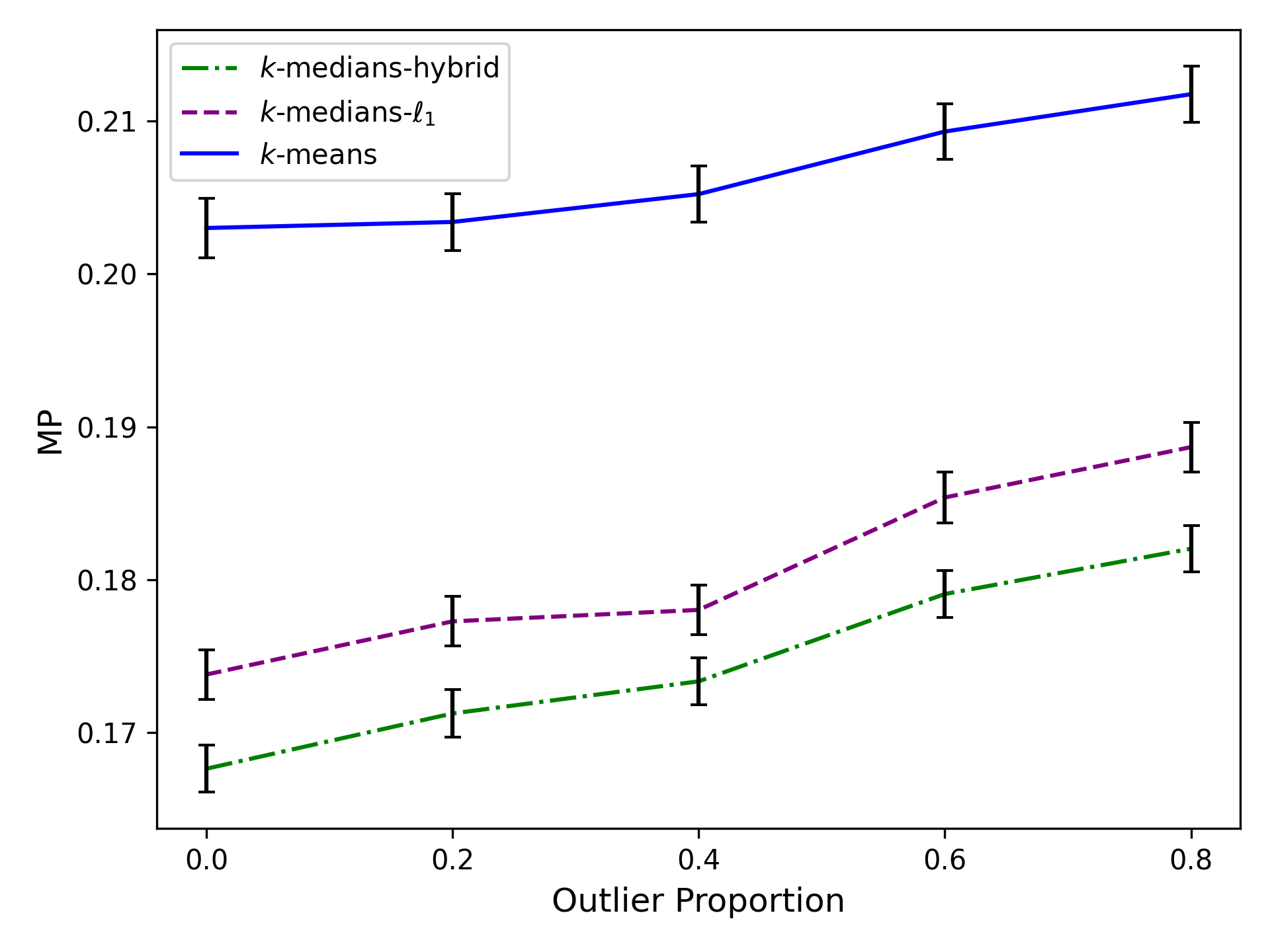}
			\end{center}
		\end{minipage}\hfill
		\begin{minipage}{0.5\textwidth}
			\begin{center}
				{\small Omniscient initialization}\\
				\includegraphics[height=5.5cm]{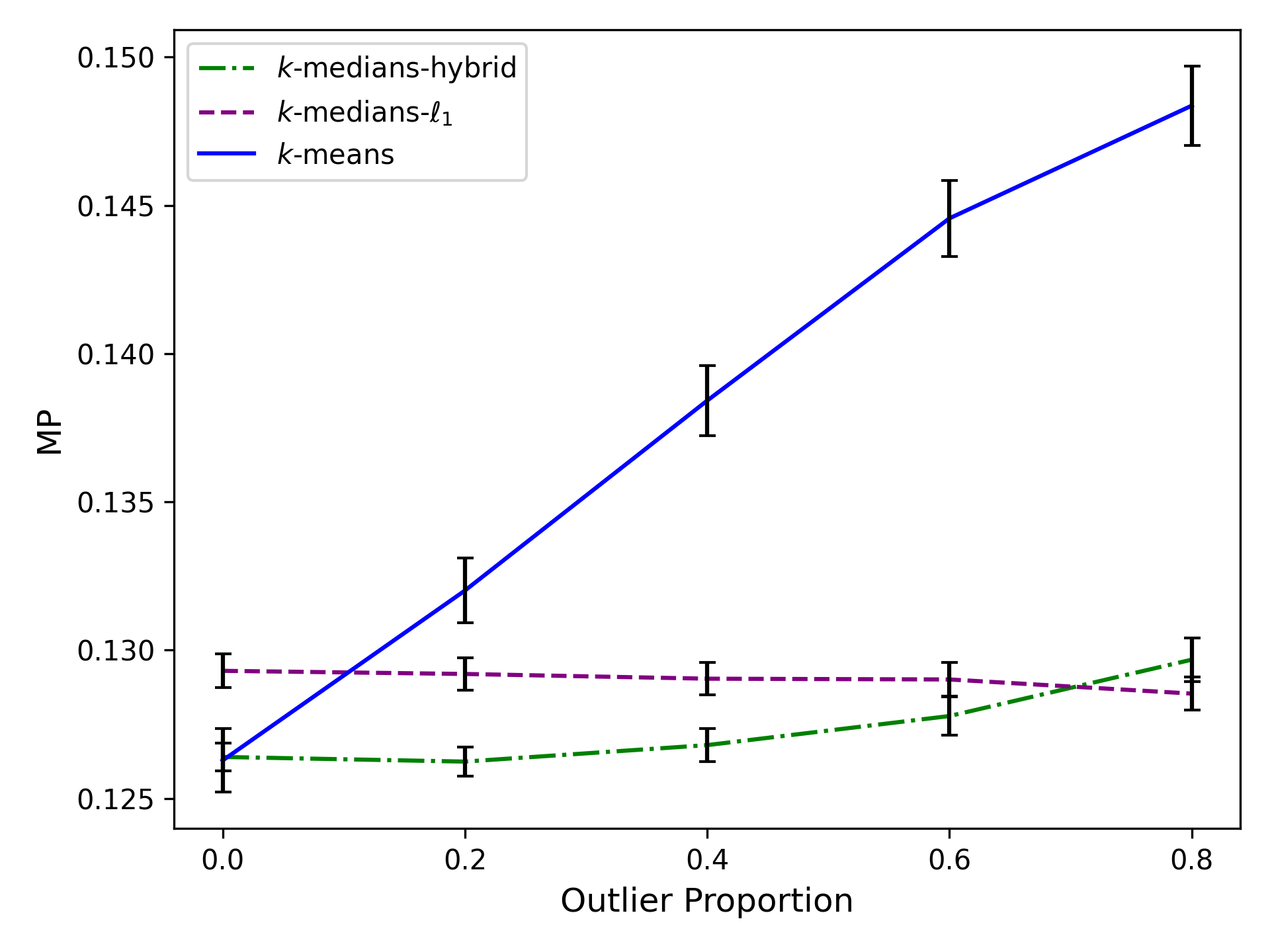}
			\end{center}
		\end{minipage}
		\caption{The effect of outlier proportion on clustering $D_{digit}$ (OMC).}
		\label{fig:pen_multicalss_prop}
	\end{figure}

	\begin{figure}[t]
		\centering
		\begin{minipage}{0.5\textwidth}
			\begin{center}
				{\small Random initialization}\\
				\includegraphics[height=5.5cm]{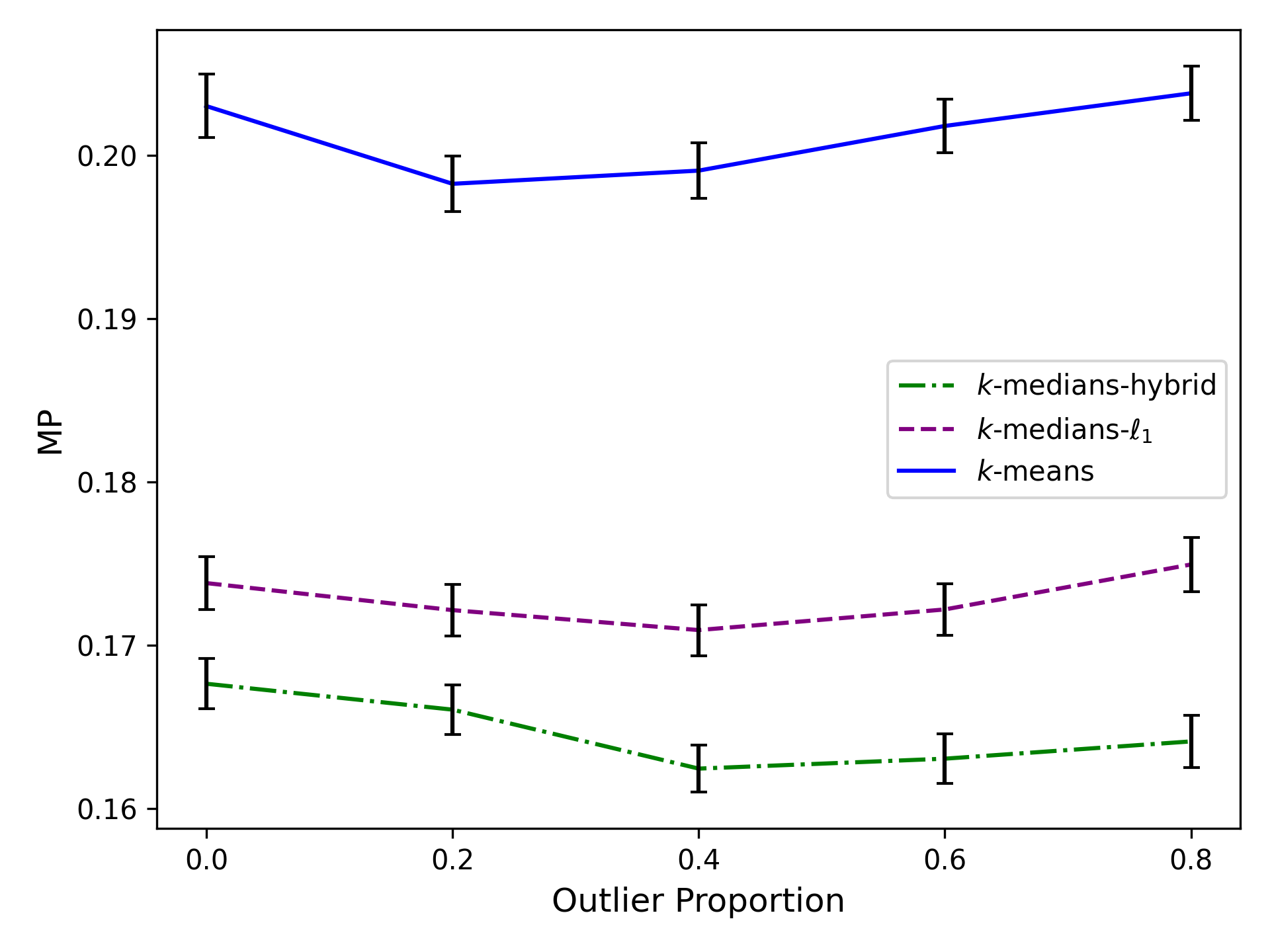}
			\end{center}
		\end{minipage}\hfill
		\begin{minipage}{0.5\textwidth}
			\begin{center}
				{\small Omniscient initialization}\\
				\includegraphics[height=5.5cm]{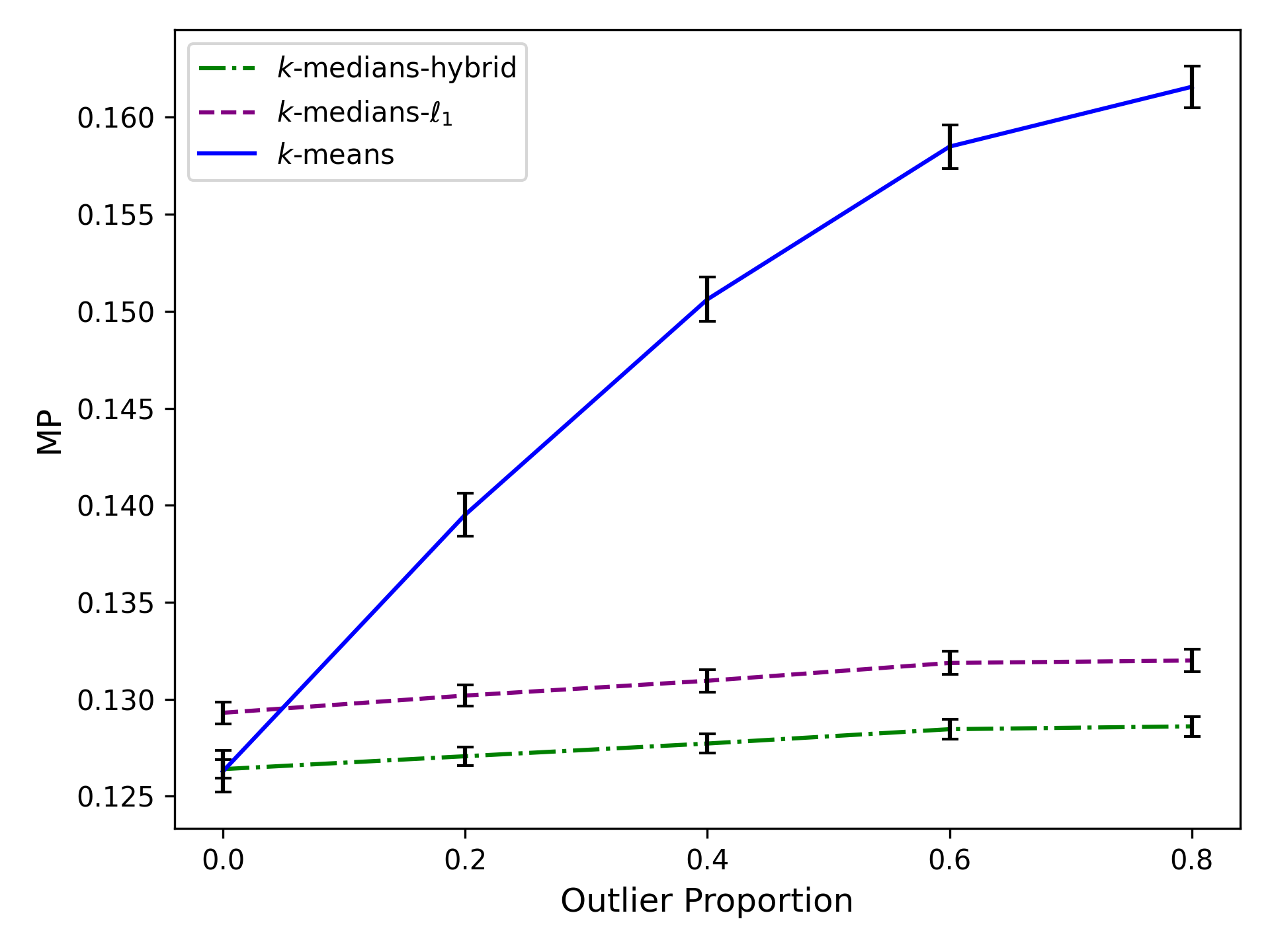}
			\end{center}
		\end{minipage}
		\caption{The effect of outlier proportion on clustering $D_{digit}$ (OOC).}
		\label{fig:pen_onecalss_prop}
	\end{figure}

	%
	%
	
	
	{ \begin{remark}[Comparisons with Robust-LP/SDP based clustering techniques]
			LP/SDP relaxations for robust $k$-means clustering are primarily implemented to obtain a robust low-dimensional projection of the data, on which a standard $k$-means algorithm can be applied. For example, see \cite[Algorithm 1]{srivastava2023robust} and \cite[Algorithm 1]{li2007noise}. For our study, we discuss \cite[Algorithm 1]{srivastava2023robust}, where, after obtaining the top eigenvectors of the data matrix $X$ using a Robust-LP method from Steps 1-3, a standard $k$-means clustering algorithm is applied in Step 4. Step 5 in their algorithm is used to detect the outliers, which we skip as it is beyond the scope of our work. In view of the above, our algorithm can be used to produce a variant of the Robust-LP clustering method, where we substitute the $k$-means step with our \lcomed strategy. Next, we revisited our real data analyses to compare the following options:
			\begin{itemize}
				\item Robust-LP based dimension reduction, followed by our proposed \lcomed algorithm with random initialization
				
				\item Robust-LP based dimension reduction, followed by the Lloyd algorithm ($k$-means) with random initialization
				\item Robust-LP-based dimension reduction, followed by \kmed with random initialization.
			\end{itemize} 
			The second method mentioned above is essentially the entire clustering method in \cite{srivastava2023robust}. We use random initialization for all our choices for comparison purposes, as finding a fast and robust initialization is beyond the scope of our work. We use random initialization to demonstrate the effect of the above modification. We focus on the Letter dataset. To run the Robust-LP based method of \cite{srivastava2023robust}, we chose the parameter $\alpha=0.2$ as prescribed in the paper, and the other parameter $\beta=0.2$ is chosen to achieve the best performance improvement compared to our previous studies without dimension reduction. Our experiment setup studies how the mislabeling errors of all the above algorithms change as we increase the proportion of outliers in the data. We repeated all the experiments 5000 times and recorded the average mislabeling proportion and  95\% confidence bands \prettyref{fig:rev-robust_LP-Letter}.
			
			Notably, most of the performances improved compared to our previous studies, which used random initialization. However, the mislabeling proportions are still significantly higher compared to our earlier studies with omniscient initialization; see \prettyref{fig:letter_multicalss_prop} and \prettyref{fig:letter_onecalss_prop}. In the OOC setup, we observe that the \lcomed based method starts to outperform the other options as the mislabeling proportion increases. The above observations suggest that robust dimension reduction indeed facilitates clustering; however, the optimal algorithm will depend on how accurately and robustly the clustering methods are initialized. We leave it for future work to perform more in-depth studies. In terms of computation time, the Robust-LP-based variants are extremely slow, which makes them infeasible for large datasets, as mentioned in \cite[Section 5]{srivastava2023robust}. In addition to solving linear programming, the Robust-LP-based dimension reduction involves computing a kernel based on mutual distances among all the data points, which takes $dn^2$ steps. The Robust-LP method requires tuning the $\beta$ parameter, which adds to excess computational cost. Compared to the above, the \lcomed method, without any Robust-LP modifications, requires $O(dn(k+\log n))$ time to reach the theoretical mislabeling limit, which is almost linear in the sample size. See \prettyref{rmk:runtime} for the details. 
			
			\begin{figure}[t]
				\centering
				\begin{minipage}{0.5\textwidth}
					\begin{center}
						{\small Outcome in the OMC setup}\\
						\includegraphics[height=5.5cm]{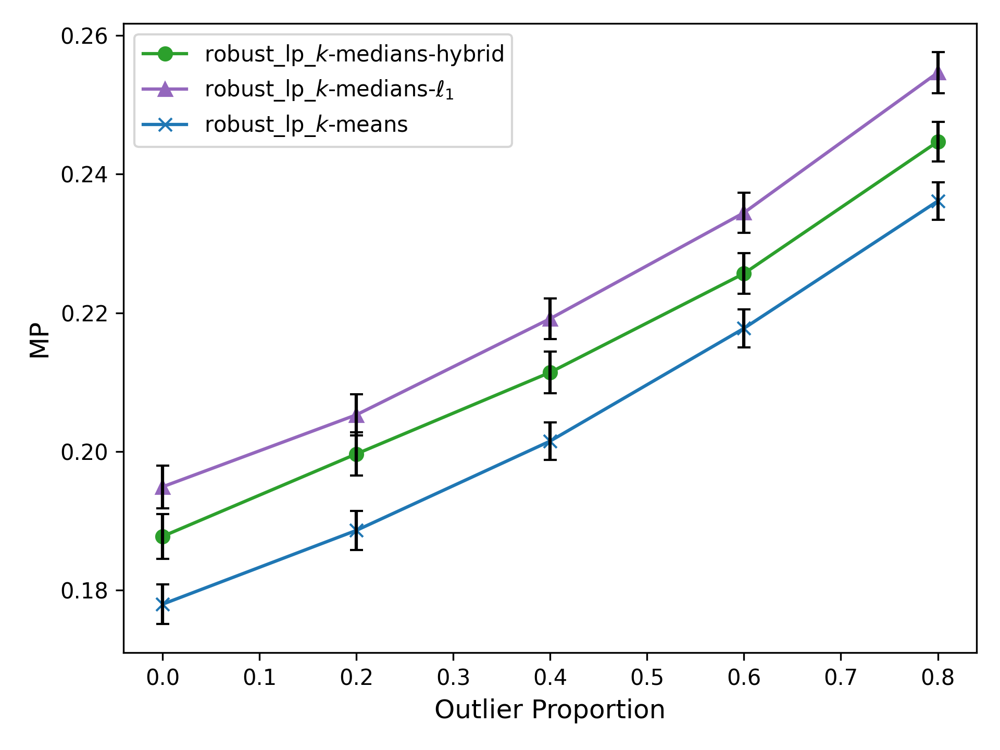}
					\end{center}
				\end{minipage}\hfill
				\begin{minipage}{0.5\textwidth}
					\begin{center}
						{\small Outcome in the OOC setup}\\
						\includegraphics[height=5.5cm]{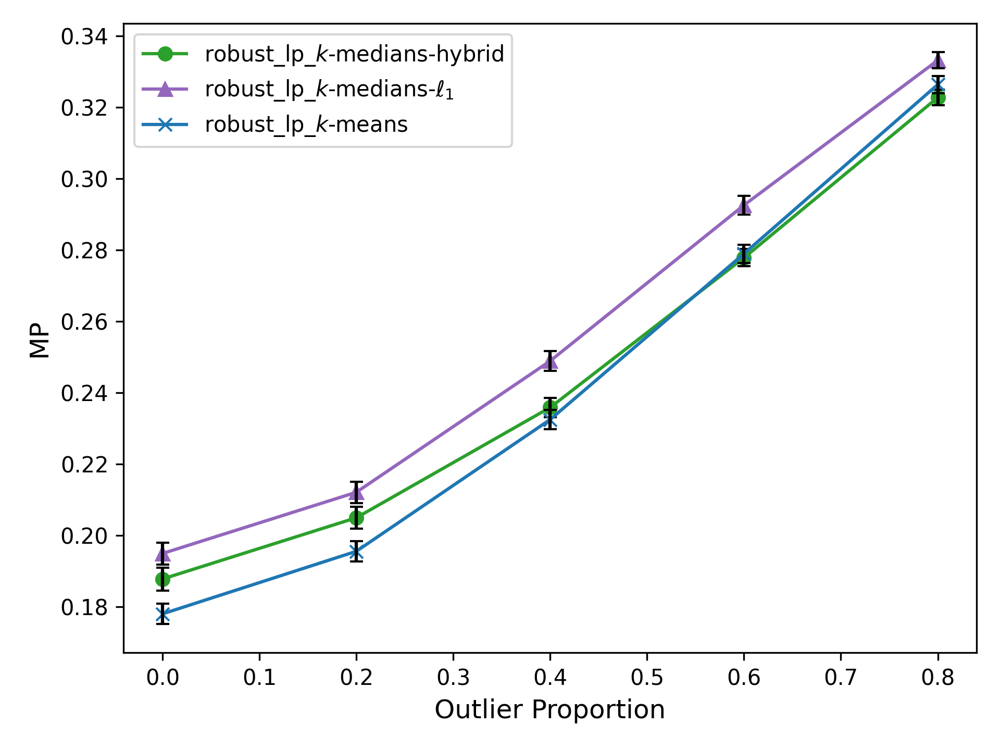}
					\end{center}
				\end{minipage}
				\caption{Comparison for randomly initialized, LP-based methods.}
				\label{fig:rev-robust_LP-Letter}
			\end{figure}


			We provide comments below on the comparison of theoretical guarantees for the Robust-LP method presented above. Note that \cite[Theorem 1]{srivastava2023robust} shows that the mislabeling guarantee provided by their method is approximately $\exp\sth{-{\Delta^2\over 64 \sigma^2}}$. The above mislabeling rate is suboptimal compared to the optimal rate $\exp\sth{-{\Delta^2\over 8 \sigma^2}}$, which may be due to the analysis of the LP-based dimension reduction step or the application of the $k$-means method, which lacks robustness properties. Nonetheless, assuming the signal-to-noise ratio is sufficiently large, we postulate that in the absence of outliers, our iterative algorithm, combined with a spectral-type low-dimensional projection method, should be able to guarantee an excellent mislabeling rate. 
			This is because it is possible to obtain Spectral type projections that guarantee the clusters are well-separated, and the projected data follow sub-Gaussian distributions (see, e.g., \cite{loffler2021optimality} for the guarantees of Projections along the top eigenvectors of the sample covariance matrix). Hence, the subsequent \lcomed algorithm with a good initialization should guarantee exponentially small mislabeling errors as in our theoretical results. In the presence of outliers, it might be possible to follow the argument of \cite{srivastava2023robust} and establish a similar mislabeling guarantee for the LP/SDP-based dimension reduction + \lcomed algorithm. However, this is beyond the scope of the current work.
		\end{remark}

		\section{Proof sketch of the main results}
		\label{sec:sketch}
		
		\subsection{Proof sketch of \prettyref{thm:ell1-subopt}}
		\label{sec:sketch-lower-bound}
		
		For proving the lower bound in \prettyref{thm:ell1-subopt} we assume that we know the true centroids and cluster the points using the $\ell_1$ distance. Note that for this specific result we have assumed that $\{w_i\}_{i=1}^n$ are independent $\calN(0,\sigma^2 I_{d\times d})$ random variables.
		Suppose that the parameter space in dimension $d$, with number of clusters $k$ known, is given by
		\begin{align}
			\Theta = \Big\{\theta:\ \theta=[\theta_1,\dots,\theta_k]\in \reals^{d\times k}, \Delta\leq \min_{g\neq h} \|\theta_g-\theta_h\|_2 \Big\}.
		\end{align}
		Given any vector $v\in \reals^d$ we define the $\ell_0,\ell_1$ and $\ell_2$ norms of $v$ respectively as
		\begin{align}
			\|v\|_1 = \sum_{i\in [d]}|v_i|,\
			\|v\|_2 = (\sum_{i\in [d]}v_i^2)^{1/2}.
		\end{align}
		Below we provide a proof sketch \prettyref{thm:ell1-subopt} (ii) for the two cluster setup. Suppose that we have estimates $\hat\theta_1,\hat\theta_2$ of the centroids, and we cluster the points using the $\ell_1$ distance. 
		Let $v=\theta_1-\theta_2, u_1=\theta_1-\hat \theta_1,u_2=\theta_2-\hat \theta_2$. Note that in view of our assumptions we have $\|u_1\|_2,\|u_2\|_2\leq C_0\sigma$. We will choose $\theta_1,\theta_2$ in such a way that $|v_j|>2\sigma C_0$ for all $j=1,\dots,d$.
		We observe that in this special case of two centroids we have
		\begin{equation*}
			\begin{aligned}
				&\PP\qth{\hat z_i=2,z_i=1} = \PP\qth{Y_i=\theta_1+w_i,\ \|Y_i-\hat\theta_1\|_1\geq \|Y_i-\hat\theta_2\|_1}
			\\
			&=\PP\qth{\|w_i+u_1\|_1\geq \|w_i+v+u_2\|_1},
			\\
			&\PP\qth{\hat z_i=1,z_i=2} = \PP\qth{Y_i=\theta_2+w_i,\ \|Y_i-\hat\theta_2\|_1\geq \|Y_i-\hat\theta_1\|_1}
			\\
			&=\PP\qth{\|w_i+u_2\|_1\geq \|w_i+v+u_1\|_1}.
			\end{aligned}
		\end{equation*}
		Hence, we can bound the expected error of mislabeling as
		\begin{equation*}
			\begin{aligned}
			\EE\qth{\ell(\hat z,z)}
			\geq 
			\frac 1n\sum_{i=1}^n\min\{
			&\PP\qth{\|w_i+u_1\|_1\geq \|w_i+v+u_2\|_1},
			\\
			&\PP\qth{\|w_i+u_2\|_1\geq \|w_i+v+u_1\|_1}\}.
		\end{aligned}
		\end{equation*}
		Without a loss of generality, let us consider the probability $\PP\qth{\|g+u_2\|_1\geq \|g+v+u_1\|_1}$, where $g\sim N(0,\sigma^2I_{d\times d})$ and $u_1,u_2$ can depend on $g$ with $\|u_1\|_2,\|u_2\|_2\leq C_0\sigma$. Then it suffices to show that for the above probability we achieve the desired lower bound. The rest of the proof can be divided in the following steps:
		\begin{itemize}
			\item \textbf{Selecting appropriate $\theta_1,\theta_2$:}
			Without a loss of generality let $d=2m,m\geq 1$. We choose $\theta_1,\theta_2$ as 
			\begin{equation}\label{eq:m1}
				\begin{gathered}
					\theta_1=\{v_{j}\}_{j=1}^{2m},\quad
					v_{2j-1} =\Delta{(1+\delta)\over \sqrt{d(1+\delta^2)}},\\
					v_{2j} =\Delta{(1-\delta)\over \sqrt{d(1+\delta^2)}},
					 j=1,\dots,d/2,
					\quad \theta_2 = (0,\dots,0),
				\end{gathered},
			\end{equation}
			with $\delta\in (0,1/2)$ is such that $2 > \frac C2+1 = 8\delta^2 > C$.
			
			\item \textbf{Obtaining bounds in terms of the distribution of $g$:} We show that
			\begin{equation*}
				\begin{aligned}
					&\PP\qth{\|g+u_2\|_1\geq \|g+v+u_1\|_1}
				\\
				&\geq  \PP\qth{g_j\geq {\|v\|_1\over 2 d}+2\max\{C_0,1\}\sigma,\ j=1,\dots, d}.
				\end{aligned}
			\end{equation*}

			\item  \textbf{Conclude using Gaussian tail bounds:} Using a Gaussian tail bound we show that whenever $\Delta$ is significantly larger than $\sigma\sqrt d$ we get
			\begin{align*}
				\PP\qth{g_j\geq {\|v\|_1\over 2 d}+2\max\{C_0,1\}\sigma} 
				\gtrsim \exp\sth{-{\Delta^2\over (8+C)d\sigma^2}},
			\end{align*}
			which leads us to the required bound
			\begin{equation*}
				\begin{aligned}
					&\PP\qth{\|g+u_2\|_1\geq \|g+v+u_1\|_1}
				\\
				&\geq 
				\prod_{j=1}^d
				\sth{\PP\qth{g_j\geq {\|v\|_1\over 2 d}+2\max\{C_0,1\}\sigma} }
				\\
				&\gtrsim 
				\exp\sth{-{\Delta^2\over (8+C)\sigma^2}}.
				\end{aligned}
			\end{equation*}
		\end{itemize}
		The idea of proving the result with a general number of clusters follows similarly, and is provided in \prettyref{app:l1-mislabel}.

		\subsection{Proof sketch of \prettyref{thm:main}}
		
		The proof of mislabeling guarantee for the \lcomed algorithm relies on the following steps: 
		\begin{enumerate}[label=(\alph*)]
			\item analyzing the accuracy of the clustering method based on current center estimates,
			\item analysis of the next center updates based on current labels.
		\end{enumerate} 
		The above steps are captured in the following lemma that bounds the mislabeling error based on the centroid estimation error in the previous step and conversely bounds the centroid estimation error based on the mislabeling proportion in the last step. This is the essence of the proof of \prettyref{thm:main}.
		\begin{lemma}\label{lmm:Hs-Ls-bound}
			Fix $\epsilon_0\in (0,\frac 12]$ and $\gamma_0\in ({10\over n\alpha},\frac 12]$. Suppose that an adversary added at most $n\alpha(1-\delta)$ many outliers for some $\delta>0$ and $n\alpha \geq c\log n$. Then there is an event $\calE_{\gamma_0,\epsilon_0}$ with
			$$\PP\qth{\calE_{\gamma_0,\epsilon_0}}\geq 1-2(k^2+k)n^{-c/4}-8dke^{-0.3n}$$
			on which the following holds:
			\begin{enumerate}[label=(\roman*)]
				\item If $\Lambda_s\leq \frac 12-\epsilon_0$ then $H_{s+1} \geq \frac 12 +\min\sth{{\delta-2\tau\over 2(2-\delta)},\frac 12-\tau}$ where $\tau={17\over 2\epsilon_0^2 (\snr\sqrt{\alpha/ (1+{dk/n})})^2 }$,
				\item If $H_s\geq \frac 12+\gamma_0$ then $\Lambda_s\leq {8\sqrt 3\over \Delta/\sigma}\sqrt{d \over  {\gamma_0\alpha}}$.
			\end{enumerate}
		\end{lemma}
		Establishing the above is arguably the most challenging part of the entire proof, as we need to examine how the presence of adversarial outliers affects clustering outcomes at each iteration. We use the above result to show that as long as the assumptions in the theorem statement are satisfied, with a high probability, we have
		\begin{align*}
			\Lambda_s \leq 0.3,\quad H_{s+1}\geq \frac 12 +\frac \delta6 \text{ for all } s\ge 1.
		\end{align*}
		Based on the above, in the next step we show that the probability with which our algorithm incorrectly assigns label $h$ to any point $X_i$ from cluster $g\neq h$, can be bounded as
		\begin{equation*}
			\begin{aligned}
			\PP\qth{z_i=g,\hat z_i^{(s+1)}=h}
			&\leq \PP\qth{\beta\|\theta_g-\theta_h\|^2
			\leq 2\langle w_i,\theta_h-\theta_g \rangle}
			\\ 
			& +\PP\qth{\|w_i\|\geq {\Delta\over 4\beta_1}},
			\end{aligned}
		\end{equation*}
		for suitably chosen parameters $\beta,\beta_1$. Then we use properties of the sub-Gaussian random variable $w_i$ to achieve the desired bound that leads to the required mislabeling guarantees. See \prettyref{app:main-proof-mislabeling} for the details.
		
		\subsection{Proof sketch of \prettyref{thm:centroid}}
		Fix $g\in [k]$. Let $U_j$ denote the set of real numbers consisting of the $j^{\text{th}}$ coordinates of $\sth{Y_i:i\in T_g^{(s)}}$, and $W_j$ denote the set of real numbers consisting of the $j^{\text{th}}$ coordinates of $\sth{Y_i:i\in T_g^*}$. Then we show that there exists $C:=C(\tau)$ such that whenever $\sqrt \delta\cdot \snr\sqrt{\alpha/d}\geq C(\tau)$, with a high probability and with $\beta_g^\out = {n^{\out}\over n_g^*}, g=1,2,\dots,k$, we have
		\begin{equation*}
			\begin{aligned}
				&W_j^{(\ceil{{n_g^*\pth{\frac 12+e^{-{\Delta^2\over (8+\tau)\sigma^2}}}}})}
			\\
			&
			\leq U_j^{(\floor{{n_g^{(s)}/2}})}
			\leq W_j^{(\floor{{n_g^*\pth{\frac 12-{\beta_g^{\out}\over 1+\beta_g^{\out}}-e^{-{\Delta^2\over (8+\tau)\sigma^2}}}}})}.
			\end{aligned}
		\end{equation*}
		In other words, each coordinate of the coordinatewise median of the estimated cluster is bounded from above and below by a perturbation of the corresponding coordinate value of the coordinatewise median of the true cluster. Using the fact that, in each coordinate, $\med(\sth{Y_i:i\in T_g^{*}})$ is close to $\theta_g$ (see \prettyref{lmm:order-concentration-smallx}) we obtain the desired result. See \prettyref{app:proof-centroid} for the details.}

	\appendices
	
	\section{Proof of \prettyref{thm:ell1-subopt}}
	\label{app:l1-mislabel}
	
	\noindent \textbf{Base case: Number of clusters is two}.
	In view of the proof sketch in \prettyref{sec:sketch-lower-bound} we only bound $\PP\qth{\|g+u_2\|_1\geq \|g+v+u_1\|_1}$, where $g\sim N(0,\sigma^2I_{d\times d})$ and $u_1,u_2$ can depend on $g$ with $\|u_1\|_2,\|u_2\|_2\leq C_0\sigma$. Then we have
	\begin{align}
		&~\PP\qth{\|g+u_2\|_1\geq \|g+v+u_1\|_1}
		\nonumber\\
		&\stepa{=} \PP\Biggl[\sum_{j=1}^d 2\min\{|g_j+u_{2,j}|,|v_j+u_{1,j}-u_{2,j}|\}\indc{\sign(g_j+u_{2,j})\neq\atop\sign(v_j+u_{1,j}-u_{2,j})}
		\nonumber \\
		&\quad \quad \quad\geq \|v+u_1-u_2\|_1\Biggr]
		\label{eq:lone4}\\
		&\stepb{\geq} \PP\qth{|g_j|\geq {|v_j| \over 2} +\frac 32 C_0\sigma,\ \sign(g_j)\neq\sign(v_j)\ j=1,\dots, d}
		\nonumber\\
		&\stepc{\geq} \PP\qth{g_j\geq {|v_j| \over 2} +\frac 32\max\{C_0,1\}\sigma,\ j=1,\dots, d}
		\label{eq:lone5}
	\end{align}
	where the justification for steps (a), (b) and (c) are given below.
	Step (a) follows using
	\begin{align*}
		&~\|g+v+u_1\|_1
		=\sum_{j=1}^d |g_j+v_j+u_{1,j}|
		\nonumber\\
		&= \sum_{j=1}^d \Big(|g_j+u_{2,j}|+|v_j+u_{1,j}-u_{2,j}|
		 -2\min\{|g_j+u_{2,j}|,
		 \\
		 &\quad \quad \quad|v_j+u_{1,j}-u_{2,j}|\}\indc{\sign(g_j+u_{2,j})\neq\atop\sign(v_j+u_{1,j}-u_{2,j})}\Big)
		\nonumber\\
		&= \|g+u_2\|_1 + \|v+u_1-u_2\|_1
		-\sum_{j=1}^d 2\min\{|g_j+u_{2,j}|,
		\\
		&\quad \quad \quad  |v_j+u_{1,j}-u_{2,j}|\}\indc{\sign(g_j+u_{2,j})\neq\atop\sign(v_j+u_{1,j}-u_{2,j})}.
	\end{align*}
	Step (b) follows as $\sup_{i\in\{1,2\}}\sup_{j\in [d]}|u_{i,j}|\leq C_0\sigma$ and $|v_j|>2\sigma C_0$ for all $j=1,\dots,d$. Step (c) follows since $\PP\qth{|g_j|\geq a, \sign(g_j)=s}=\PP\qth{g_j\geq a}$ for $a\geq 0$. Then we can bound the final probability term in \eqref{eq:lone5} using the following tail bound for $z\sim N(0,\sigma^2)$ \cite[Section A.4]{lu2016statistical}: for $t\geq \sigma$,
	\begin{equation}
		\label{eq:normal-tail}
		\begin{aligned}
			&\PP\qth{z\geq t} 
		= \frac 1{\sqrt {2\pi} \sigma}
		\int_{t}^\infty e^{-{x^2\over 2\sigma^2}}dx
		\\ &
		\geq \frac 1{\sqrt {2\pi}}{t/\sigma\over \pth{t\over \sigma}^2+1}e^{-{t^2\over 2\sigma^2}}
		\geq \frac 1{\sqrt {2\pi}}{\sigma\over 2t}e^{-{t^2\over 2\sigma^2}}.
		\end{aligned}
	\end{equation}
	Using the above we get
	\begin{equation*}
		\begin{aligned}
		&\PP\qth{g_j\geq \frac 12 |v_j| + \frac 32 \max\{C_0,1\}\sigma} 
		\\& 
		\geq {\sigma \cdot \exp\sth{-{(|v_j| +3\max\{C_0,1\}\sigma)^2\over 8\sigma^2}}
			\over ( |v_j| + 3\max\{C_0,1\}\sigma) \sqrt{2\pi} }
		\nonumber\\
		&\geq 
		e^{-\frac 1{8\sigma^2}\{(|v_j| +3\max\{C_0,1\}\sigma)^2+C_2\sigma|v_j|+C_3\sigma^2\}}.
		\end{aligned}
	\end{equation*}
	Then continuing \eqref{eq:lone5} with the independence of $g_i$-s we get
	\begin{equation*}
		\begin{aligned}
			&~\PP\qth{\|g+u_2\|_1\geq \|g+v+u_1\|_1}
		\\
		&\geq 
		\exp\sth{-\frac 1{8\sigma^2}\{\sum_{j=1}^d|v_j|^2+C_4\sigma
			\sum_{j=1}^d|v_j|+C_5d\sigma^2\}}
		\\
		&\geq 
		\exp\sth{-\frac 1{8\sigma^2}\{\Delta^2+C_4\sigma
			\Delta\sqrt d+C_5d\sigma^2\}}
		\end{aligned}
	\end{equation*}
	where the last inequality follows using $\sum_{j=1}^d|v_j|^2=\Delta^2$ and from Cauchy-Schwarz inequality we have $\sum_{j=1}^d |v_j|\leq \sqrt d \|v\|_2.$
	This finishes the proof of part (i).

	To prove the worst case bound in part (ii), without loss of generality and for simplicity of notation, let $d$ be an even number. We choose $\theta_1,\theta_2$ as in \eqref{eq:m1}. This implies that 
	\begin{align}
		\label{eq:lone6}
		\begin{gathered}
			v=\theta_1-\theta_2,\quad  \| v \|_1={\Delta\sqrt d\over \sqrt{1+\delta^2}},\quad \|v\|_2=\Delta,\\
			\min_j |v_j|=\Delta{(1-\delta)\over \sqrt{d(1+\delta^2)}}
			= {\|v\|_1\over 2d}+{\Delta\over \sqrt{d(1+\delta^2)}}{\pth{\frac 12-\delta}}.
		\end{gathered}
	\end{align}
	Then depending on $C,C_0$, we can pick $c_0>0$ large enough such that $\Delta>c_0\sigma\sqrt d$ ensures
	\begin{align}\label{eq:lone2}
		\min_j |v_j|\geq {\|v\|_1\over 2d}+4C_0\sigma.
	\end{align}
	Then we continue to bound \eqref{eq:lone4} as
	\begin{align}
		&~\PP\qth{\|g+u_2\|_1\geq \|g+v+u_1\|_1}
		\nonumber\\
		&\geq \PP\Biggl[\sum_{j=1}^d \min\{|g_j|-|u_{2,j}|,|v_j|-|u_{1,j}|-|u_{2,j}|\}\\
		&\quad \quad \quad \indc{\sign(g_j+u_{2,j})\neq\sign(v_j+u_{1,j}-u_{2,j})}\geq {\|v\|_1\over 2d}+C_0\sigma\Biggr]
		\nonumber\\
		&\stepa{\geq} \PP\Biggl[|g_j|-|{u_{2,j}}|\geq {\|v\|_1\over 2d}+C_0\sigma,
			\nonumber\\
			& \sign(g_j)\neq\sign(v_j),j\in [d]\Biggr]
		\nonumber\\
		&\stepb{\geq} \PP\qth{g_j\geq {\|v\|_1\over 2 d}+2\max\{C_0,1\}\sigma, j\in [d]}.
		\label{eq:lone3}
	\end{align}
	where (a) follows from the fact that \eqref{eq:lone2} implies 
	$$\begin{gathered}
		|v_j|-|u_{1,j}|-|u_{2,j}|\geq {\|v\|_1\over 2d}+C_0\sigma,\\ \sign(v_j+u_{1,j}-u_{2,j})=\sign(v_j),
	\end{gathered}$$
	and $\sign(g_j+u_{2,j})=\sign(g_j)$ when $|g_j|-|{u_{2,j}}|$ is positive, and (b) follows using $\max_{j\in [d]}|u_{2,j}|\leq C_0 \sigma$ and since $\PP\qth{|g_j|\geq a, \sign(g_j)=s}=\PP\qth{g_j\geq a}$ for $a\geq 0$ and any $s$. For simplicity of notation, let $y={\Delta\over \sqrt{(1+\delta^2)}}+4\max\{C_0,1\}\sigma\sqrt d$. Then we can bound the final probability term in \eqref{eq:lone3} using \eqref{eq:lone6} and \eqref{eq:normal-tail} as
	\begin{equation*}
		\begin{aligned}
			&~\PP\qth{g_j\geq {\|v\|_1\over 2 d}+2\max\{C_0,1\}\sigma} 
		\\
		&=
		\PP\qth{g_j\geq {y\over 2\sqrt d}} 
		\geq {4\sigma\sqrt d\over y\sqrt{2\pi} }\exp\sth{-{y^2\over 8d\sigma^2}}
		\end{aligned}
	\end{equation*}
	where the last inequality follows using \eqref{eq:normal-tail}.
	Given $C,C_0>0$, as $8\delta^2=\frac C2 +1>C$, we can pick $c_1>0$ large enough such that, whenever $\Delta\geq c_1\sigma\sqrt d$, the right most term in the above display is at least $\exp\sth{-{y^2\over (8+C)d\sigma^2}}$. In view of \eqref{eq:lone3} we get 
	\begin{equation*}
		\begin{aligned}
			&~\PP\qth{\|g+u_2\|_1\geq \|g+v+u_1\|_1}
		\\
		&\geq 
		\prod_{j=1}^d
		\sth{\PP\qth{g_j\geq {\|v\|_1\over 2 d}+2\max\{C_0,1\}\sigma} }
		\\
		&\geq 
		\exp\sth{-{y^2\over (8+C)\sigma^2}}.
		\end{aligned}
	\end{equation*}
	This finishes our proof. 
	
	\noindent \textbf{General number of centroids.} For general $k\geq 2$, we pick $\theta_1,\theta_2$ to be the centroids that are exactly $\Delta$ distance away and chosen according to the previous two centroid case. As there are at least $n\alpha$ points in the cluster $T_1$, we get
	$$
	\frac 1n\sum_{i=1}^n\PP\qth{\hat z_i\neq z_i}
	\geq  \frac 1n\sum_{i\in T_1}\PP\qth{\hat z_i\neq z_i}
	\geq \alpha \min_{i\in T_1}\PP\qth{\hat z_i\neq 1}.
	$$
	As we are labeling the points using the $\ell_1$ metric, for any $i\in T_1$, $\|Y_i-\hat \theta_1\|_1\geq \|Y_i-\hat \theta_2\|_1$ is a sufficient criteria to have $\hat z_i\neq 1$. This implies we have
	$$
	\frac 1n\sum_{i=1}^n\PP\qth{\hat z_i\neq z_i}
	\geq \alpha \min_{i\in T_1}\PP\qth{\|Y_i-\hat \theta_1\|_1\geq \|Y_i-\hat \theta_2\|_1}
	$$
	Then we can replicate the analysis for the $k=2$ case, with an appropriate $\delta>0$, to achieve $\alpha e^{-{\Delta^2\over (8+\tilde C)\sigma^2}}$ lower bound with $\tilde C> C$. As $\Delta>> \sigma\sqrt{\log(1/\alpha)}$, we achieve the desired lower bound.

	\section{Proof of \prettyref{thm:main} and \prettyref{lmm:Hs-Ls-bound}}
	\label{app:main-proof-mislabeling}
	
	\begin{proof}[Proof of \prettyref{thm:main}]
		For a simplicity of notations, we will use $\|\cdot\|$ and $\|\cdot\|_2$ interchangeably, unless specified otherwise, in all the proofs in this section. For $c_1,c_2$ to be chosen later, we define
		$$\gamma_0=\frac {c_1}{\pth{\snr\sqrt{\alpha\over d}}^2},\quad \epsilon_0={c_2\over {\sqrt \delta \cdot \snr\sqrt{\alpha\over 1+dk/n}}}.$$
		Then from \prettyref{lmm:Hs-Ls-bound} it follows that \footnote{Note that \prettyref{lmm:Hs-Ls-bound} required $\gamma_0\geq {10\over n\alpha}$, which translates to the requirement $\snr<\sqrt{nd}$. This is good enough for us as the target mislabeling bound of $\exp(-\Theta((\snr)^2))$ becomes trivial for $\snr>\Omega(\sqrt{\log n})$.}, we can choose $c_1,c_2,c_3,c_4>0$ such that whenever $\snr\sqrt{\alpha/d}>c_3$ and $\sqrt{\delta}\cdot \snr\sqrt{\alpha/d} \geq c_4$, on the high probability event $\calE_{\gamma_0,\epsilon_0}$ defined from \prettyref{lmm:Hs-Ls-bound} we have
		\begin{itemize}
			\item if $\Lambda_0\leq \frac 12-\epsilon_0$ then $H_1\geq \frac 12 + {\delta\over 6}$.
			\item if $H_0\geq \frac 12+\gamma_0$ then $\Lambda_0\leq 0  .3$.
		\end{itemize}
		A second application of \prettyref{lmm:Hs-Ls-bound} guarantees that we can choose $c_3$ large enough such that if  
		$\sqrt{\delta}\cdot {\snr}\sqrt{\alpha\over d}\geq c_3$
		then on the high probability event $\calE_{\gamma_1,\epsilon_1}$ defined from \prettyref{lmm:Hs-Ls-bound}, with $\epsilon_1=0.2,\gamma_1={\delta\over 6}$, if $H_{s}\geq \frac 12 +\frac \delta6$ then for large enough $c_3,c_4$ we can ensure
		$$\Lambda_s\leq {\tau_1\over \sqrt \delta\cdot \snr\sqrt{\alpha/d}} \leq 0.3,$$ 
		where $\tau_1$ is an absolute constant, and if $\Lambda_s\leq 0.3$ then $H_{s+1}\geq \frac 12 +\frac \delta6$.
		Combining the above displays we get that on the event
		\begin{align}
			\label{eq:Prob-E2}
			\calE=\calE_{\gamma_0,\epsilon_0}\cap \calE_{\gamma_1,\epsilon_1}
			,\quad \PP\qth{\calE}
			\geq 1-4(k^2+k)n^{-c/4}-16de^{-0.3n}
		\end{align} 
		we have for large enough $c_3,c_4$,
		\begin{align}
			\label{eq:km27}
			\Lambda_s\leq {\tau_1\over \sqrt \delta \cdot \snr\sqrt{\alpha/d}} \leq 0.3,\quad H_{s+1}\geq \frac 12 +\frac \delta6 \text{ for all } s\ge 1.
		\end{align}
		We will show that $\PP\qth{\left.z_i\neq \hat z_i^{(s+1)}\right|\calE}$ is small for each $i\in [n]$. This will imply that on the event $\calE$, with a large probability $\ell(\hat z^{(s+1)},z)=\frac 1n\sum_{i=1}^n\indc{z_i\neq \hat z_i^{(s+1)}}$ is also small. From this, using a Markov inequality we will conclude the result.

		Note that $\indc{z_i\neq\hat z_i^{(s+1)}}
		=\sum_{h\in [k]\atop h\neq z_i} \indc{\hat z_i^{(s+1)}=h}$. Fix a choice for $z_i$, say equal to $g\in [k]$.
		For any $h\in[k],h\neq g$,
		\begin{align}\label{eq:km21}
			&\indc{z_i=g,\hat z_i^{(s+1)}=h}
			\leq \indc{\|Y_i-\hat\theta_h^{(s)}\|^2
				\leq \|Y_i-\hat\theta_g^{(s)}\|^2, i\in T_g^*}
			\nonumber\\
			&=\indc{\|\theta_g+w_i-\hat\theta_h^{(s)}\|^2
				\leq \|\theta_g+w_i-\hat\theta_g^{(s)}\|^2}
			\\
			&=\indc{\|\theta_g-\hat\theta_h^{(s)}\|^2-\|\theta_g-\hat\theta_g^{(s)}\|^2\leq 2\langle w_i,\hat\theta_h^{(s)}-\hat\theta_g^{(s)}\rangle}.
		\end{align}
		Using $$\|\theta_g-\hat\theta_g^{(s)}\|\leq \Lambda_s\Delta
		\leq \Lambda_s\|\theta_g-\theta_h\|$$ for all $g\in [k]$, and the triangle inequality we have
		\begin{equation}\label{eq:m2}
			\begin{aligned}
				&\|\theta_g-\hat\theta_h^{(s)}\|^2\geq \pth{\|\theta_g-\theta_h\|-\|\theta_h-\hat\theta_h^{(s)}\|}^2\\
				&\geq (1-\Lambda_s)^2\|\theta_g-\theta_h\|^2,\\
				&\|\theta_g-\hat\theta_h^{(s)}\|^2-\|\theta_g-\hat\theta_g^{(s)}\|^2
				\\
				&\geq (1-\Lambda_s)^2\|\theta_g-\theta_h\|^2
				-\Lambda_s^2\|\theta_g-\theta_h\|^2
				\\
				&\geq (1-2\Lambda_s)\|\theta_g-\theta_h\|^2.
			\end{aligned}
		\end{equation}
		In view of \eqref{eq:km21} the last display implies
		\begin{equation}
			\label{eq:km10}
			\indc{z_i=g,\hat z_i^{(s+1)}=h}
			\leq \indc{(1-2\Lambda_s)\|\theta_g-\theta_h\|^2\leq 2\langle w_i,\hat\theta^{(s)}_h-\hat\theta^{(s)}_g \rangle}.
		\end{equation}
		Note that \eqref{eq:km27} implies for some absolute constant $\tau_1>0$ 
		$$ \Lambda_s\leq {\tau_1\over \sqrt \delta\cdot \snr\sqrt{\alpha\over d}},\quad 1-2\Lambda_s\geq {1-{2\tau_1\over \sqrt \delta\cdot \snr\sqrt{\alpha\over d}}}.$$
		Define
		\begin{equation*}
			\beta=1-{2\tau_1\over \sqrt \delta\cdot \snr\sqrt{\alpha\over d}}-\beta_1, 
			\quad \beta_1=\pth{\tau_1\over \sqrt \delta\cdot \snr\sqrt{\alpha\over d}}^{1/2}.
		\end{equation*} 
		Define $\Delta_h=\hat\theta_h^{(s)}-\theta_h,h\in[k]$. Then we use \eqref{eq:km10} to get
		\begin{equation}
			\label{eq:km28}
			\begin{aligned}
			&\indc{z_i=g,\hat z_i^{(s+1)}=h}
			\leq \indc{(\beta+\beta_1)\|\theta_g-\theta_h\|^2\leq 2\langle w_i,\hat\theta^{(s)}_h-\hat\theta^{(s)}_g \rangle}\\
			&\leq \indc{\beta\|\theta_g-\theta_h\|^2\leq 2\langle w_i,\theta_h-\theta_g \rangle}
			+\indc{\beta_1\|\theta_g-\theta_h\|^2\leq 2\langle w_i,\Delta_h-\Delta_g \rangle}.
		\end{aligned}
		\end{equation}
		For the first term on the right most side we get
		\begin{equation}
			\begin{aligned}
			\label{eq:km32}
		&	\EE\qth{\indc{\beta\|\theta_g-\theta_h\|^2\leq 2\langle w_i,\theta_h-\theta_g \rangle}}
			\leq \exp{\sth{-{\beta^2\|\theta_h-\theta_g\|^2\over 8\sigma^2}}}
			\\
			&\leq \exp\sth{-{\beta^2\Delta^2\over 8\sigma^2}}
			\leq \exp\sth{-{\beta^2(\snr)^2\over 2}}.
		\end{aligned}
		\end{equation}
		To bound the second term, using $\|\Delta_h\|,\|\Delta_g\|\leq \Lambda_s\Delta$ and $\Lambda_s\leq {\beta_1^2}$ we have
		\begin{equation}
			\begin{aligned}
			\label{eq:km29}
			&\indc{\beta_1\|\theta_g-\theta_h\|^2\leq 2\langle w_i,\Delta_h-\Delta_g \rangle}
			\leq 
			\indc{\beta_1\Delta^2\leq 2\|w_i\|\|\Delta_h-\Delta_g\|}
			\\
			&\leq
			\indc{\beta_1\Delta^2\leq 4\|w_i\|\Lambda_s\Delta}
			\leq
			\indc{\|w_i\|\geq {\Delta\over 4\beta_1}}.
		\end{aligned}
		\end{equation}
		Next we use the following tail bound for a $\subG(\sigma^2)$ random vector, which follows from \cite[Theorem 2.1]{hsu2012tail} by choosing $A$ to be an identity matrix and $\mu=0$.
		\begin{lemma}\label{lmm:subg-norm-tail}
			Let $w\in \reals^d$ be a $\subG(\sigma^2)$ random variable. Then 
			$\PP\qth{\|w\|_2>\sigma\pth{\sqrt d+\sqrt{2t}}}\leq e^{-t}$.
		\end{lemma}
		Choose $t={\pth{{\Delta\over 4\beta_1\sigma}-\sqrt d}^2\over 2}$. 
		This imples for a large enough value of $\sqrt \delta \cdot \snr\sqrt{\alpha\over d}$ we have
		$${\Delta\over \beta_1\sigma}=
		\sqrt{d\over \alpha}\frac 1{\sqrt{\tau_1}}\pth{\sqrt \delta\cdot \snr\sqrt{\alpha\over d}}^{1/2}
		\geq 8\sqrt d,$$
		and as a consequence we get
		$$t\geq {\Delta^2\over 128\beta_1^2\sigma^2}
		\geq {\Delta^2\over 8\sigma^2}
		=\frac {(\snr)^2}2.$$
		In view of this we continue \eqref{eq:km29} to get
		\begin{align*}
			&\PP\qth{\beta_1\|\theta_g-\theta_h\|^2\leq 2\langle w_i,\Delta_h-\Delta_g \rangle}
			\leq \PP\qth{\|w_i\|\geq {\Delta\over 4\beta_1}}
			\\
			&=\PP\qth{\|w_i\|\geq \sigma(\sqrt d+\sqrt {2t})}
			\leq \exp\sth{-\frac {(\snr)^2}2}.
		\end{align*}
		Combining the above with \eqref{eq:km28}, \eqref{eq:km32} we get
		\begin{align*}
			&\PP\qth{\left.z_i\neq \hat z_i^{(s+1)}\right|\calE}\leq
			k^2 \max_{g,h\in [k]\atop g\neq h}\PP\qth{\left.z_i=g, \hat z_i^{(s+1)}=h \right|\calE}
			\\
			&\leq 2k^2\exp\sth{-{\beta^2\over 2}{(\snr)^2}}.
		\end{align*}
		This implies
		\begin{align*}
			&\EE\qth{\ell(\hat z^{(s)},z)|\calE}
			=\frac 1n\sum_{i=1}^n\PP\qth{\left.z_i\neq z_i^{(s+1)}\right|\calE}
			\\
			&\leq 2k^2\exp\sth{-{\beta^2\over 2}{(\snr)^2}}.
		\end{align*}
		Combining the above with \eqref{eq:Prob-E2} we get
		\begin{align*}
			\EE\qth{\ell(\hat z^{(s)},z)}&\leq 4(k^2+k)n^{-c/4}+16de^{-0.3n}+ 2k^2e^{-{\beta^2\over 2}{(\snr)^2}}.
		\end{align*}
		This implies for any $t>0$ we get using Markov's inequality
		\begin{equation*}
			\begin{aligned}
			&\PP\qth{\ell(\hat z^{(s)},z)\geq t}
			\leq \frac 1t \EE\qth{\ell(\hat z^{(s)},z)}
			\\
			&\leq \frac{1}t \pth{4(k^2+k)n^{-c/4}+16de^{-0.3n}+ 2k^2e^{-{\beta^2\over 2}{(\snr)^2}}}
		\end{aligned}
		\end{equation*}
		If ${\beta^2(\snr)^2}\leq 8{\log n}$, we choose $t=e^{-\pth{\beta^2-{4\over \snr}}{(\snr)^2\over 2}}$. Then we get $\frac 1t\leq n^4$, which implies
		\begin{align*}
			&~\PP\qth{\ell(\hat z^{(s)},z)\geq \exp\sth{-\pth{\beta^2-{4\over \snr}}{(\snr)^2\over 2}}}
			\\
			&\leq 4(k^2+k)n^{-(c/4-2)}+16dn^4e^{-0.3n}+ 2k^2\exp\sth{-{\snr\over 2}}.
		\end{align*}
		Otherwise, if ${\beta^2(\snr)^2} > 8{\log n}$, then choosing $t=\frac 1n$ and noting that $\ell(\hat z^{(s)},z)$ takes values in $\sth{0,\frac 1n,\frac 2n,\dots, 1}$ we get
		\begin{equation*}
			\begin{aligned}
			&\PP\qth{\ell(\hat z^{(s)},z) > 0}
			=\PP\qth{\ell(\hat z^{(s)},z)\geq \frac 1n}
			\nonumber\\
			&\leq 4(k^2+k)n^{-(c/4-1)}+16dne^{-0.3n}+ 2k^2n e^{-{{\beta^2\over 2}(\snr)^2}}
			\nonumber\\
			&\leq 4(k^2+k)n^{-(c/4-1)}+16dne^{-0.3n}+ {2k^2\over n^3}.
		\end{aligned}
		\end{equation*}
		This finishes our proof of \prettyref{thm:main}.
	\end{proof} 
	
	We conclude this section by providing a proof of \prettyref{lmm:Hs-Ls-bound}.
	
	\begin{proof}[Proof of \prettyref{lmm:Hs-Ls-bound}]
		
		For any $g\neq h\in[k]\times [k]$, using the arguments of \eqref{eq:km21} and \eqref{eq:m2} we get
		\begin{align}\label{eq:llone4}
			\indc{z_i=g,\hat z_i^{(s+1)}=h}\leq \indc{\|\theta_g-\hat\theta_h^{(s)}\|^2-\|\theta_g-\hat\theta_g^{(s)}\|^2\leq 2\langle w_i,\hat\theta_h^{(s)}-\hat\theta_g^{(s)}\rangle}.
		\end{align}
		and
		\begin{align}
			\label{eq:llone1}
			&\|\theta_g-\hat\theta_h^{(s)}\|^2-\|\theta_g-\hat\theta_g^{(s)}\|^2
			\nonumber\\
			&\geq (1-2\Lambda_s)\|\theta_g-\theta_h\|^2\geq 2\epsilon_0\|\theta_g-\theta_h\|^2.
		\end{align}
		Denote by $\Delta_h=\hat\theta_h^{(s)}-\theta_h$ for $h\in[k]$. In view of the last inequality, continuing \eqref{eq:llone4} we get
		\begin{equation*}
			\begin{aligned}
			&\indc{z_i=g,\hat z_i^{(s+1)}=h}
			\\
			&\leq \indc{\epsilon_0\|\theta_g-\theta_h\|^2\leq \langle w_i,\theta_h-\theta_g+\Delta_h-\Delta_g\rangle}
			\\
			&\leq \indc{{\epsilon_0\over 2}\|\theta_g-\theta_h\|^2\leq \langle w_i,\theta_h-\theta_g\rangle}
			+\indc{{\epsilon_0\over 2}\|\theta_g-\theta_h\|^2\leq \langle w_i,\Delta_h-\Delta_g\rangle}.
			\\
			&\leq \indc{{\epsilon_0\over 2}\|\theta_g-\theta_h\|^2\leq \langle w_i,\theta_h-\theta_g\rangle}+\frac 4{\epsilon_0^2\Delta^4}(w_i^\sfT(\Delta_h-\Delta_g))^2,
		\end{aligned}
		\end{equation*}
		where the last inequality follows as
		\begin{equation*}
		\begin{aligned}
			&\indc{{\epsilon_0\over 2}\|\theta_g-\theta_h\|^2\leq \langle w_i,\Delta_h-\Delta_g\rangle}
		\\
		&\leq {4\pth{w_i^\sfT(\Delta_h-\Delta_g)}^2\over \epsilon_0^2\|\theta_g-\theta_h\|^4}
		\leq \frac 4{\epsilon_0^2\Delta^4}(w_i^\sfT(\Delta_h-\Delta_g))^2.
		\end{aligned}
		\end{equation*}
		Summing $\indc{z_i=g,\hat z_i^{(s+1)}=h}$ over $\sth{i\in T_g^*}$
		\begin{align}
			\label{eq:llone2}
			n_{gh}^{(s+1)}&\leq \sum_{i\in T_g^*} \indc{{\epsilon_0\over 2}\|\theta_g-\theta_h\|^2\leq \langle w_i,\theta_h-\theta_g\rangle}
			\nonumber\\
			&+ \frac 4{\epsilon_0^2\Delta^4}\sum_{i\in T_g^*} (w_i^\sfT(\Delta_h-\Delta_g))^2
		\end{align}
		In view of \prettyref{lmm:Chernoff-binom}, as $n_g^* \geq {n\alpha}$ we get on the event $\econ_{\epsilon_0}$ 
		\begin{align}
			\label{eq:llone3}
			\sum_{i\in T_g^*} \indc{{\epsilon_0\over 2}\|\theta_g-\theta_h\|^2\leq \langle w_i, \theta_h-\theta_g \rangle}
			&\leq \frac{10n_g^*\sigma^2}{\epsilon_0^2\Delta^2}.
		\end{align}
		Next we note that as $\|\Delta_g-\Delta_h\|^2\leq 4\Lambda_s^2\Delta^2$, we have
		\begin{equation*}
			\begin{aligned}
				&\sum_{i\in T_g^*}\pth{w_i^\sfT(\Delta_h-\Delta_g)}^2
			\nonumber\\
			&=\sum_{i\in T_g^*} (\Delta_h-\Delta_g) \pth{\sum_{i\in T_g^*}w_iw_i^\sfT}(\Delta_h-\Delta_g)
			\nonumber\\
			&\leq \lambda_{\max}\pth{\sum_{i\in T_g^*}w_iw_i^\sfT} \|\Delta_g-\Delta_h\|^2
			\nonumber\\
			&\leq 4\Lambda_s^2\Delta^2\lambda_{\max}\pth{\sum_{i\in T_g^*}w_iw_i^\sfT}. 
			\end{aligned}
		\end{equation*}
		This implies on the event $\eEigen$ as in \prettyref{lmm:eigen_bound}
		\begin{align*}
			\sum_{i\in T_g^*}\pth{w_i^\sfT(\Delta_h-\Delta_g)}^2
			&\leq
			24\Lambda_s^2\Delta^2\pth{n_g^*+ d}.
		\end{align*}
		In view of \eqref{eq:llone2} and \eqref{eq:llone3} we get that on the set $\econ_{\epsilon_0}$ for all $g\neq h\in [k]$
		\begin{equation}
			\label{eq:llone5}
			n_{gh}^{(s+1)}
			\leq 
			\frac{10n_g^*\sigma^2}{\epsilon_0^2\Delta^2}
			+{96\Lambda_s^2\sigma^2\over \epsilon_0^2\Delta^2}\pth{{n_g^*}+  d}.
		\end{equation}
		Using the last display and noting that $k\leq \min\sth{\frac 1{\alpha}, n },n_g^*\geq n\alpha$ and $\Lambda_s<\frac 12$ we get
		\begin{align}\label{eq:km23}
			{\sum_{h\in[k]\atop h\neq g}n_{gh}^{(s+1)}\over n_g^*}
			&\leq
			\frac{10\sigma^2}{\epsilon_0^2\Delta^2\alpha}
			+{96\Lambda_s^2\sigma^2\over \epsilon_0^2\Delta^2\alpha}\pth{1+ {dk\over n}}
			\nonumber\\
			&\leq {17\over 2\epsilon_0^2 (\snr\sqrt{\alpha\over 1+{dk/n}})^2},
		\end{align}
		where we note that the last expression is defined as $\tau$ in the lemma statement. Define the rightmost term in the above display as $\tau$. Using $n_g^{*}=\sum_{h\in [k]}n_{gh}^{(s+1)}$ for all $s\geq 0$, we get 
		\begin{align}\label{eq:km24}
			{n_{gg}^{(s+1)}\over n_g^*}\geq 1-\tau.
		\end{align}	 
		Next we switch $g,h$ in \eqref{eq:llone5} and sum over $h\in [k], h\neq g$. We get
		\begin{align*}
			\sum_{h\in[k]\atop h\neq g}n_{hg}^{(s+1)}
			\leq 
			\frac{10n\sigma^2}{\epsilon_0^2\Delta^2}
			+{96\Lambda_s^2\sigma^2\over \epsilon_0^2\Delta^2}\pth{n+  dk}
			\leq n\alpha\tau.
		\end{align*}
		Using the above and noticing that in addition to the points in $\cup_{h\in[k]}\sth{Y_i:i\in T_{h}^*\cap T_g^{(s+1)}}$, $\sth{Y_i:i\in T_g^{(s+1)}}$ can at most have $n\alpha(1-\delta)$ many extra points, accounting for the outliers, we get
		\begin{align*}
			&{n_{gg}^{(s+1)}\over n_g^{(s+1)}}\geq {n_{gg}^{(s+1)}\over n_{gg}^{(s+1)}+n\alpha\tau+n\alpha(1-\delta)}
			\nonumber\\
			&\geq \frac 1{1+{n_g^*(1-\delta+\tau)\over n_{gg}^{(s+1)}}}
			\geq \frac 1{1+{1-\delta+\tau\over 1-\tau}}
			=\frac 12 + {\delta-2\tau\over 2(2-\delta)}.
		\end{align*}
		Combining the last display with \eqref{eq:km24} we get
		\begin{align*}
			&H_{s+1}=\min_{g\in [k]}\sth{\min\sth{{n_{gg}^{(s+1)}\over n_g^*},{n_{gg}^{(s+1)}\over n_{g}^{(s+1)}}}}
			\nonumber\\
			&\geq \frac 12+\min\sth{{\delta-2\tau\over 2(2-\delta)},\frac 12-\tau},
		\end{align*}
		as required.
		
		We present below the proof of \prettyref{lmm:Hs-Ls-bound}(ii). Recall the definitions in \eqref{eq:cluster-def} and let $\hat \theta_g^{(s)}$ be the location estimate in iteration $s$, as defined before, and $\tilde \theta_g^{(s)}$ be the coordinatewise median of data points corresponding to $T_{g}^{(s)}\cap T_{g}^*$
		\begin{align*}
			\begin{gathered}
				\hat \theta_g^{(s)}
			=\med(\sth{Y_i:i\in T_g^{(s)}}),\\
			\tilde \theta_g^{(s)}
			=\med(\sth{Y_i:i\in T_{g}^{(s)}\cap T_{g}^*}).
			\end{gathered}
		\end{align*} 
		For $j=1,\dots,d$, let $V_j$ denote the set comprising of the $j^{\rm th}$ coordinates of the vectors in $T_{g}^{(s)}\cap T_{g}^*$.
		We will first show the deterministic result: given any $g\in[k]$ 
		\begin{align}
			\label{eq:km16}
			\|\tilde \theta_{g}^{(s)}-\hat \theta_g^{(s)}\|_2^2\leq \sum_{j=1}^d\pth{{ V_j^{( \floor{\gamma_0 {n_{gg}^{(s)}}})} } - V_j^{(\ceil{{1\over 1+2\gamma_0}n_{gg}^{(s)}})}}^2.
		\end{align}
		To prove this, note that
		\begin{align*}
			n_{gg}^{(s)}\geq H_s n_g^{(s)}
			\geq \pth{\frac 12+\gamma_0}n_g^{(s)}.
		\end{align*}
		For each $j=1,\dots,d,$ let $U_j$ denote the set comprising of the $j^{\rm th}$ coordinates of the vectors in $T_{g}^{(s)}$. Then $V_j$ is a subset of $U_j$. Using $\gamma_0>{10\over n\alpha}$ and $n_g^{(s)}\geq n_{gg}^{(s)}\geq \frac 12 n_g^*\geq \frac {n\alpha}2
		$ we have
		\begin{align}\label{eq:km15}
			\begin{gathered}
				|U_j| - |V_j| = n_g^{(s)}-n_{gg}^{(s)} < \pth{\frac12 -\gamma_0} n_g^{(s)}
				\leq {n_g^{(s)}\over 2}-5,\\
				n_{gg}^{(s)}>\pth{\frac 12+\gamma_0}n_g^{(s)}
				\geq {n_g^{(s)}\over 2}+{\gamma_0}n_{g}^{(s)}
				\geq {n_g^{(s)}\over 2}+5.
			\end{gathered}
		\end{align}
		In view of this we use \prettyref{lmm:ord-stat_after_addn} to control the difference between the order statistics of $U_j$ and $V_j$. We have
		\begin{align}\label{eq:km18}
			&V_j^{(\ceil{ {n_g^{(s)}/2}})}
			\stepa{\leq} U_j^{(\ceil{ {n_g^{(s)}/2}})}
			= \med(U_j)
			\nonumber\\
			&\stepb{\leq} V_j^{(\ceil{n_g^{(s)}/2}-n_g^{(s)}+n_{gg}^{(s)})}
			\stepc{\leq}
			V_j^{(\floor{n_{gg}^{(s)}-n_g^{(s)}/2})}
			\stepd{\leq} V_j^{(\floor{\gamma_0 n_{g}^{(s)}})}
		\end{align}
		where (a) and (b) follow by using \prettyref{lmm:ord-stat_after_addn} with $$X=V_j,\tilde X=U_j,m=n_{gg}^{(s)},\ell=n_{g}^{(s)}-n_{gg}^{(s)}, t=\ceil{n_g^{(s)}/2},$$ (c) follows by using $\ceil{x+y}\geq \ceil{x}+\floor{y}$ for any  $x,y>0$, and (d) follows by using \eqref{eq:km15}.
		Using ${n_{gg}^{(s)}\over 2}> \pth{\frac 12+\gamma_0}{n_g^{(s)}\over 2}\geq \gamma_0 n_{g}^{(s)}$ from \eqref{eq:km15} we also get
		\begin{align*}
			V_j^{(\ceil{ {n_g^{(s)}/2}})}
			\leq V_j^{(\ceil{n_{gg}^{(s)}/2})}=\med(V_j)
			\leq  V_j^{(\floor{\gamma_0 n_{g}^{(s)}})}.
		\end{align*}
		Combining the above with \eqref{eq:km18} we get
		\begin{align*}
			\abs{\med(V_j)-\med(U_j)}
			&\leq \abs{V_j^{(\floor{\gamma_0 n_g^{(s)}})}-V_j^{(\ceil{ {n_g^{(s)}/2}})}}.
		\end{align*}
		Hence we have
		\begin{align}\label{eq:km19}
			&\|\hat \theta_g^{(s)}-\tilde \theta_{g}^{(s)}\|_2^2
			=\sum_{j=1}^d\abs{\med(V_j)-\med(U_j)}^2
			\nonumber\\
			&
			\leq \sum_{j=1}^d\pth{V_j^{( \floor{\gamma_0 {n_g^{(s)}}})} - V_j^{(\ceil{ n_g^{(s)}/2})}}^2.
		\end{align}
		Using $n_{g}^{(s)}\leq {2\over 1+2\gamma_0} n_{gg}^{(s)}$ from \eqref{eq:km15} and $n_{g}^{(s)}\geq n_{gg}^{(s)}$ we get $$V_j^{(\ceil{{1\over 1+2\gamma_0}n_{gg}^{(s)}})}\leq V_j^{(\ceil{ n_g^{(s)}/2})}\leq V_j^{( \floor{\gamma_0 {n_g^{(s)}}})}
		\leq V_j^{( \floor{\gamma_0 {n_{gg}^{(s)}}})}.$$
		In view of the previous display this establishes \eqref{eq:km16}.
		
		Next, we provide a high probability bound on the right hand term in the above equation using \prettyref{lmm:order-concentration}. As $n_{gg}^{(s)}\geq \frac 12n_g^*\geq {n\alpha\over 2}$, using \prettyref{lmm:order-concentration} twice with $p_0=\frac {2\gamma_0}{1+2\gamma_0}$ and $p_0=\gamma_0$ respectively, we get that for each $j\in[d]$, with probability $1-4e^{-0.3 n}$
		\begin{align*}
				&V_j^{(\ceil{{1\over 1+2\gamma_0}n_{gg}^{(s)}})}
				\geq -\sigma\sqrt{\frac {2(1+2\gamma_0)}{\gamma_0}\pth{\frac 2\alpha+1}}\nonumber\\
				&
				\geq -\sigma\sqrt{\frac 4{\gamma_0}\pth{\frac 2\alpha+1}}
				\geq -\sigma\sqrt{\frac {12}{\gamma_0\alpha}},
			\end{align*}
		and
		\begin{align*}
				V_j^{( \floor{\gamma_0 {n_{gg}^{(s)}}})}\leq \sigma\sqrt{\frac 4{\gamma_0}\pth{\frac 2\alpha+1}}
				\leq \sigma\sqrt{\frac {12}{\gamma_0\alpha}}.
		\end{align*}
		In view of the above, using \eqref{eq:km19} and a union bound over $j\in[d]$, we get that with probability at least $1-4de^{-0.3 n}$
		\begin{align}\label{eq:km9}
			\|\tilde \theta_{g}^{(s)}-\hat \theta_g^{(s)}\|_2^2
			\leq {48\sigma^2 d\over \gamma_0\alpha} .
		\end{align}
		Next we note that
		\begin{align*}
			\tilde \theta_g^{(s)}-\theta_g
			=\med\sth{w_i:i\in T_{gg}^{(s)}}.
		\end{align*} 
		In view of the above using \prettyref{lmm:order-concentration} with $p_0=\frac 12$ and $n_{gg}^{(s)}\geq {n\alpha\over 2}$ we get that each coordinate of $\tilde \theta_g^{(s)}-\theta_g$ lies in $\pth{-\sigma\sqrt{\frac {24}\alpha },\sigma\sqrt{\frac {24}\alpha}}$ with probability at least $1-2e^{-0.3n}$. Repeating the argument in all the coordinates we have with probability at least $1-2de^{-0.3n}$
		\begin{align*}
			\|\tilde \theta_g^{(s)}-\theta_g\|_2^2
			\leq {96\sigma^2 d\over\alpha}
			\leq {48\sigma^2 d\over\gamma_0\alpha}.
		\end{align*}
		Combining with \eqref{eq:km9} and using $\gamma_0<\frac 12$ we get with probability at least $1-6de^{-0.3n}$
		\begin{align*}
			&\|\hat \theta_g^{(s)}-\theta_g\|_2^2
			\leq 2(\|\hat \theta_g^{(s)}-\hat \theta_g^{(s)}\|_2^2+\|\hat \theta_g^{(s)}-\theta_g\|_2^2)\nonumber\\
			&\leq 2\pth{{48\sigma^2 d\over \gamma_0\alpha}+{48\sigma^2 d\over \alpha}}
			\leq {192\sigma^2 d\over \gamma_0\alpha}.
		\end{align*}
		Dividing both sides in the above display with $\Delta^2$ and using a union bound over $g\in[k]$ we get the desired result.
	\end{proof}

	\section{Proof of \prettyref{thm:centroid}}
	
	\label{app:proof-centroid}
	For this section, let us define the ratio of total outliers with respect to individual clusters as
	\begin{align}
		\beta_g^\out = {n^{\out}\over n_g^*},\quad g=1,2,\dots,k.
	\end{align}
	Note that in order to identify all the clusters accurately we require $\max_g\beta_g^\out<1$; otherwise, given any algorithm, an adversary can always create a constellation involving only the outliers that will force the algorithm to detect it as an individual cluster.
	Using the definition
	\begin{align*}
		A_s=\frac 1n \sum_{i=1}^n\indc{\hat z_i^{(s)}\neq z_i}
	\end{align*} 
	we get for any $g\in [k]$ and $s\geq 3$
	\begin{align*}
		\max\sth{\sum_{h\in[k]\atop h\ne g}n_{hg}^{(s)},\sum_{h\in[k]\atop h\ne g}n_{gh}^{(s)}}
		\leq \sum_{i=1}^n\indc{\hat z_i^{(s)}\neq z_i}
		\eqdef nA_s.
	\end{align*}
	Note that using $n\leq {n_g^*\over \alpha}$ and $n_{gg}^{(s)}\geq n_{g}^*-nA_s\geq n_g^*(1-A_s/\alpha)$, using the last display we can write
	\begin{align*}
		&{n_{gg}^{(s)}\over n_g^{(s)}}\geq {n_{gg}^{(s)}\over n_{gg}^{(s)}+\sum_{h\in[k]\atop h\ne g}n_{hg}^{(s)}+n^{\out}}\nonumber\\
		&
		\geq \frac 1{1+{n_g^*(\frac 1\alpha A_s+\beta_g^{\out})\over n_{gg}^{(s+1)}}}
		\geq \frac 1{1+{\beta_g^{\out}+ A_s/\alpha\over 1-A_s/\alpha}}
		=\frac 12 + {1-\beta_g^{\out}-\frac 2\alpha A_s\over 2(1+{\beta_g^{\out}})},
	\end{align*}
	which implies
	\begin{align}\label{eq:km30}
		&\ceil{ {n_g^{(s)}/2}}-n_g^{(s)}+n_{gg}^{(s)}\geq n_{gg}^{(s)}-{n_g^{(s)}\over 2}\nonumber\\
		&
		\geq \pth{1-\beta_g^{\out}-\frac 2\alpha A_s\over 1+\beta_g^{\out}}{n_g^{(s)}\over 2}
		= \pth{\frac 12-{\beta_g^{\out}+A_s/\alpha \over 1+\beta_g^{\out}}}{n_g^{(s)}}
	\end{align}
	As $n_g^*\geq n\alpha$, the above implies for all $g\in [k]$
	\begin{gather}
		n_g^{(s)}\geq 	n_{gg}^{(s)}=n_g^*-\sum_{h\in[k]\atop h\ne g}n_{gh}^{(s)}
		\geq n_g^*-nA_s \geq n_g^*\pth{1-\frac 1\alpha A_s}
		\label{eq:km31}
	\end{gather}
	Fix $g\in [k]$. Let $U_j$ denote the set of real numbers consisting of the $j^{\text{th}}$ coordinates of $\sth{Y_i:i\in T_g^{(s)}}$, $V_j$ denote the set of real numbers consisting of the $j^{\text{th}}$ coordinates of $\sth{Y_i:i\in T_g^*\cap T_g^{(s)}}$, and $W_j$ denote the set of real numbers consisting of the $j^{\text{th}}$ coordinates of $\sth{Y_i:i\in T_g^*}$. Then we get
	\begin{align*}
		U_j^{(\ceil{ {n_g^{(s)}/2}})}
		&\stepa{\leq} V_j^{(\ceil{n_g^{(s)}/2}-n_g^{(s)}+n_{gg}^{(s)})}
		\nonumber\\
		&\stepb{\leq }
		V_j^{(\floor{\pth{1-\beta_g^{\out}-\frac 2\alpha A_s\over 1+\beta_g^{\out}}{n_g^{(s)}\over 2}})}
		\nonumber\\
		&\stepc{\leq} W_j^{(\floor{\pth{1-\beta_g^{\out}-\frac 2\alpha A_s\over 1+\beta_g^{\out}}{n_g^{(s)}\over 2}})}\nonumber\\
		&
		\stepd{\leq} W_j^{(\floor{{n_g^{*}}\pth{\frac 12-{\beta_g^{\out}+A_s/\alpha \over 1+\beta_g^{\out}}}\pth{1-\frac 1\alpha A_s}})}\nonumber\\
		&
		\leq W_j^{(\floor{{n_g^{*}}\pth{\frac 12-{\beta_g^{\out}\over 1+\beta_g^{\out}}-\frac 3{2\alpha}A_s}})},
	\end{align*}
	where
	\begin{itemize}
		\item (a) follows using \prettyref{lmm:ord-stat_after_addn} with
		$
		X=V_j,\ \tilde X=U_j,\ \ell=n_g^{(s)}-n_{gg}^{(s)},\ t=\ceil{n_g^{(s)}\over 2},
		$
		\item (b) follows using \eqref{eq:km30} and $\ceil{x}-y\geq \floor{x-y}$ for $x,y\geq 0$
		
		\item (c) follows using \prettyref{lmm:ord-stat_after_addn} with 
		$X=V_j,\ \tilde X=W_j,\ \ell=n_g^*-n_{gg}^{(s)},\ t=\floor{\pth{\frac 12-{\beta_g^{\out}+ A_s/\alpha\over 1+\beta_g^{\out}}}{n_g^{(s)}}},$
		
		\item (d) follows using \eqref{eq:km31}.
	\end{itemize}
	Similarly, we also have
	\begin{align*}
		U_j^{(\ceil{ {n_g^{(s)}/2}})}
		&\stepa{\geq} V_j^{(\ceil{n_g^{(s)}/2})}
		\nonumber\\
		&\stepb{\geq} 
		V_j^{(\ceil{{n_g^*\over 2}\cdot \frac 1{1-{2A_s/\alpha}}})}
		\stepc{\geq} 
		V_j^{(\ceil{{n_g^*}(\frac 12+\frac 2\alpha A_s)})}
		\nonumber\\
		&\stepd{\geq} W_j^{(\ceil{{n_g^*}(\frac 12+\frac 2\alpha A_s)}+\sum_{h=[k]\atop h\neq k}n_{gh})}
		\nonumber\\
		&\stepe{\geq} W_j^{(\ceil{{n_g^*}(\frac 12+\frac 3\alpha A_s)})}.
	\end{align*}
	\begin{itemize}
		\item (a) follows using \prettyref{lmm:ord-stat_after_addn} with
		$
		X=V_j,\ \tilde X=U_j,\ \ell=n_g^{(s)}-n_{gg}^{(s)},\ t=\ceil{n_g^{(s)}\over 2},
		$
		
		\item (b) follows using \eqref{eq:km31} and (c) follows using $(1-x)^{-1}\leq 1+2x$ for $x=\frac 2\alpha A_s\leq \frac 12$.
		
		\item (c) used \prettyref{lmm:ord-stat_after_addn} with 
		$X=V_j,\ \tilde X=W_j,\ \ell=\sum_{h=[k]\atop h\neq k}n_{gh},\ t=\ceil{{n_g^*\over 2}(1+\frac 4\alpha A_s)}+\sum_{h=[k]\atop h\neq k}n_{gh},$

		\item (d) used that for any $x>0$ and positive integer $y$, $\ceil{x}+y = \ceil{x+y}$, and $\sum_{h=[k]\atop h\neq k}n_{gh}\leq nA_s$.
	\end{itemize}
	Then \prettyref{thm:main} implies that given any $\tau>0$, there exists $C:=C(\tau)$ such that whenever $\sqrt \delta\cdot \snr\sqrt{\alpha/d}\geq C(\tau)$, we have $\frac 6\alpha A_s\leq e^{-{\Delta^2\over (8+\tau)\sigma^2}}$. Using this we combine the last two displays to get
	\begin{align*}
		&W_j^{(\ceil{{n_g^*\pth{\frac 12+e^{-{\Delta^2\over (8+\tau)\sigma^2}}}}})}
		\leq U_j^{(\floor{{n_g^{(s)}/2}})}\nonumber\\
		&\leq W_j^{(\floor{{n_g^*\pth{\frac 12-{\beta_g^{\out}\over 1+\beta_g^{\out}}-e^{-{\Delta^2\over (8+\tau)\sigma^2}}}}})}.
	\end{align*}
	Then using the result \prettyref{lmm:order-concentration-smallx} on the concentration of the order statistics we conclude our proof.

	
	\section{Technical results}\label{app:technical}

	\begin{lemma}\label{lmm:order-concentration}
		Let $\sth{Z_1,\dots,Z_n}$ be a set of independent real-valued $\subG(\sigma^2)$ random variables. Fix $p_0\in(0,\frac 12]$. Then there is an event $\eord_{p_0}$ with $\PP\qth{\eord_{p_0}}\geq 1-2e^{-0.3 n}$ on which for all $p\in [p_0,\frac 12]$ and all sets $V\subseteq \sth{Z_1,\dots,Z_n}$ with $|V|\geq \max\sth{{n\alpha},\frac 2{p_0}}$
		\begin{align*}
			-\sigma\sqrt{\frac 4{p_0}\pth{\frac 1\alpha+1}}
			&\leq V^{(\ceil{(1-p)|V|})}\leq V^{(\floor{p|V|})}\nonumber\\
			&\leq \sigma\sqrt{\frac 4{p_0}\pth{\frac 1\alpha+1}},
		\end{align*}
		where for a set of $n$ real numbers $\{v_1,\dots,v_n\}$ its order statistics are given by $v^{(1)}\geq \dots\geq v^{(n)}$.
	\end{lemma}
	
	\begin{proof}
		First we analyze for a fixed set $V$ and then take a union bound over all possible choices of $V$, which at most $2^n$ many. For an ease of notation let $|V|=m$. Fix $t=\sigma\sqrt{\frac 4{p_0}\pth{\frac 1\alpha+1}}$ and let $B_Z=\indc{Z\geq t}$ for all $Z\in V$.  This implies
		\begin{align}
			&\PP\qth{V^{(\floor{pm})}\geq t}
			\leq \PP\qth{V^{(\floor{p_0m})}\geq t}
			\nonumber\\
			&= \PP\qth{\sum_{Z\in V} B_Z\geq \floor{mp_0}}
			\leq \PP\qth{\sum_{Z\in V} B_Z\geq mp_0-1}
			.\label{eq:km17}
		\end{align} 
		Let $q=\sup_{z\in\subG(\sigma^2)}\PP\qth{z\geq t}$, where the supremum is taken over all one-dimensional sub-Gaussian random variables. Using tail bounds on sub-Gaussian random variables \cite[Section 2.1.2]{wainwright2019high} and $e^{-x}\leq \frac 1x$ for $x>0$ we get
		\begin{align*}
			q {\leq} e^{-{t^2\over 2\sigma^2}}
			\leq e^{-\frac 2{p_0}\pth{\frac 1\alpha+1}}
			{\leq} {p_0\over 2\pth{\frac 1\alpha +1}}.
		\end{align*}
		Let $S\sim \Binom(m,q)$. Then using stochastic dominance of $S$ over $\sum_{Z\in V} B_Z$ we get that 
		$$
		\PP\qth{\sum_{Z\in V} B_Z\geq mp_0-1}
		\leq \PP\qth{S\geq mp_0-1}.
		$$
		Next we use the Chernoff inequality given in \eqref{eq:chernoff} with $q\leq {p_0\over 2}\leq p_0-\frac 1m$ and $a=p_0-\frac 1m$ we get
		\begin{align}\label{eq:km6}
			&\PP\qth{V^{(\floor{p_0m})}\geq t}
			\leq \exp\pth{-mh_{q}\pth{p_0-\frac 1m}}
			\nonumber\\
			&\leq \exp\Biggl(-m\Biggl\{(p_0-\frac 1m)\log{p_0-\frac 1m\over q}
			\nonumber\\
			&+(1-p_0+\frac 1m)\log{1-p_0+\frac 1m\over 1-q}\Biggr\}\Biggr).
		\end{align}
		Note the following:
		\begin{itemize}
			\item To analyze the first summand in the above exponent we use $q\leq {p_0\over 2\pth{\frac 1\alpha +1}}$, $p_0-\frac 1m\geq {p_0\over 2}$, and the fact $x\log x\geq -\frac 12$ for all $|x|<1$ to get
			\begin{align*}
			&(p_0-\frac 1m)\log{p_0-\frac 1m\over q}
			\nonumber\\
			&\geq \pth{p_0-\frac 1m}\pth{\log\pth{p_0-\frac 1m}+ \frac 2{p_0}\pth{\frac 1\alpha+1}}
			\nonumber\\
			&\geq -\frac 12 + \pth{\frac 1\alpha +1}
			\geq \frac 1\alpha+\frac 12.
			\end{align*}
			\item For the second term using $x\log x\geq -\frac 12$ for $|x|<1$ we have
			$$
			\begin{aligned}
				&(1-p_0+\frac 1m)\log{1-p_0+\frac 1m\over 1-q}
			\\
			&\geq (1-p_0+\frac 1m)\log(1-p_0+\frac 1m)
			\geq -\frac 12.
			\end{aligned}
			$$
		\end{itemize}
		Combining these with \eqref{eq:km6} we get for all $p\geq p_0$
		\begin{align*}
			\PP\qth{V^{(\floor{pm})}\geq \sigma\sqrt{\frac 4{p_0}\pth{\frac 1\alpha+1}}}
			\leq \exp\pth{-\frac m\alpha}\leq e^{-n}.
		\end{align*}
		As the total possible choices of $V$ is at most $2^n$, we use union bound to conclude that with probability at least $1-2^ne^{-n}\geq 1-e^{-0.3n}$ we get for all $V\subseteq [n]$ with $|V|\geq {n\alpha}$
		\begin{align*}
			\PP\qth{V^{(\floor{p|V|})}\leq \sigma\sqrt{\frac 4{p_0}\pth{\frac 1\alpha+1}}}.
		\end{align*}
		The first inequality in the lemma statement can be obtained in a similar fashion by considering the random variables $-Z_1,\dots,-Z_m$.
	\end{proof}

	\begin{lemma}
		\label{lmm:Chernoff-binom}
		Fix $\epsilon_0 >0$ and let $n\alpha\geq c\log n$. Then there is an event $\econ_{\epsilon_0}$ with $\PP\qth{\econ_{\epsilon_0}}\geq 1-kn^{-c/4}$ on which for any $h\in [k]$
		$$\sum_{i\in T_g^*}\indc{\epsilon_0^2\|\theta_g-\theta_h\|^2\leq \langle w_i,\theta_h-\theta_g\rangle}
		\leq \frac {5n_g^*}{2\epsilon_0^4(\Delta/\sigma)^2}, \quad g\in [k], g\neq h.$$
	\end{lemma}
	
	\begin{proof}
		Define $x=\sup_{w\in \subG(\sigma^2)}\PP\qth{\epsilon_0^2\|\theta_g-\theta_h\|^2\leq \langle w,\theta_h-\theta_g\rangle}$. Note that $x$ satisfies
		\begin{align*}
			x\leq e^{-{\epsilon_0^4\|\theta_g-\theta_h\|^2\over 2\sigma^2}}.
		\end{align*}
		Let $S_{g}$ be a random variable distributed as $\Binom(n_g^*,x)$. Then using stochastic dominance we get 
		$$\PP\qth{\sum_{i\in T_g^*}\indc{\epsilon_0^2\|\theta_g-\theta_h\|^2\leq \langle w_i,\theta_h-\theta_g\rangle}\geq y} \leq \PP\qth{S_{g}\geq y},\ y\geq 0,$$
		and hence it is enough to analyze the tail probabilities of $S_{g}$. As $\log (1/x)\geq {\epsilon_0^4(\Delta/\sigma)^2\over 2}$, it suffices to show that
		\begin{align}
			\label{eq:km8}
			\PP\qth{S_{g} \geq \frac {5n_g^*}{4\log(1/x)}}\leq n^{-c/4} \text{ for any } g\in [k].
		\end{align}
		Then using a union bound over $g\in [k]$ we get the desired result. We continue to analyze \eqref{eq:km8} using the Chernoff's inequality for the Binomial random variable $S_g$ from \eqref{eq:chernoff} with $a={\frac 5{4\log(1/x)}}$ and $m=n_g^*$. We get $x=e^{-4/(5a)}$, and using $\log(1/x)<\frac 1{x}$ we get $y>x$. Using $y\log y\geq -0.5$ for $y\in (0,1)$ we get
		\begin{align*}
			&~\PP\qth{S_g \geq \frac {5n_g^*}{4\log(1/x)}}\\
			&\leq \exp\pth{-mh_x(a)}
			\nonumber\\
			&\leq \exp\pth{-m\pth{a\log{a\over x}+(1-a)\log{1-a\over 1-x}}}
			\nonumber\\
			&\leq \exp\pth{-m\sth{a\log{a\over e^{-{5/ (4a)}}}+(1-a)\log(1-a)}}
			\nonumber\\
			&=  \exp\pth{-m\sth{a\log a+(1-a)\log(1-a)+\frac 54}}
			\leq e^{-n_g^*/4}.
		\end{align*}
		As $n_g^*\geq n\alpha\geq c\log n$, we get \eqref{eq:km8}. 
	\end{proof}

	\begin{lemma}
		\label{lmm:eigen_bound}
		For any symmetric matrix $A$ let $\lambda_{\max}(A)$ denote its maximum eigen value. Let $n\alpha\geq c\log n$. Then there exists an event $\eEigen$ with $\PP\qth{\eEigen}\geq 1-kn^{-c/2}$ on which
		$\lambda_{\max}\pth{\sum_{i\in T_g^*} w_iw_i^{\sfT}}\leq  6\sigma^2(n_g^*+d)$ for all $g\in [k]$.
	\end{lemma}
	
	\begin{proof}
		From \cite[Lemma A.2]{lu2016statistical} we note that
		given any $g\in[k]$, the set $\sth{w_i: i \text{ is such that } i\in T_g^*}$ of $\subG(\sigma^2)$ random vectors satisfy
		\begin{align*}
			\PP\qth{\lambda_{\max}\pth{\sum_{i\in T_g^*}w_iw_i^\sfT}\leq 6\sigma^2(n_g^*+d)}\geq 1-e^{-0.5 n_g^*}.
		\end{align*}
		Using a union bound for $g\in [k]$ and the assumption $n_g^*\geq n\alpha\geq c\log n$ we infer that 
		\begin{align*}
			&\PP\qth{\lambda_{\max}\pth{\sum_{i\in T_g^*}w_iw_i^\sfT}\leq 6\sigma^2(n_g^*+d) \text{ for all } g\in [k]}
			\\ &
			\geq 1-ke^{-0.5 n_g^*}
			\geq 1-kn^{-c/2}.
		\end{align*}
	\end{proof}
	
	\begin{lemma}[Order statistics after addition]\label{lmm:ord-stat_after_addn}
		Let $X=\sth{X_1,\dots ,X_m}$ be any set of $m$ real numbers. Suppose that we add $\ell<m$ many new entries to the set and call the new set $\tilde X$. Then we have 
		$$
		X^{(t)}\leq \tilde X^{(t)}\leq X^{(t-\ell)}, \quad t\in \naturals
		$$
		with the notation that negative order statistics are defined to be infinity and for a set of size $m$ all $a(\geq m+1)$-th order statistics are defined to be $-\infty$.
	\end{lemma}
	
	\begin{proof}
		We first prove the statement when $\ell=1$ and $t$ is any positive integer, i.e., when we add only one entry. Suppose that the new entry is $x$. Then we have the following for all $1<t\leq m$:
		\begin{itemize}
			\item $\tilde X^{(t)}=X^{(t)}$ if $x\leq X^{(t)}$
			\item $\tilde X^{(t)}=\min\{X^{(t-1)},x\}$ if $x> X^{(t)}$.
		\end{itemize}
		This proves our case.

		Next, we use induction to prove the statement for $\ell\geq 2$. Assume that the lemma statement holds for $\ell-1$, i.e., if $V$ is the updated set after adding $\ell-1$ many entries, then 
		$$
		X^{(t-1)}\leq V^{(t-1)}\leq X^{(t-\ell)},\quad
		X^{(t)}\leq V^{(t)}\leq X^{(t-\ell+1)}.
		$$
		Using the base case $\ell=1$ we get
		\begin{align*}
			V^{(t)}\leq \tilde X^{(t)}\leq V^{(t-1)}.
		\end{align*}
		Combining the last two displays we get the result.
	\end{proof}

	\begin{lemma}\label{lmm:order-concentration-smallx}
		Let $Z=\sth{Z_1,\dots,Z_m}$ be a set of independent real random variables, $q\in(0,\frac 12]$ and $x_0$ be real numbers that satisfy
		$$\PP\qth{Z_j\geq x_0}\leq q,\quad \PP\qth{Z_j\leq -x_0}\leq q,\quad j=1,\dots,m.$$
		Then given any $p_0\in (q+\frac 1m,\frac 12]$
		\begin{align*}
			\PP\qth{-x_0< Z^{(\ceil{(1-p_0)m})}\leq Z^{(\floor{p_0m})}< x_0}\geq 1-2e^{-2m(p_0-q-\frac 1m)^2}.
		\end{align*}
	\end{lemma}

	\begin{proof}
		Let $B_i=\indc{Z_i\geq x_0}$ for all $i=1,\dots,m$.  This implies
		\begin{align*}
			&\PP\qth{Z^{(\floor{p_0m})}\geq x_0}
			\\
			&= \PP\qth{\sum_{i=1}^m B_i\geq \floor{mp_0}}
			\leq \PP\qth{\sum_{i=1}^m B_i\geq mp_0-1}
		\end{align*} 
		Note that $B_i$'s are \iid Bernoulli random variables with success probability $q_i=\PP\qth{B_i\geq x_0}\leq q$.	Let $S\sim \Binom(m,q)$. Then using stochastic dominance, we get that 
		$$
		\PP\qth{\sum_{i=1}^m B_i\geq mp_0-1}
		\leq \PP\qth{S\geq mp_0-1}.
		$$
		Note the Chernoff inequality for a $\Binom(m,q)$ random variable \cite[Section 2.2]{boucheron2013concentration}:
		\begin{align}
			\label{eq:chernoff}
			\begin{gathered}
				\PP\qth{S\geq ma}
				\leq \exp\pth{-mh_q(a)};\quad q<a<1,\\
				h_q(a)=a\log{a\over q}+(1-a)\log{1-a\over 1-q}.
			\end{gathered}
		\end{align}
		Further, using the Pinsker inequality between the Kulback-Leibler divergence and the total variation distance \cite[Theorem 4.19]{boucheron2013concentration}
		we get $h_q(a)\geq 2(q-a)^2$. Combining this with the last display with $a=p_0-\frac 1m$ we get 
		$$
		\PP\qth{Z^{(\floor{p_0m})}\geq x_0}
		\leq e^{-2m(p_0-q-\frac 1m)^2}.
		$$
		The lower tail bound can be obtained in a similar fashion by considering the random variables $-Z_1,\dots,-Z_m$.
	\end{proof}

	%



	\bibliographystyle{apalike}
	
	\bibliography{References}

\begin{thebibliography}{}

\bibitem[Abbasi and Younis, 2007]{ABBASI20072826}
Abbasi, A.~A. and Younis, M. (2007).
\newblock A survey on clustering algorithms for wireless sensor networks.
\newblock {\em Computer Communications}, 30(14):2826--2841.
\newblock Network Coverage and Routing Schemes for Wireless Sensor Networks.

\bibitem[Abbe et~al., 2022]{abbe2022}
Abbe, E., Fan, J., and Wang, K. (2022).
\newblock An $\ell_p$ theory of pca and spectral clustering.
\newblock {\em The Annals of Statistics}, 50(4):2359--2385.

\bibitem[Ajala~Funmilola et~al., 2012]{ajala2012fuzzy}
Ajala~Funmilola, A., Oke, O., Adedeji, T., Alade, O., and Adewusi, E. (2012).
\newblock Fuzzy kc-means clustering algorithm for medical image segmentation.
\newblock {\em Journal of information Engineering and Applications, ISSN},
  22245782:2225--0506.

\bibitem[Alpaydin and Alimoglu,
  1998]{misc_pen_based_recognition_of_handwritten_digits_81}
Alpaydin, E. and Alimoglu, F. (1998).
\newblock {Pen-Based Recognition of Handwritten Digits}.
\newblock UCI Machine Learning Repository.
\newblock {DOI}: https://doi.org/10.24432/C5MG6K.

\bibitem[Anandkumar et~al., 2012]{anandkumar2012method}
Anandkumar, A., Hsu, D., and Kakade, S.~M. (2012).
\newblock A method of moments for mixture models and hidden markov models.
\newblock In {\em Conference on Learning Theory}, pages 33--1. JMLR Workshop
  and Conference Proceedings.

\bibitem[Anegg et~al., 2020]{anegg2020technique}
Anegg, G., Angelidakis, H., Kurpisz, A., and Zenklusen, R. (2020).
\newblock A technique for obtaining true approximations for k-center with
  covering constraints.
\newblock In {\em Integer Programming and Combinatorial Optimization: 21st
  International Conference, IPCO 2020, London, UK, June 8--10, 2020,
  Proceedings}, pages 52--65. Springer.

\bibitem[Appert and Catoni, 2021]{appert2021new}
Appert, G. and Catoni, O. (2021).
\newblock New bounds for $ k $-means and information $ k $-means.
\newblock {\em arXiv preprint arXiv:2101.05728}.

\bibitem[Arthur and Vassilvitskii, 2007]{arthur2007k}
Arthur, D. and Vassilvitskii, S. (2007).
\newblock K-means++ the advantages of careful seeding.
\newblock In {\em Proceedings of the eighteenth annual ACM-SIAM symposium on
  Discrete algorithms}, pages 1027--1035.

\bibitem[Awasthi and Balcan, 2014]{awasthi2014center}
Awasthi, P. and Balcan, M.-F. (2014).
\newblock Center based clustering: A foundational perspective.

\bibitem[Awasthi et~al., 2012]{awasthi2012center}
Awasthi, P., Blum, A., and Sheffet, O. (2012).
\newblock Center-based clustering under perturbation stability.
\newblock {\em Information Processing Letters}, 112(1-2):49--54.

\bibitem[Bajaj, 1986]{bajaj1986proving}
Bajaj, C. (1986).
\newblock Proving geometric algorithm non-solvability: An application of
  factoring polynomials.
\newblock {\em Journal of Symbolic Computation}, 2(1):99--102.

\bibitem[Bakshi et~al., 2020]{bakshi2020outlier}
Bakshi, A., Diakonikolas, I., Hopkins, S.~B., Kane, D., Karmalkar, S., and
  Kothari, P.~K. (2020).
\newblock Outlier-robust clustering of gaussians and other non-spherical
  mixtures.
\newblock In {\em 2020 IEEE 61st Annual Symposium on Foundations of Computer
  Science (FOCS)}, pages 149--159. IEEE.

\bibitem[Balcan et~al., 2017]{balcan2017differentially}
Balcan, M.-F., Dick, T., Liang, Y., Mou, W., and Zhang, H. (2017).
\newblock Differentially private clustering in high-dimensional euclidean
  spaces.
\newblock In {\em International Conference on Machine Learning}, pages
  322--331. PMLR.

\bibitem[Belkin and Sinha, 2010]{belkin2009learning}
Belkin, M. and Sinha, K. (2010).
\newblock Toward learning gaussian mixtures with arbitrary separation.
\newblock In {\em COLT}, pages 407--419.

\bibitem[Bickel, 1964]{bickel1964some}
Bickel, P.~J. (1964).
\newblock On some alternative estimates for shift in the p-variate one sample
  problem.
\newblock {\em The Annals of Mathematical Statistics}, pages 1079--1090.

\bibitem[Bojchevski et~al., 2017]{bojchevski2017robust}
Bojchevski, A., Matkovic, Y., and G{\"u}nnemann, S. (2017).
\newblock Robust spectral clustering for noisy data: Modeling sparse
  corruptions improves latent embeddings.
\newblock In {\em Proceedings of the 23rd ACM SIGKDD international conference
  on knowledge discovery and data mining}, pages 737--746.

\bibitem[Boucheron et~al., 2013]{boucheron2013concentration}
Boucheron, S., Lugosi, G., and Massart, P. (2013).
\newblock {\em Concentration inequalities: A nonasymptotic theory of
  independence}.
\newblock Oxford university press.

\bibitem[Bradley et~al., 1996]{bradley1996clustering}
Bradley, P., Mangasarian, O., and Street, W. (1996).
\newblock Clustering via concave minimization.
\newblock {\em Advances in neural information processing systems}, 9.

\bibitem[Brunet-Saumard et~al., 2022]{brunet2022k}
Brunet-Saumard, C., Genetay, E., and Saumard, A. (2022).
\newblock K-bmom: A robust lloyd-type clustering algorithm based on bootstrap
  median-of-means.
\newblock {\em Computational Statistics \& Data Analysis}, 167:107370.

\bibitem[Chakraborty and Chaudhuri, 1996]{chakraborty1996transformation}
Chakraborty, B. and Chaudhuri, P. (1996).
\newblock On a transformation and re-transformation technique for constructing
  an affine equivariant multivariate median.
\newblock {\em Proceedings of the American mathematical society},
  124(8):2539--2547.

\bibitem[Chakraborty and Chaudhuri, 1999]{chakraborty1999note}
Chakraborty, B. and Chaudhuri, P. (1999).
\newblock A note on the robustness of multivariate medians.
\newblock {\em Statistics \& Probability Letters}, 45(3):269--276.

\bibitem[Charikar et~al., 2001]{charikar2001algorithms}
Charikar, M., Khuller, S., Mount, D.~M., and Narasimhan, G. (2001).
\newblock Algorithms for facility location problems with outliers.
\newblock In {\em SODA}, volume~1, pages 642--651. Citeseer.

\bibitem[Chaudhuri, 1996]{chaudhuri1996geometric}
Chaudhuri, P. (1996).
\newblock On a geometric notion of quantiles for multivariate data.
\newblock {\em Journal of the American statistical association},
  91(434):862--872.

\bibitem[Chen et~al., 2018]{chen2018robust}
Chen, M., Gao, C., and Ren, Z. (2018).
\newblock Robust covariance and scatter matrix estimation under huber’s
  contamination model.
\newblock {\em The Annals of Statistics}, 46(5):1932--1960.

\bibitem[Chen and Zhang, 2024]{chen2021optimal}
Chen, X. and Zhang, A.~Y. (2024).
\newblock Achieving optimal clustering in gaussian mixture models with
  anisotropic covariance structures.
\newblock {\em Advances in Neural Information Processing Systems},
  37:113698--113741.

\bibitem[Chew and Dyrsdale~III, 1985]{chew1985voronoi}
Chew, L.~P. and Dyrsdale~III, R.~L. (1985).
\newblock Voronoi diagrams based on convex distance functions.
\newblock In {\em Proceedings of the first annual symposium on Computational
  geometry}, pages 235--244.

\bibitem[Cohen et~al., 2016]{cohen2016geometric}
Cohen, M.~B., Lee, Y.~T., Miller, G., Pachocki, J., and Sidford, A. (2016).
\newblock Geometric median in nearly linear time.
\newblock In {\em Proceedings of the forty-eighth annual ACM symposium on
  Theory of Computing}, pages 9--21.

\bibitem[Cuesta-Albertos et~al., 1997]{cuesta1997trimmed}
Cuesta-Albertos, J.~A., Gordaliza, A., and Matr{\'a}n, C. (1997).
\newblock Trimmed $ k $-means: an attempt to robustify quantizers.
\newblock {\em The Annals of Statistics}, 25(2):553--576.

\bibitem[Dasgupta and Schulman, 2007]{dasgupta2007probabilistic}
Dasgupta, S. and Schulman, L.~J. (2007).
\newblock A probabilistic analysis of em for mixtures of separated, spherical
  gaussians.
\newblock {\em Journal of Machine Learning Research}, 8:203--226.

\bibitem[Dave, 1991]{dave1991characterization}
Dave, R.~N. (1991).
\newblock Characterization and detection of noise in clustering.
\newblock {\em Pattern Recognition Letters}, 12(11):657--664.

\bibitem[Dave, 1993]{dave1993robust}
Dave, R.~N. (1993).
\newblock Robust fuzzy clustering algorithms.
\newblock In {\em [Proceedings 1993] Second IEEE International Conference on
  Fuzzy Systems}, pages 1281--1286. IEEE.

\bibitem[Dav{\'e} and Krishnapuram, 1997]{dave1997robust}
Dav{\'e}, R.~N. and Krishnapuram, R. (1997).
\newblock Robust clustering methods: a unified view.
\newblock {\em IEEE Transactions on fuzzy systems}, 5(2):270--293.

\bibitem[Day, 1969]{day1969estimating}
Day, N.~E. (1969).
\newblock Estimating the components of a mixture of normal distributions.
\newblock {\em Biometrika}, 56(3):463--474.

\bibitem[{de Souza} and de~A.T.~{de Carvalho}, 2004]{DESOUZA2004353}
{de Souza}, R.~M. and de~A.T.~{de Carvalho}, F. (2004).
\newblock Clustering of interval data based on city–block distances.
\newblock {\em Pattern Recognition Letters}, 25(3):353--365.

\bibitem[Deshpande et~al., 2020]{deshpande2020robust}
Deshpande, A., Kacham, P., and Pratap, R. (2020).
\newblock Robust $ k $-means++.
\newblock In {\em Conference on Uncertainty in Artificial Intelligence}, pages
  799--808. PMLR.

\bibitem[Dhillon et~al., 2004]{dhillon2004kernel}
Dhillon, I.~S., Guan, Y., and Kulis, B. (2004).
\newblock Kernel k-means: spectral clustering and normalized cuts.
\newblock In {\em Proceedings of the tenth ACM SIGKDD international conference
  on Knowledge discovery and data mining}, pages 551--556.

\bibitem[Diakonikolas et~al., 2019]{diakonikolas2019robust}
Diakonikolas, I., Kamath, G., Kane, D., Li, J., Moitra, A., and Stewart, A.
  (2019).
\newblock Robust estimators in high-dimensions without the computational
  intractability.
\newblock {\em SIAM Journal on Computing}, 48(2):742--864.

\bibitem[Ding et~al., 2010]{ding2010clustering}
Ding, L., Shi, P., and Liu, B. (2010).
\newblock The clustering of internet, internet of things and social network.
\newblock In {\em 2010 Third International Symposium on Knowledge Acquisition
  and Modeling}, pages 417--420. IEEE.

\bibitem[Doss et~al., 2023]{doss2020optimal}
Doss, N., Wu, Y., Yang, P., and Zhou, H.~H. (2023).
\newblock {Optimal estimation of high-dimensional Gaussian location mixtures}.
\newblock {\em The Annals of Statistics}, 51(1):62 -- 95.

\bibitem[Gao et~al., 2018]{gao2018community}
Gao, C., Ma, Z., Zhang, A.~Y., and Zhou, H.~H. (2018).
\newblock {Community detection in degree-corrected block models}.
\newblock {\em The Annals of Statistics}, 46(5):2153 -- 2185.

\bibitem[Garc{\'\i}a-Escudero et~al., 2010]{garcia2010review}
Garc{\'\i}a-Escudero, L.~A., Gordaliza, A., Matr{\'a}n, C., and Mayo-Iscar, A.
  (2010).
\newblock A review of robust clustering methods.
\newblock {\em Advances in Data Analysis and Classification}, 4:89--109.

\bibitem[Gupta et~al., 2017]{gupta2017local}
Gupta, S., Kumar, R., Lu, K., Moseley, B., and Vassilvitskii, S. (2017).
\newblock Local search methods for k-means with outliers.
\newblock {\em Proceedings of the VLDB Endowment}, 10(7):757--768.

\bibitem[Hardin and Rocke, 2004]{hardin2004outlier}
Hardin, J. and Rocke, D.~M. (2004).
\newblock Outlier detection in the multiple cluster setting using the minimum
  covariance determinant estimator.
\newblock {\em Computational Statistics \& Data Analysis}, 44(4):625--638.

\bibitem[Hopkins and Li, 2019]{hopkins2019hard}
Hopkins, S.~B. and Li, J. (2019).
\newblock How hard is robust mean estimation?
\newblock In {\em Conference on learning theory}, pages 1649--1682. PMLR.

\bibitem[Hsu et~al., 2012]{hsu2012tail}
Hsu, D., Kakade, S., and Zhang, T. (2012).
\newblock A tail inequality for quadratic forms of subgaussian random vectors.
\newblock {\em Electronic communications in Probability}, 17:1--6.

\bibitem[Huber, 1965]{huber1965robust}
Huber, P.~J. (1965).
\newblock A robust version of the probability ratio test.
\newblock {\em The Annals of Mathematical Statistics}, pages 1753--1758.

\bibitem[Huber, 1992]{huber1992robust}
Huber, P.~J. (1992).
\newblock Robust estimation of a location parameter.
\newblock {\em Breakthroughs in statistics: Methodology and distribution},
  pages 492--518.

\bibitem[Jana et~al., 2025]{jana2024provable}
Jana, S., Fan, J., and Kulkarni, S. (2025).
\newblock A provable initialization and robust clustering method for general
  mixture models.
\newblock {\em IEEE Transactions on Information Theory}, 71(9):7176--7207.

\bibitem[Jayasumana et~al., 2015]{jayasumana2015kernel}
Jayasumana, S., Hartley, R., Salzmann, M., Li, H., and Harandi, M. (2015).
\newblock Kernel methods on riemannian manifolds with gaussian rbf kernels.
\newblock {\em IEEE transactions on pattern analysis and machine intelligence},
  37(12):2464--2477.

\bibitem[Jian, 2009]{jian2009data}
Jian, A.~K. (2009).
\newblock Data clustering: 50 years beyond k-means, pattern recognition
  letters.
\newblock {\em Corrected Proof}.

\bibitem[Jolion et~al., 1991]{jolion1991robust}
Jolion, J.-M., Meer, P., and Bataouche, S. (1991).
\newblock Robust clustering with applications in computer vision.
\newblock {\em IEEE transactions on pattern analysis and machine intelligence},
  13(8):791--802.

\bibitem[Kaufman and Rousseeuw, 2009]{kaufman2009finding}
Kaufman, L. and Rousseeuw, P.~J. (2009).
\newblock {\em Finding groups in data: an introduction to cluster analysis}.
\newblock John Wiley \& Sons.

\bibitem[Klein, 1989]{klein1989concrete}
Klein, R. (1989).
\newblock {\em Concrete and abstract Voronoi diagrams}, volume 400.
\newblock Springer Science \& Business Media.

\bibitem[Klochkov et~al., 2021]{klochkov2021robust}
Klochkov, Y., Kroshnin, A., and Zhivotovskiy, N. (2021).
\newblock Robust k-means clustering for distributions with two moments.
\newblock {\em The Annals of Statistics}, 49(4):2206--2230.

\bibitem[Krishnapuram and Keller, 1993]{krishnapuram1993possibilistic}
Krishnapuram, R. and Keller, J.~M. (1993).
\newblock A possibilistic approach to clustering.
\newblock {\em IEEE transactions on fuzzy systems}, 1(2):98--110.

\bibitem[Kumar et~al., 2004]{kumar2004simple}
Kumar, A., Sabharwal, Y., and Sen, S. (2004).
\newblock A simple linear time (1+/spl epsiv/)-approximation algorithm for
  k-means clustering in any dimensions.
\newblock In {\em 45th Annual IEEE Symposium on Foundations of Computer
  Science}, pages 454--462. IEEE.

\bibitem[Li et~al., 2007]{li2007noise}
Li, Z., Liu, J., Chen, S., and Tang, X. (2007).
\newblock Noise robust spectral clustering.
\newblock In {\em 2007 IEEE 11th International Conference on Computer Vision},
  pages 1--8. IEEE.

\bibitem[Liu and Moitra, 2023]{liu2023robustly}
Liu, A. and Moitra, A. (2023).
\newblock Robustly learning general mixtures of gaussians.
\newblock {\em Journal of the ACM}.

\bibitem[Lloyd, 1982]{lloyd1982least}
Lloyd, S. (1982).
\newblock Least squares quantization in pcm.
\newblock {\em IEEE transactions on information theory}, 28(2):129--137.

\bibitem[L{\"o}ffler et~al., 2021]{loffler2021optimality}
L{\"o}ffler, M., Zhang, A.~Y., and Zhou, H.~H. (2021).
\newblock Optimality of spectral clustering in the gaussian mixture model.
\newblock {\em The Annals of Statistics}, 49(5):2506--2530.

\bibitem[Lopuhaa, 1989]{lopuhaa1989relation}
Lopuhaa, H.~P. (1989).
\newblock On the relation between s-estimators and m-estimators of multivariate
  location and covariance.
\newblock {\em The Annals of Statistics}, pages 1662--1683.

\bibitem[Lu and Zhou, 2016]{lu2016statistical}
Lu, Y. and Zhou, H.~H. (2016).
\newblock Statistical and computational guarantees of lloyd's algorithm and its
  variants.
\newblock {\em arXiv preprint arXiv:1612.02099}.

\bibitem[Lugosi and Mendelson, 2021]{lugosi2021robust}
Lugosi, G. and Mendelson, S. (2021).
\newblock Robust multivariate mean estimation: the optimality of trimmed mean.

\bibitem[Lyu and Xia, 2025]{lyu2022optimal}
Lyu, Z. and Xia, D. (2025).
\newblock Optimal clustering by lloyd’s algorithm for low-rank mixture model.
\newblock {\em Journal of the Royal Statistical Society Series B: Statistical
  Methodology}, page qkaf041.

\bibitem[Makarychev et~al., 2019]{makarychev2019performance}
Makarychev, K., Makarychev, Y., and Razenshteyn, I. (2019).
\newblock Performance of johnson-lindenstrauss transform for k-means and
  k-medians clustering.
\newblock In {\em Proceedings of the 51st Annual ACM SIGACT Symposium on Theory
  of Computing}, pages 1027--1038.

\bibitem[Malkomes et~al., 2015]{MKCWM2015}
Malkomes, G., Kusner, M.~J., Chen, W., Weinberger, K.~Q., and Moseley, B.
  (2015).
\newblock Fast distributed k-center clustering with outliers on massive data.
\newblock In Cortes, C., Lawrence, N., Lee, D., Sugiyama, M., and Garnett, R.,
  editors, {\em Advances in Neural Information Processing Systems}, volume~28.
  Curran Associates, Inc.

\bibitem[Maravelias, 1999]{maravelias1999habitat}
Maravelias, C.~D. (1999).
\newblock Habitat selection and clustering of a pelagic fish: effects of
  topography and bathymetry on species dynamics.
\newblock {\em Canadian Journal of Fisheries and Aquatic Sciences},
  56(3):437--450.

\bibitem[Mishra et~al., 2007]{mishra2007clustering}
Mishra, N., Schreiber, R., Stanton, I., and Tarjan, R.~E. (2007).
\newblock Clustering social networks.
\newblock In {\em Algorithms and Models for the Web-Graph: 5th International
  Workshop, WAW 2007, San Diego, CA, USA, December 11-12, 2007. Proceedings 5},
  pages 56--67. Springer.

\bibitem[Moitra and Valiant, 2010]{moitra2010settling}
Moitra, A. and Valiant, G. (2010).
\newblock Settling the polynomial learnability of mixtures of gaussians.
\newblock In {\em 2010 IEEE 51st Annual Symposium on Foundations of Computer
  Science}, pages 93--102. IEEE.

\bibitem[Ng et~al., 2006]{ng2006medical}
Ng, H., Ong, S., Foong, K., Goh, P.-S., and Nowinski, W. (2006).
\newblock Medical image segmentation using k-means clustering and improved
  watershed algorithm.
\newblock In {\em 2006 IEEE southwest symposium on image analysis and
  interpretation}, pages 61--65. IEEE.

\bibitem[Olukanmi and Twala, 2017]{olukanmi2017k}
Olukanmi, P.~O. and Twala, B. (2017).
\newblock K-means-sharp: modified centroid update for outlier-robust k-means
  clustering.
\newblock In {\em 2017 Pattern Recognition Association of South Africa and
  Robotics and Mechatronics (PRASA-RobMech)}, pages 14--19. IEEE.

\bibitem[Pearson, 1894]{pearson1894contributions}
Pearson, K. (1894).
\newblock Contributions to the mathematical theory of evolution.
\newblock {\em Philosophical Transactions of the Royal Society of London. A},
  185:71--110.

\bibitem[Pigolotti et~al., 2007]{pigolotti2007species}
Pigolotti, S., L{\'o}pez, C., and Hern{\'a}ndez-Garc{\'\i}a, E. (2007).
\newblock Species clustering in competitive lotka-volterra models.
\newblock {\em Physical review letters}, 98(25):258101.

\bibitem[Rousseeuw and Kaufman, 1987]{rousseeuw1987clustering}
Rousseeuw, P. and Kaufman, P. (1987).
\newblock Clustering by means of medoids.
\newblock In {\em Proceedings of the statistical data analysis based on the L1
  norm conference, neuchatel, switzerland}, volume~31.

\bibitem[Sasikumar and Khara, 2012]{sasikumar2012k}
Sasikumar, P. and Khara, S. (2012).
\newblock K-means clustering in wireless sensor networks.
\newblock In {\em 2012 Fourth international conference on computational
  intelligence and communication networks}, pages 140--144. IEEE.

\bibitem[Slate, 1991]{misc_letter_recognition_59}
Slate, D. (1991).
\newblock {Letter Recognition}.
\newblock UCI Machine Learning Repository.
\newblock {DOI}: https://doi.org/10.24432/C5ZP40.

\bibitem[Srivastava et~al., 2023]{srivastava2023robust}
Srivastava, P.~R., Sarkar, P., and Hanasusanto, G.~A. (2023).
\newblock A robust spectral clustering algorithm for sub-gaussian mixture
  models with outliers.
\newblock {\em Operations Research}, 71(1):224--244.

\bibitem[Van~der Maaten and Hinton, 2008]{van2008visualizing}
Van~der Maaten, L. and Hinton, G. (2008).
\newblock Visualizing data using t-sne.
\newblock {\em Journal of machine learning research}, 9(11).

\bibitem[Vempala and Wang, 2004]{vempala2004spectral}
Vempala, S. and Wang, G. (2004).
\newblock A spectral algorithm for learning mixture models.
\newblock {\em Journal of Computer and System Sciences}, 68(4):841--860.

\bibitem[Wainwright, 2019]{wainwright2019high}
Wainwright, M.~J. (2019).
\newblock {\em High-dimensional statistics: A non-asymptotic viewpoint},
  volume~48.
\newblock Cambridge University Press.

\bibitem[Weiszfeld, 1937]{weiszfeld1937point}
Weiszfeld, E. (1937).
\newblock Sur le point pour lequel la somme des distances de n points
  donn{\'e}s est minimum.
\newblock {\em Tohoku Mathematical Journal, First Series}, 43:355--386.

\bibitem[Yin et~al., 2018]{yin2018byzantine}
Yin, D., Chen, Y., Kannan, R., and Bartlett, P. (2018).
\newblock Byzantine-robust distributed learning: Towards optimal statistical
  rates.
\newblock In {\em International Conference on Machine Learning}, pages
  5650--5659. PMLR.

\bibitem[Zhang et~al., 2022]{zhang2022practical}
Zhang, E., Li, H., Huang, Y., Hong, S., Zhao, L., and Ji, C. (2022).
\newblock Practical multi-party private collaborative k-means clustering.
\newblock {\em Neurocomputing}, 467:256--265.

\bibitem[Zhang and Rohe, 2018]{zhang2018understanding}
Zhang, Y. and Rohe, K. (2018).
\newblock Understanding regularized spectral clustering via graph conductance.
\newblock {\em Advances in Neural Information Processing Systems}, 31.

\bibitem[Zhang and Wang, 2023]{zhang2023upper}
Zhang, Z. and Wang, J. (2023).
\newblock Upper bound estimations of misclassification rate in the
  heteroscedastic clustering model with sub-gaussian noises.
\newblock {\em Stat}, 12(1):e505.

\end{thebibliography}

\end{document}